\documentclass[11pt,twoside]{atmp}

\usepackage{amsmath,amssymb}

\usepackage[all]{xy}
\usepackage{color}

\newtheorem{proposition}{Proposition}[part]
\newtheorem{definition}{Definition}[part]
\newtheorem{remark}{Remark}[part]
\newtheorem{postulate}{Postulate}[part]

\newcommand{\abs}[1]{\left|#1\right|}

\begin{document}

\title[Mechanics of Crystalline Solids]
{On the Mechanics of Crystalline Solids with a Continuous Distribution of Dislocations}

\arxurl{1212.5125}

\author[D.~Christodoulou and I.~Kaelin]{Demetrios Christodoulou and Ivo Kaelin}

\address{Department of Mathematics \\ ETH Zurich \\ R\"amistrasse 101 \\ 8092 Zurich \\ Switzerland}  
\addressemail{demetri@math.ethz.ch and ivo.kaelin@math.ethz.ch}

\begin{abstract}
We formulate the laws governing the dynamics of a crystalline solid in which a continuous distribution of dislocations is present. Our formulation is based on new differential geometric concepts, which in particular relate to Lie groups. We then consider the static case, which describes crystalline bodies in equilibrium in free space. The mathematical problem in this case is the free minimization of an energy integral, and the associated Euler-Lagrange equations constitute a nonlinear elliptic system of partial differential equations. We solve the problem in the simplest cases of interest.
\end{abstract}

\maketitle

\section*{Introduction and Summary}

Classical elasticity theory rests on the hypothesis that a solid may exist in a stress-free state in Euclidean space. The dynamics of the solid is then described by a time-dependent mapping taking the position of each material particle in this relaxed reference state to its actual position at a given time. Now the hypothesis of the existence of a relaxed state is violated when dislocations are present in the solid. These cause internal stresses even in the absence of external forces.

Solids in which a continuous distribution of dislocations is present have been treated in the linear approximation (\cite{K1}, \cite{K2}, \cite{N1}). However, the conceptual framework of this linear theory is inadequate for the formulation of an exact theory. The appropriate conceptual framework was introduced in \cite{C2} and an exact nonlinear theory was proposed. The basic concept introduced in \cite{C2} is that of the material manifold which captures those properties of a crystalline solid which are intrinsic to it, being independent of the way it is embedded in space. The dynamics of the elastic body is then described by a mapping from space-time into this material manifold.

While an exact theory of crystalline solids with a continuous distribution of dislocations was formulated in \cite{C2}, the theory was left undeveloped up to the present time. The aim of the present work is to develop the theory and derive results which may be brought into contact with experimental data. Our focus in the present paper is on the static case, where we have a crystalline solid with a uniform distribution of elementary dislocations in equilibrium in free space. As a basic example we study, in the continuum limit, a two-dimensional crystalline solid with a uniform distribution of edge dislocations, one of the two elementary types of dislocations. In this case the material manifold is the affine group of the real line, the hyperbolic plane. We then study, in the continuum limit, a three-dimensional crystalline solid with a uniform distribution of screw dislocations, the other elementary type of dislocation. In this case the material manifold is the Heisenberg group. 

The purpose of the present work is to introduce to the mathematics and theoretical physics community a field where beautiful differential geometric structures, in particular Lie groups, form the basis of a classical physical theory. Moreover, the laws of the theory form a nonlinear system of variational partial differential equations, which in the static case is elliptic and in the dynamical case is of hyperbolic type. Both cases constitute a worthy challenge for the geometric analyst.

\cutpage 

\setcounter{page}{2}

\noindent

This paper is organized as follows. In Part I, we give a general introduction to the theory of crystalline solids containing an arbitrary distribution of dislocations, based on the work of Christodoulou \cite{C1}, \cite{C2}. We introduce the basic concepts of material manifold and crystalline structure. In the present work, the fundamental notion is the canonical form, which is used to define the dislocation density. We illustrate the theory by giving two basic examples of solids with a uniform distribution of dislocations. In order to state the laws governing the dynamics, the thermodynamic space is introduced and the energy function defined. The dynamics is described by a mapping from space-time into the material manifold.
The equivalence relations for crystalline structures and for the mechanical properties of a solid are discussed to give the proper physical interpretation of the theory.
Finally, the Eulerian picture is given including the non-relativistic limit. In the Eulerian picture the material manifold is eliminated.

In Part II, we focus our attention on the static case. We derive the boundary value problem from an action principle and give the Legendre-Hadamard conditions for the energy function. We then consider the two examples of uniform distributions of elementary dislocations and motivate the choice of our model energy function. We separately discuss the special cases of uniform distributions of edge and screw dislocations in two and three dimensions, respectively. 
In concluding the second part, we derive the scaling properties of the theory.

Part III is devoted to the analysis of equilibrium configurations of a crystalline solid with a uniform distribution of elementary dislocations in two and three dimensions. We solve the problem in two dimensions and give the method for the solution of the anisotropic problem in the case of a uniform distribution of screw dislocations in three dimensions.

\newpage


\part{The General Setting} \label{Setting}
\setcounter{section}{0}


\section{The Material Manifold}

Let $\mathcal N$ be an oriented $n$-dimensional differentiable manifold, called the {\bf material manifold}. It describes a material together with those of its properties which are intrinsic to it being independent of the way in which it is to be embedded in physical space. A point $y \in \mathcal N$ represents a particle. We denote by $\mathcal X (\mathcal N)$ the $C^{\infty}$-vectorfields on the material manifold $\mathcal N$. Let
\begin{equation} \label{evalmap}
\begin{array}{rcl} 
\epsilon_y :  \mathcal X (\mathcal N) & \to & T_y\mathcal N \\
X & \mapsto & \epsilon_y(X)=X(y)
\end{array}
\end{equation}
be the evaluation map at a point $y \in \mathcal N$.

\begin{definition}
A {\bf crystalline structure} on $\mathcal N$ is a distinguished linear subspace $\mathcal V$ of  $\mathcal X (\mathcal N)$ such that the evaluation map $\epsilon_y$ restricted to $\mathcal V$ is an isomorphism for each $y \in \mathcal N$.
\end{definition}

\begin{remark}
The orientation of $\mathcal N$ induces an orientation in $\mathcal V$ such that $\epsilon_y$ is orientation preserving.
$\mathcal N$ admits a crystalline structure if and only if $\mathcal N$ is parallelizable, see \cite{C2}.
\end{remark}
We introduce a $1$-form $\nu$ on $\mathcal N$ with values in $\mathcal V$ defined by
\begin{equation} \label{defnu}
\nu_y(Y_y)=\epsilon_y^{-1}(Y_y) \quad : \forall Y_y \in T_y\mathcal N \, .
\end{equation}
In the case that $(\mathcal N, \mathcal V)$ is a Lie group with corresponding Lie algebra, $\nu$ is the \emph{Maurer-Cartan} form.

\begin{remark} \label{mfo}
The most fundamental object is the $1$-form $\nu$. We may in fact replace $\mathcal V$ by an abstract real vectorspace $W$ of the same dimension, $n$, as the manifold $\mathcal N$. Then we define a canonical form $\nu$ to be a $W$-valued $1$-form on $\mathcal N$ such that 
\begin{equation*}
\nu_y := \left. \nu \right|_{T_y\mathcal N} : T_y \mathcal N \to W
\end{equation*}
is an isomorphism for each $y \in \mathcal N$. Given an element $v \in W$, we may then define a vectorfield $Y_v$ on $\mathcal N$ by
\begin{equation*}
Y_v(y)=\nu_y^{-1}(v) \quad : \, \forall y \in \mathcal N
\end{equation*} 
We then define the crystalline structure $\mathcal V$ by
\begin{equation*}
\mathcal V = \left\{ X_v : v \in W \right\} \, .
\end{equation*}
Then the canonical form corresponds to the $1$-form $\nu$ defined in (\ref{defnu}).
\end{remark}

We say that the crystalline structure $\mathcal V$ on $\mathcal N$ is {\bf complete} if each $X \in \mathcal V$ is a complete vectorfield on $\mathcal N$. Then each element of $\mathcal V$ generates a $1$-parameter group of diffeomorphisms of $\mathcal N$, which represents (in the continuum limit) a group of translations of the crystal lattice with parameter proportional to the number of atoms traversed.

A complete crystalline structure $\mathcal V$ on the material manifold $\mathcal N$ defines an {\bf exponential map}
\begin{equation*}
\exp : \mathcal N \times \mathcal V \to \mathcal N
\end{equation*}
as follows. Let $\exp(y,X)$ be the point in $\mathcal N$ that is at parameter value $1$ from $y$ along the integral curve of $X$ initiating at $y$. For each $y \in \mathcal N$, let 
\begin{equation*} \begin{array}{rcl}
\exp_y : \mathcal V & \to & \mathcal N \\
X  & \mapsto & \exp_y(X)=\exp(y, X) \, .
\end{array}
\end{equation*}
We have
\begin{equation*}
\exp_y(0)=y \quad , \quad d\exp_y(0)=\epsilon_y \, .
\end{equation*}
Thus $d\exp_y(0)$ is an isomorphism for each $y \in \mathcal N$. By the implicit function theorem it follows that, for each $y \in \mathcal N$, there is a neighborhood $\mathcal U_y$ of the zero vector in $\mathcal V$ such that $\exp_y$ restricted to $\mathcal U_y$ is a diffeomorphism onto its image in $\mathcal N$.

Now choose a totally antisymmetric $n$-linear form $\omega$ on $\mathcal V$ which is positive when evaluated on a positive basis.
The $n$-form $\omega$ defines a {\bf volume form} $d\mu_{\omega}$, called {\bf mass form} on $\mathcal N$ by:
\begin{equation}\label{volumeform}
\begin{array}{rcl}
d\mu_{\omega}\left(Y_{1,y},\ldots, Y_{n,y}\right) & = & \omega\left(\epsilon_y^{-1}(Y_{1,y}),\ldots,\epsilon_y^{-1}(Y_{n,y})\right) \\
 &  &  : \, \forall y \in \mathcal N, \forall Y_{1,y}, \ldots, Y_{n,y} \in T_y\mathcal N \, .
\end{array}
\end{equation}
The volume assigned by $d\mu_{\omega}$ to a domain $\mathcal D \subset \mathcal N$,
\begin{equation*}
\int_{\mathcal D}d\mu_{\omega} \, ,
\end{equation*}
is the rest mass of $\mathcal D$. 

\begin{definition}
Given a crystalline structure $\mathcal V$ on $\mathcal N$ we can define a mapping 
\begin{equation*}
\Lambda:\mathcal N  \to  \mathcal L(\mathcal V \wedge \mathcal V, \mathcal V) 
\end{equation*}
by:
\begin{equation} \label{defdisden}
\Lambda(y)(X,Y)=\epsilon_y^{-1}\left([X,Y](y)\right) \in \mathcal V \quad , \quad \forall y \in \mathcal N, X,Y \in \mathcal V \, .
\end{equation}
We call $\Lambda$ {\bf dislocation density}.
\end{definition}

Suppose $X, Y \in \mathcal V$ are complete and generate $1$-parameter groups of diffeomorphisms $\Phi_t$, $\Psi_t$, $t \in \mathbb R$ from $\mathcal N$ to $\mathcal N$. For $y \in \mathcal N$
\begin{equation*}
(t,s) \mapsto \Psi_{-s}\left(\Phi_{-t}\left(\Psi_s\left(\Phi_t(y)\right)\right)\right)
\end{equation*}
coincides with $(t,s)\mapsto \Xi_{ts}$, where $\Xi_t$ is the $1$-parameter group of diffeomorphisms of $\mathcal N$ generated by $\Lambda(y)(X,Y)$.

The dislocation density is a concept that arises in the continuum limit of a distribution of elementary dislocations in a crystal lattice. An elementary dislocation has the property that, if we start at an atom in the crystal lattice and move according to one group of lattice transformations $k$ atoms in one direction, then according to a different group $l$ atoms in a second direction, according to the first $-k$ atoms and finally $-l$ atoms in the second directions, then we arrive at a different atom than we started from, but which, provided the circuit encloses a single elementary dislocation, is reached at in a single step corresponding to a lattice translation. This step is called {\bf Burger's vector}.

If $\Lambda$ is constant on $\mathcal N$ then, for all $X, Y \in \mathcal V$, there is a $Z \in \mathcal V$ such that
\begin{equation*} 
[X,Y]=Z \, ,
\end{equation*}
corresponding to a uniform distribution of elementary dislocations of the same kind, see below.
Thus $\mathcal V$ constitutes in this case a {\bf Lie algebra}, i.e.~a vectorspace $\mathcal V$ over $\mathbb R$ with a bracket
\begin{equation*}
[\, .\, ,\, .\, ] \, : \, \mathcal V \wedge \mathcal V \to \mathcal V 
\end{equation*}
satisfying the Jacobi identity.

By the fundamental theorems of Lie group theory, upon choosing an identity element $e \in \mathcal N$, the material manifold $\mathcal N$ can then be given the structure of a {\bf Lie group} such that $\mathcal V$ is the space of vectorfields on $\mathcal N$ which generate the right action of the group on itself. $\mathcal V$ is then the space of vectorfields on $\mathcal N$ which are left invariant, i.e.~invariant under left group multiplication.
The dual space $\mathcal V^*$ is then the space of left invariant $1$-forms on $\mathcal N$.
 
Let us consider $d\nu$, a $2$-form on $\mathcal N$ : for any pair of vectorfields $X, Y$ on $\mathcal N$ we have 
\begin{equation} \label{derofnu}
d\nu(X,Y)=X(\nu(Y))-Y(\nu(X))-\nu([X,Y]) 
\end{equation}
by the formula for the exterior derivative of a $1$-form. 
In particular, this holds for $X, Y \in \mathcal V$. Now for $X \in \mathcal V$ we have
\begin{equation*}
\nu(X)(y)=\nu_y(X_y)=\epsilon_y^{-1}(X_y)=X \quad , \quad \forall y \in \mathcal N, X \in \mathcal V \, ,
\end{equation*}
a constant $\mathcal V$-valued function on $\mathcal N$. Similarly with $X$ replaced by $Y$.
Therefore, from (\ref{derofnu}) we have
\begin{equation}\label{dernu}
d\nu (X,Y) = -\nu([X,Y])\quad , \, \forall X,Y \in \mathcal V \, .
\end{equation}
On the other hand,
\begin{equation} \label{Lambdaeqnu}
\Lambda(y)(X,Y)=\epsilon_y^{-1}([X,Y](y))=\nu_y([X,Y](y)) \quad , \,  \forall y \in \mathcal N, X,Y \in \mathcal V \, .
\end{equation}
Thus, comparing (\ref{dernu}) and (\ref{Lambdaeqnu}), we obtain
\begin{equation}\label{dnumlam}
\left(d\nu(X,Y)\right)(y)=-\nu_y\left([X,Y](y)\right)=-\Lambda(y)(X,Y) \, .
\end{equation}
Let us then define the $\mathcal V$-valued $2$-form $\lambda$ on $\mathcal N$ by:
\begin{equation*}
\lambda_y(Y_{1,y},Y_{2,y}) = \Lambda(y)\left(\epsilon_y^{-1}(Y_{1,y}),\epsilon_y^{-1}(Y_{2,y})\right) \quad  : \, \forall  Y_{1,y}, Y_{2,y} \in T_y\mathcal N \, ,
\end{equation*}
at any point $y \in \mathcal N$. We conclude from (\ref{dnumlam}) that 
\begin{equation*}
d\nu= -\lambda \, .
\end{equation*}
Let $\Gamma$ be a closed curve in $\mathcal N$ and let $\Sigma$ be any surface spanning $\Gamma$, i.e.~
$\partial \Sigma=\Gamma$. We finally conclude 
\begin{equation} \label{sumofBv}
 -\int_{\Gamma} \nu = \int_{\Sigma}\lambda  \, .
\end{equation}
The right-hand side of (\ref{sumofBv}) is the sum of all Burger's vectors enclosed by the curve $\Gamma$ (or threading $\Sigma$).

\subsection{Uniform Dislocation Distributions and Lie Groups}
At the atomic level, two kinds of elementary dislocations are found. They are called {\bf edge} and {\bf screw} dislocations.
In the following, we consider these two types of dislocations taking them as our model cases. For a uniform distribution of these two types of elementary dislocations, we determine the corresponding Lie groups, the affine group and the Heisenberg group, respectively. For a detailed description see \cite{IK}.

\subsubsection{Edge Dislocations and the Affine Group} \label{edag}
The most basic type of a dislocation in a $2$-dimensional crystal lattice is an edge dislocation. It
appears in a $2$-dimensional lattice in which an extra half-line of atoms has been inserted along the positive $1$st axis. A circuit of translations in the directions of the $1$st and $2$nd axis, alternately, which encloses the origin, ends at an atom which is reached in a single step by a translation in the direction of the $2$nd axis. On the other hand, circuits not enclosing the origin close. Mathematically, this phenomenon is represented by the commutation relation $[E_1,E_2]=E_2$, where $E_1, E_2$ are the vectorfields along the coordinate axis.

We want to show that a uniform distribution of edge dislocations in a $2$-dimensional lattice gives rise in the continuum limit to the {\bf affine group}. This group is characterized by transformations of the real line of the form
\begin{equation*} \begin{array}{rcl}
\mathbb R & \to & \mathbb R \\
x & \mapsto & e^{y^1}x+y^2 \, ,
\end{array}
\end{equation*}
where $(y^1, y^2) \in \mathbb R^2$ are two parameters. The subgroups of the affine group are $t\mapsto e^{y^1}t$ (multiplication) and $s\mapsto s + y^2$ (translation).

We have as the group manifold $\mathbb R^2$ equipped with the following multiplication
\begin{equation*}
(y^1,y^2)(\tilde y^1,\tilde y^2)=(y^1+\tilde y^1,y^2+e^{y^1}\tilde y^2) \, .
\end{equation*}
\begin{equation*}
X=\frac{\partial}{\partial y^1} \quad \textrm{and} \quad Y=e^{y^1}\frac{\partial}{\partial y^2}
\end{equation*}
generate the right action with
\begin{equation*}
[X,Y] = \frac{\partial}{\partial y^1}e^{y^1}\frac{\partial}{\partial y^2}-e^{y^1}\frac{\partial}{\partial y^2}\frac{\partial}{\partial y^1} =  e^{y^1}\frac{\partial}{\partial y^2}= Y \, .
\end{equation*}
The Lie algebra of the affine group is thus generated by the vectorfields $X, Y$, which satisfy the commutation relation
\begin{equation*}
[X,Y]=Y \, .
\end{equation*}
Therefore, the affine group is the material manifold $\mathcal N$ endowed with the crystalline structure.

If we take $\{E_1=X, E_2=Y\}$ as a basis of $\mathcal V$, we have the following dual basis $\{\omega^1, \omega^2\}$ for $\mathcal V^*$
\begin{equation*}
\omega^1 = dy^1  \quad , \quad  \omega^2=e^{-y^1}dy^2  \, .
\end{equation*}
The corresponding left invariant metric
\begin{equation} \label{hypmetafgr}
\stackrel{\circ}{n}=(\omega^1)^2+(\omega^2)^2=(dy^1)^2+e^{-2y^1}(dy^2)^2 
\end{equation}
on $\mathcal N$ makes $\mathcal N$ the {\bf hyperbolic plane}.

The dislocation density $\lambda$ is 
\begin{equation*}
\lambda(E_1,E_2)(y)=\Lambda(y)(E_1,E_2) = \epsilon_y^{-1}\left([E_1,E_2](y)\right) = \epsilon_y^{-1}\left(\epsilon_y(E_2)\right) = E_2 \, .
\end{equation*}
Since $\omega_1 \wedge \omega_2=e^{-y_1}\left(dy^1\wedge dy^2\right)$, it follows
\begin{equation*}
\lambda = e^{-y_1}\left(dy^1\wedge dy^2\right)E_2 \, ,
\end{equation*}
and therefore
\begin{equation*}
\int_{\Sigma}\lambda =\int_{\Sigma}e^{-y_1}\left(dy^1\wedge dy^2\right)E_2=A(\Sigma)E_2 \, , 
\end{equation*}
where $A(\Sigma)$ is the area of the surface $\Sigma$, a domain in $\mathcal N$. This makes sense since the sum of the Burger vectors associated to a domain in a uniform distribution of edge dislocations should be proportional to the area of the domain.

\subsubsection{Screw Dislocations and the Heisenberg Group} \label{sdhg}
The second kind of elementary dislocation is called a {\bf screw dislocation}. It appears in a $3$-dimensional lattice in the following way. A circuit of translations along the direction of the $1$st and $2$nd axis, alternately, which encloses the $3$rd, ends at an atom which is reached at a single step by a translation in the direction of the $3$rd axis, while circuits not enclosing the $3$rd axis close.
 Mathematically, this phenomenon is represented by the commutation relations $[E_1, E_2]=E_3$, $[E_1, E_3]=[E_2, E_3]=0$, where $E_1, E_2, E_3$ are the vectorfields along the coordinate axis.

We want to show that a uniform distribution of dislocations of the screw type give rise in the continuum limit to the {\bf Heisenberg group}. This group is represented as a group of unitary transformations on the space of square integrable complex valued functions $\Psi$ on $\mathbb R$ as follows
\begin{equation*} \begin{array}{rcl}
L^2(\mathbb R, \mathbb C) & \to & L^2(\mathbb R, \mathbb C) \\
\Psi(x) & \mapsto & \Psi'(x)=e^{i(y^2 x+y^3)}\Psi(x+y^1)  \, ,
\end{array}
\end{equation*}
where $(y^1,y^2,y^3) \in \mathbb R^3$ are three parameters. The subgroups of the Heisenberg group are $t \mapsto \Psi(x+t)$ (translation in position), $s \mapsto e^{is x}\Psi (x)$ (translation in momentum), and $u \mapsto e^{iu}\Psi(x)$ (multiplication by a phase).

We have as the group manifold $\mathbb R^3$ equipped with the following multiplication
\begin{equation*}
(y^1,y^2,y^3)(\tilde y^1,\tilde y^2,\tilde y^3)=(y^1+\tilde y^1,y^2+\tilde y^2,y^3+\tilde y^3+y^1\tilde y^2) \, .
\end{equation*}
\begin{equation*}
X=\frac{\partial}{\partial y^1} \quad , \quad Y=\frac{\partial}{\partial y^2}+y^1\frac{\partial}{\partial y^3} \quad \textrm{and} \quad Z=\frac{\partial}{\partial y^3}
\end{equation*}
generate the right action with
\begin{equation*}
\left[X,Y\right] = \left[\frac{\partial}{\partial y^1},\frac{\partial}{\partial y^2}+y^1\frac{\partial}{\partial y^3}\right]=\frac{\partial}{\partial y^3}= Z \quad , \quad \left[X,Z\right] = \left[Y,Z\right] = 0  \, .
\end{equation*}
The Heisenberg group is thus generated by the vectorfields $E_1=X, E_2=Y, E_3=Z$, which fulfill the commutation relations
\begin{equation*}
[E_1, E_2]=E_3 \, , \, [E_1, E_3]=0 \, , \, [E_2, E_3]=0 \, ,
\end{equation*}
and generate the right multiplication. The linear span of $(E_1, E_2, E_3)$ forms a Lie algebra (crystalline structure) corresponding to a uniform distribution of screw dislocations in a three-dimensional crystal lattice. Therefore, we can associate the Heisenberg group, which is the group corresponding to the crystalline structure, with the material manifold $\mathcal N$. 

If we take $\{E_1, E_2, E_3\}$ as a basis of $\mathcal V$, we have the following dual basis $\{\omega^1, \omega^2, \omega^3\}$ for $\mathcal V^*$
\begin{equation*}
\omega^1 =dy^1 \quad  , \quad  \omega^2=dy^2 \quad , \quad \omega^3=dy^3-y^1 dy^2 \, ,
\end{equation*}
and the corresponding metric
\begin{equation}\label{homspacemet}
\stackrel{\circ}{n}=(\omega^1)^2+(\omega^2)^2+(\omega^3)^2=(dy^1)^2+(dy^2)^2+(dy^3-y^1 dy^2)^2 \, , 
\end{equation}
which is a {\bf Bianchi type VII} metric. The manifold $\mathcal N$ endowed with this metric is a {\bf homogeneous space}.

Here, the dislocation density $\lambda$ turns out to be
\begin{equation*}
\lambda(E_1,E_2)(y)=E_3 \, , \, \lambda(E_1,E_3)(y)=\lambda(E_2,E_3)(y)=0 \, ,
\end{equation*}
and therefore
\begin{equation*}
\lambda =\left(\omega_1 \wedge \omega_2\right) E_3 = \left(dy^1\wedge dy^2\right)E_3 \, .
\end{equation*}
The integral of $\lambda$ over a surface $\Sigma$ in $\mathcal N$ is
\begin{equation*}
\int_{\Sigma}\lambda =\int_{\Sigma}\left(dy^1\wedge dy^2\right)E_3=A(\Pi\Sigma)E_3 \, ,
\end{equation*}
where $\Pi$ is the projection map of the line bundle of the homogeneous space (\ref{homspacemet}) over $\mathbb R^2$ with the standard metric on the base (the curvature of the bundle being $-dy^1\wedge dy^2$).


\section{The Thermodynamic State Space}\label{tdss}
Consider the space $S_2^+(\mathcal V)$ of inner products on the crystalline structure $\mathcal V$. 
The {\bf thermodynamic state space} is defined as the product 
\begin{equation*}
S_2^+(\mathcal V) \times \mathbb R^+
\end{equation*}
and its elements are $(\gamma,\sigma)$, where $\gamma \in S_2^+(\mathcal V)$ is the {\bf configuration} and $\sigma \in \mathbb R^+$ is the {\bf entropy per unit mass}.

Each $\gamma \in S_2^+(\mathcal V)$ defines a totally antisymmetric $n$-linear form $\omega_{\gamma}$ on the crystalline structure $\mathcal V$ by the condition that if $(E_1, \ldots, E_n)$ is a positive basis for $\mathcal V$, orthonormal with respect to $\gamma$, i.e.~$\gamma_{AB}:=\gamma(E_A,E_B)=\delta_{AB}$, then 
\begin{equation*}
\omega_{\gamma}(E_1,\ldots,E_n)=1 .
\end{equation*}
It follows that there is a positive function $V$ on $S_2^+(\mathcal V)$ such that 
\begin{equation*}
\omega_{\gamma}=V(\gamma)\omega \, .
\end{equation*}
The positive real number $V(\gamma)$ is the {\bf volume per unit mass} corresponding to the configuration $\gamma$. 

The {\bf thermodynamic state function} $\kappa$ is a real-valued function on the thermodynamic state space $S_2^+(\mathcal V)\times \mathbb R^+$. The Lagrangian which determines the dynamics will be defined through this function.

The {\bf thermodynamic stress} corresponding to a thermodynamic state $(\gamma,\sigma)$ is the element $\pi(\gamma,\sigma)$ of $(S_2(\mathcal V))^*$ defined by
\begin{equation*}
\frac{\partial\left(\kappa(\gamma,\sigma)V(\gamma)\right)}{\partial\gamma}=-\frac 1 2 \pi(\gamma,\sigma)V(\gamma) \, .
\end{equation*}
The {\bf temperature} corresponding to a thermodynamic state $(\gamma,\sigma)$ is the real number $\vartheta(\gamma,\sigma)$ given by
\begin{equation*}
\vartheta(\gamma,\sigma)=\frac{\partial\left(\kappa(\gamma,\sigma)V(\gamma)\right)}{\partial\sigma} \, , 
\end{equation*}
with the requirement that $\vartheta(\gamma,\sigma)$ is positive and tends to zero as $\sigma$ tends to zero.


\section{The Dynamics}
In the {\bf general theory of relativity} the space-time manifold is an oriented $(n+1)$-dimensional differentiable manifold $\mathcal M$ endowed with a {\bf Lorentzian metric} $g$, that is a continuous assignment of a symmetric bilinear form $g_x$ of index $1$ in $T_x\mathcal M, \, \forall x \in \mathcal M$. The Lorentzian metric divides $T_x\mathcal M$ into three different subsets $I_x, N_x, S_x$, the set of {\bf timelike}, {\bf null} and {\bf spacelike} vectors at $x$ respectively, according to whether $g_x$ restricted to the corresponding subset is negative, zero or positive. The subset $N_x$ is a double cone called the null cone at $x$. The subset $I_x$ is the interior of this cone, an open set of two components, the future and past component. $S_x$ is the subset outside the null cone, a connected open set for $n>1$. A curve $\gamma$ is called {\bf causal} if its tangent vector belongs to $I_x \cup N_x, \, \forall x \in \gamma$, and it is called timelike if its tangent vector belongs to $I_x$, $\forall x \in \gamma$. We assume that $(\mathcal M, g)$ is time oriented, that is a continuous choice of future component $I_x^+$ of $I_x$ can and has been made $\forall x \in \mathcal M$. This choice determines the future component $N_x^+$ of $N_x$ at each $x \in \mathcal M$. A timelike or causal curve is then future or past directed, according to which component its tangent vector belongs to. A hypersurface $\mathcal H$ is called spacelike if at each $x \in \mathcal H$ the restriction of $g_x$ to $T_x\mathcal H$ is positive definite. A spacelike hypersurface in $\mathcal M$ is called a {\bf Cauchy hypersurface} if each causal curve in $\mathcal M$ intersects $\mathcal H$ exactly once. We assume that $(\mathcal M, g)$ possesses such a Cauchy hypersurface \cite{CBG}.

The motion of the material continuum is described by a mapping $f$ from the space-time manifold $\mathcal M$ into the material manifold $\mathcal N$,
\begin{equation} \label{defmapf}
f: \mathcal M \to \mathcal N \, .
\end{equation}
This mapping specifies which material particle is at a given event in space-time. 
It is subject to the following conditions:
\begin{itemize}
\item[{\bf i)}] The mapping $f$, restricted to a Cauchy hypersurface $\mathcal H$, $\left. f \right|_{\mathcal H}$, is one to one.
\item[{\bf ii)}] The differential of the mapping $f$, $df(x)$, has a $1$-dimensional kernel contained in $I_x$, $\forall x \in \mathcal M$. 
\end{itemize}
Then, for each $y \in f(\mathcal M) \subset \mathcal N$, $f^{-1}(y)$ is a timelike curve in $\mathcal M$.
The {\bf material velocity} $u$ is the future directed unit tangent vectorfield of the timelike curve $f^{-1}$:
\begin{equation*}
\mathrm{span}(u_x)=\ker\left(df(x)\right)=T_xf^{-1}(y) \, , \, f(x)=y \, , \, g(u_x,u_x)=-1 \, , \forall x \in \mathcal M \, .
\end{equation*}
The {\bf simultaneous space} at $x$ is the orthogonal complement of the linear span of $u_x$:
\begin{equation*}
\Sigma_x = \left(\textrm{span}(u_x)\right)^{\perp} \, .
\end{equation*}
Note that $g_{\Sigma_x}=\left.g_x \right|_{\Sigma_x}$ is positive definite. The restriction of the differential $df(x)$ to $\Sigma_x$ is an isomorphism of $\Sigma_x$ onto $T_y\mathcal N$, where $y=f(x)$.
\begin{itemize}
\item[{\bf iii)}] The isomorphism $\left.df(x)\right|_{\Sigma_x}$ is orientation preserving.
\end{itemize}

The {\bf equations of motion}, a second order system of partial differential equations for the mapping $f$, are derived from a {\bf Lagrangian} $L$, a function on $\mathcal M$ which is constructed from $f$. The action $\mathcal A$ in a domain $\mathcal D \subset \mathcal M$ is the integral
\begin{equation*}
\int_{\mathcal D} L \, d\mu_g \, ,
\end{equation*}
where $d\mu_g$ is the volume form of $(\mathcal M,g)$.

Any mapping $f$ fulfilling the three requirements stated above, defines an orientation preserving isomorphism $j_{f,x}$ of $\mathcal V$ onto $\Sigma_x$ by
\begin{equation}\label{isoj}
j_{f,x}=\left(\left.df(x)\right|_{\Sigma_x}\right)^{-1}\circ \epsilon_{f(x)} \, .
\end{equation}
Define 
\begin{equation}
j_{f,x}^* g_x =\gamma \, ,
\end{equation}
$\gamma \in S_2^+(\mathcal V)$ and depends only on $f$ and $x$.
Note that 
\begin{equation*}
\mu(x) = \frac{1}{V\left(j^*_{f,x}g_x\right)}
\end{equation*}
is the rest mass density at $x$.

To take into account thermal effects in the Lagrangian picture we invoke the adiabatic condition which states that the entropy per unit mass of any element of the material remains unchanged. This allows us to consider the entropy per unit mass $\sigma$ as a given function on the material manifold $\mathcal N$.

The {\bf Lagrangian} function $L$ on $\mathcal M$ is defined by:
\begin{equation}\label{lagrangian}
L(x)=\kappa\left(j^*_{f,x}g_x,\sigma(f(x))\right) \, , \, \forall x \in \mathcal M \, ,
\end{equation}
where $\kappa$ is the thermodynamic state function on $S_2^+(\mathcal V) \times \mathbb R^+$. Note that $L(x)$ depends on $g$ only through $g_x$ and not on derivatives of $g$.

The {\bf energy-momentum-stress tensor} $T_x$ at $x$ is an element of the dual space $\left(S_2(T_x\mathcal M)\right)^*$ defined by
\begin{equation}\label{partiall}
\frac{\partial\left(L(x)d\mu_g(x)\right)}{\partial g_x}=-\frac 1 2 T_x d\mu_g(x) \, .
\end{equation}

From \cite{C2} we have the following
\begin{proposition} \label{ems}
We can write the energy-momentum-stress tensor at $x$ as follows:
\begin{equation} \label{enmomstrtens}
T_x=\rho(x) u_x\otimes u_x +S_x 
\end{equation}
with the mass-energy density $\rho$ (since we are in the relativistic framework this includes the rest mass energy) given by
\begin{equation} \label{rhokappav}
\rho(x)=\kappa(j^*_{f,x}g_x,\sigma(f(x))) \, .
\end{equation}
The stress tensor is given by
\begin{equation*}
S_x(\dot g_x)=\pi(j^*_{f,x}g_x,\sigma(f(x)))(j^*_{f,x}\dot g_x) \, .
\end{equation*}
Since $j^*_{f,x}\dot g_x=j^*_{f,x}\dot g_{\Sigma_x}$, $S_x$ can be viewed as an element of $\left(S_2(\Sigma_x)\right)^*$.
\end{proposition}

\begin{remark}
Hence we see from Proposition \ref{ems} that in the case of pure continuum mechanics (in the absence of electromagnetic fields) the energy per unit mass is $e=\kappa V$.
\end{remark}

Define the {\bf principal pressures} $p_i$, $i=1,\ldots,n$ as the eigenvalues of $S_{\flat \flat}$ relative to $\left. g\right|_{\Sigma}$, where the subscript $\flat \flat$ means lowering of the indices with respect to $g$. Note that the principal pressures $p_i : i=1,\ldots,n$ can equivalently be described as the eigenvalues of $\pi_{\flat \flat}$ relative to $\gamma$.

The positivity condition on $T_x$ requires that 
\begin{eqnarray}
H_x^+ & \to & T_x \mathcal M \nonumber \\
v^{\mu} & \mapsto & -T_{\nu}^{\mu} v^{\nu} \label{poscond}
\end{eqnarray}
maps into $\overline{I_x^+}$. This is equivalent to the condition that $|p_i|\leq \rho \quad : \, i=1,\ldots,n$. Here, in the case of crystalline solids, we assume the stronger condition that the range of (\ref{poscond}) lies in $I_x^+$ which is in turn equivalent to $|p_i| < \rho \quad : \, i=1,\ldots,n$.

\subsection{Variation of the Lagrangian}
Let $v=df(x) \in U_{(x,y)} \subset \mathcal L(T_x\mathcal M, T_y \mathcal N)$, $y=f(x)$, where $U_{(x,y)}$ is the open subset of the linear space $\mathcal L (T_x \mathcal M , T_y \mathcal N)$ consisting of those $v \in \mathcal L (T_x \mathcal M , T_y \mathcal N)$ which verify the two conditions 
\begin{itemize}
\item[1)] $\textrm{ker}\,v$ is a time like line in $T_x \mathcal M$,
\item[2)] with $\Sigma_x$ the $g_x$-orthogonal complement of $\textrm{ker}\,v$ in $T_x \mathcal M$, the isomorphism 
\begin{equation*}
\left. v \right|_{\Sigma_x} : \Sigma_x \to T_y \mathcal N
\end{equation*}
is orientation-preserving.
\end{itemize}
These above conditions 1) and 2) correspond to the conditions {\bf ii)} and {\bf iii)}, respectively. Note that {\bf i)} is not a local condition so it does not reduce to a condition on $v$.

Since the isomorphism 
\begin{equation*}
j_{f,x}= \left(\left. v\right|_{\Sigma_x}\right)^{-1} \circ \epsilon_y \, ,
\end{equation*}
as defined in (\ref{isoj}), depends only on $v$, we may write 
\begin{equation*}
i(v)=j_{f,x} \, ,
\end{equation*}
so  the configuration $\gamma_{f,x}$ corresponding to the isomorphism $v \in  U_{(x,y)}$ is given by the positive quadratic form on $\mathcal V$
\begin{equation} \label{defgamma}
\gamma (v) = i^*(v) \left.g\right|_{\Sigma_x} \, .
\end{equation}
Now $v \mapsto \gamma(v)$ is a mapping of
\begin{equation} \label{bundleB}
\mathcal B := \bigcup_{(x,y)\in \mathcal M \times \mathcal N} U_{(x,y)} \, ,
\end{equation} 
into $S_2^+(\mathcal V)$. The bundle $\mathcal B$, a bundle over $\mathcal M \times \mathcal N$, is the bundle over which the Lagrangian is defined.

The mapping $v \mapsto \gamma(v)$ is described by the functions $\gamma_{AB}(v)$ on $\mathcal B$, where
\begin{equation} \label{gammaofv}
\gamma_{AB}(v)=\gamma(v)\left(E_A,E_B\right) \, ,
\end{equation}
where $E_A: A=1,\ldots,n$ is a basis of $\mathcal V$ such that $\omega(E_1,\ldots,E_n)=1$.

By the properties of $v$ we have a positive basis $(X_A(x) \, : \, A=1,\ldots,n)$ of $\Sigma_x$ defined by:
\begin{equation*}
v \cdot X_A(x) =E_A(y) \, : \, A=1,\ldots,n \, ,
\end{equation*}
where $E_A(y)=\epsilon_y(E_A) \in T_y\mathcal N (A=1,\ldots,n)$, and thus from (\ref{defgamma}) and (\ref{gammaofv})
\begin{equation*}
\gamma_{AB}=\left.g\right|_{\Sigma_x}(X_A,X_B) \, .
\end{equation*}
Note that $(u,X_1,\ldots, X_n)$ is a frame field for $\mathcal M$.
Let $\dot v \in \mathcal L(T_x \mathcal M, T_y \mathcal N)$ be a variation of $v \in  U_{(x,y)}$. To describe $\dot v$ we must give $\dot v \cdot u \in T_y\mathcal N$ and $\dot v \cdot X_A \in T_y \mathcal N$,
\begin{eqnarray*}
\dot v \cdot u(x) & = & \dot v_0^A E_A(y) \, ,\\
\dot v \cdot X_A(x) & = & \dot v_A^B E_B(y) \, .
\end{eqnarray*} 
This is because $(E_A(y): A=1,\ldots,n)$ is a basis for $T_y\mathcal N$. So the $(\dot v_0^A : A=1, \ldots,n ; \dot v_A^B : A,B=1, \ldots,n)$ can be thought of as the the components of $\dot v$.

In view of (\ref{lagrangian}) and (\ref{rhokappav}) the Lagrangian function $L$ is $L(v)=\rho\left(\gamma(v),\sigma(y)\right)$, where $\rho$ is the relativistic energy-density which includes the rest mass contribution, a function of a thermodynamic state $\left(\gamma(v),\sigma(y)\right) \in S_2^+(\mathcal V)\times \mathbb R^+$. $L$ is a function on the bundle $\mathcal B$ (\ref{bundleB}). 
The Lagrangian form, a top degree form on $\mathcal M$, is $L d\mu_g$. 

The Lagrangian $L$ depends on $v$ through the configuration $\gamma$, $L(\gamma_{AB})$, and the first variation reads 
\begin{equation*}
\dot L=\frac{\partial \rho}{\partial \gamma_{AB}}\dot \gamma_{AB} \, ,
\end{equation*}
where the first variation of $\gamma_{AB}$ is given by
\begin{equation*} \label{firstvargamma}
\dot \gamma_{AB}=-\dot v_A^C\gamma_{BC}-\dot v_B^C \gamma_{AC} \, 
\end{equation*}
(see \cite{C1}). 
To formulate the hyperbolicity condition (see below), we need to consider the second variation of $L$ with respect to $v$, 
\begin{equation} \label{defofh*}
\ddot L=\frac{\partial^2 L}{\partial v^2}\cdot(\dot v,\dot v)=h(\dot v,\dot v) \, ,
\end{equation}
where, in general,
\begin{equation} \label{h*indetail}
h(\dot v,\dot v)=h^{00}_{AB} \dot v_0^A \dot v_0^B+ 2 h^{C0}_{AB} \dot v_C^A \dot v_0^B+ h^{CD}_{AB} \dot v_C^A \dot v_D^B \, .
\end{equation}
For $L=L(\gamma_{ab})$, using the formula for the second variation of $\gamma_{AB}$,
\begin{equation*} \label{secondvargamma}
\ddot \gamma_{AB}=2\left(\gamma_{AC}\gamma_{BD}\dot v_0^C\dot v_0^D+\gamma_{CD}\dot v_A^C\dot v_B^D+\gamma_{AC}\dot v_B^D\dot v_D^C+\gamma_{BC}\dot v_A^D\dot v_D^C \right) 
\end{equation*}
(see \cite{C1}), we obtain
\begin{eqnarray*}
\ddot L & = & \frac{\partial \rho}{\partial \gamma_{AB}}\ddot \gamma_{AB}+\frac{\partial^2 \rho}{\partial \gamma_{AB} \partial \gamma_{CD}}\dot \gamma_{AB}\dot \gamma_{CD} \\
 & = & 2\frac{\partial \rho}{\partial \gamma_{AB}}\left(\gamma_{AC}\gamma_{BD}\dot v_0^C\dot v_0^D+\gamma_{AC}\dot v_B^D\dot v_D^C+\gamma_{BC}\dot v_A^D\dot v_D^C+\gamma_{CD}\dot v_A^C\dot v_B^D \right) \\
 &   & + \frac{\partial^2 \rho}{\partial \gamma_{AB} \partial \gamma_{CD}}\big( \dot v_A^E\gamma_{EB}+\dot v_B^E \gamma_{EA} \big)\left( \dot v_C^F\gamma_{FD}+\dot v_D^F \gamma_{FC} \right) \\
 & = & 2 \frac{\partial \rho}{\partial \gamma_{CD}} \gamma_{AC}\gamma_{BD}\dot v_0^A\dot v_0^B + 2 \frac{\partial \rho}{\partial \gamma_{CD}} \gamma_{AB}\dot v_C^A\dot v_D^B \\
 &  & + 2 \left(\frac{\partial \rho}{\partial \gamma_{CE}} \gamma_{BE} \delta_A^D + \frac{\partial \rho}{\partial \gamma_{DE}} \gamma_{AE}\delta_B^C\right) \dot v_C^A\dot v_D^B\\
 &  & + 4  \frac{\partial^2 \rho}{\partial \gamma_{CE}\partial \gamma_{DF}}\gamma_{AE}\gamma_{BF}\dot v_C^A\dot v_D^B  \, .
\end{eqnarray*}
Comparing coefficients with (\ref{h*indetail}) using (\ref{defofh*}), we see that:
\begin{eqnarray*}
h^{00}_{AB} & = & 2 \frac{\partial \rho}{\partial \gamma_{CD}} \gamma_{AC}\gamma_{BD} \, ,\qquad h^{C0}_{AB}=0 \, , \\
h^{CD}_{AB} & = &  4  \frac{\partial^2 \rho}{\partial \gamma_{CE}\partial \gamma_{DF}}\gamma_{AE}\gamma_{BF}  \\ 
 &  & + 2\left(\frac{\partial \rho}{\partial \gamma_{CD}} \gamma_{AB}+  \frac{\partial \rho}{\partial \gamma_{CE}} \gamma_{BE} \delta_A^D + \frac{\partial \rho}{\partial \gamma_{DE}} \gamma_{AE}\delta_B^C \right) \, ,
\end{eqnarray*}
where the last expression will be made use of in the derivation of the Legendre-Hadamard conditions in Part \ref{static}.

\subsection{Hyperbolicity and Characteristic Speeds}
Set 
\begin{equation*}
G_{AB}= -h_{AB}^{00} \quad , \quad M_{AB}^{CD}=h_{AB}^{CD} \, .
\end{equation*}
Then
\begin{equation*} 
h(\dot v, \dot v)= -G_{AB}\dot v_0^A \dot v_0^B + M_{AB}^{CD} \dot v_C^A \dot v_D^B \, .
\end{equation*}
The hyperbolicity condition in general is that there is a pair $(T,\theta) \in T_x \mathcal M \times T_x^* \mathcal M$ with $\theta \cdot T > 0$ such that $h$ is \emph{negative-definite} on 
\begin{equation} \label{hnegdef}
\left\{ \dot v : \dot v = \theta \otimes Y , Y \in T_y \mathcal N \right\}
\end{equation}
and \emph{positive-definite} on 
\begin{equation} \label{hposdef}
\left\{ \dot v : \dot v = \kappa \otimes Y , \kappa \cdot T=0,  Y \in T_y \mathcal N \right\} 
\end{equation}
(see \cite{C1}).
We stipulate here that the above conditions hold with $(T,\theta)$ defined by the rest frame of the material at $x$, i.e.~$T=u_x$, $\Sigma_x=\textrm{ker}\,\theta$, $\theta\cdot T=1$.

For $\dot v$ in (\ref{hnegdef}) we have:
\begin{equation*}
\dot v \cdot u = Y \quad , \quad \dot v \cdot X_A = 0 \, ,
\end{equation*}
i.e.
\begin{equation*}
\dot v_0^A = Y^A \quad , \quad \dot v_A^B = 0 \, ,
\end{equation*}
and the condition that $h$ is negative definite on (\ref{hnegdef}) is the condition that
\begin{equation} \label{firstcondG}
G_{AB} Y^A Y^B > 0 \quad : \, \forall Y \neq 0 \, ,
\end{equation}
i.e.~that $G$ is positive-definite.

Since
\begin{equation*}
G_{AB}=\pi_{AB} + \rho \gamma_{AB}
\end{equation*}
(we are raising and lowering indices with respect to $\gamma_{AB}$) and the principal pressures $p_1,\ldots, p_n$ are the eigenvalues of $\pi_{AB}$ with respect to $\gamma_{AB}$, this condition is:
\begin{equation*}
\min_i p_i > -\rho \, ,
\end{equation*}
which of course follows from $\max_i |p_i| < \rho$.
For $\dot v$ in (\ref{hposdef}) we have:
\begin{equation*}
\dot v \cdot u = 0 \quad , \quad \dot v \cdot X_A = \kappa_A Y \, ,
\end{equation*}
where $\kappa_A= \kappa \cdot X_A$, i.e.~
\begin{equation*}
\dot v_0^A = 0 \quad , \quad \dot v_A^B = \kappa_A Y^B \, ,
\end{equation*}
and the condition that $h$ is positive definite on (\ref{hposdef}) is the condition that
\begin{equation} \label{hposdeftilde} 
M_{AB}^{CD} \kappa_C \kappa_D Y^A Y^B >0 \quad : \, \forall \kappa \neq 0, \forall Y \neq 0 \, .
\end{equation}
Set 
\begin{equation*}
H_{AB}(\kappa)=M_{AB}^{CD} \kappa_C \kappa_D \, .
\end{equation*}
Then condition (\ref{hposdeftilde}) is that $H_{AB}(\kappa)$ is positive-definite for all $\kappa \neq 0$.

If the first condition (\ref{firstcondG}) is satisfied, the second condition (\ref{hposdeftilde}) becomes the condition that for $\kappa \neq 0$ the eigenvalues $\lambda_1(\kappa),\ldots, \lambda_n(\kappa)$ of $H_{AB}(\kappa)$ with respect to $G_{AB}$ are all positive. Note that $H_{AB}(\kappa)$ is homogeneous of degree $2$ in $\kappa$, hence so are the $\lambda_1(\kappa), \ldots , \lambda_n(\kappa)$.

Now the characteristic matrix is
\begin{equation*}
\chi_{AB}(\xi)=h_{AB}^{\mu\nu} \xi_{\mu}\xi_{\nu} \, .
\end{equation*}  
We use the basis $(u,X_1,\ldots,X_n)$ for $T_x\mathcal M$. We denote by $\omega$ the frequency and by $\kappa_A$, $A=1,\ldots,n$, the wave number components, i.e.~
\begin{equation*}
\omega := \xi_0 := \xi \cdot u \quad , \quad \kappa_A := \xi_A := \xi \cdot X_A \, .
\end{equation*}
Then
\begin{equation*}
\chi_{AB}=-G_{AB} \omega^2 + H_{AB}(\kappa) \, ,
\end{equation*}
so
\begin{equation*}
\det \chi = \det G \cdot \prod_{i=1}^n \left( \lambda_i(\kappa)-\omega^2\right) \, .
\end{equation*}
We see that the second condition is equivalent to the $2n$ roots $\pm \sqrt{\lambda_i(\kappa)}$, $i:1,\ldots,n$, of the characteristic polynomial $\det \chi$ as a polynomial in $\omega$ being real. The characteristic speeds are
\begin{equation} \label{defetai}
\eta_i =\frac{\sqrt{\lambda_i(\kappa)}}{|\kappa|} \quad , \quad \textrm{where} \,\, |\kappa|=\sqrt{\left(\gamma^{-1}\right)^{AB}\kappa_A \kappa_B} \, .
\end{equation}
Note that the $\eta_i$ are homogeneous of degree $0$ in $\kappa$.
The causality condition is that the inner characteristic core in $T_x^* \mathcal M$ contains the null cone of $g$ in $T_x^* \mathcal M$. This reads (in units $c=1$)
\begin{equation*}
\eta_i < 1  \quad , \, \forall i=1,\ldots,n.
\end{equation*}

\section{Properties of the Energy per unit Mass}

{\postulate \label{pmine} We stipulate that the energy per unit mass $e=\rho V$ has a strict minimum at a certain inner product $\stackrel{\circ}{\gamma}$.}
{\remark The above postulate has an intuitive physical interpretation. Note that the set of inner products $\gamma$ is an open positive cone in the linear space of quadratic forms. For large $\gamma$ (large expansion) and also for $\gamma$ near the boundary (large compression) the energy $e$ is physically expected to blow up, so we can restrict ourselves to a compact set of inner products, where $e$ necessarily attains a minimum.}

\vspace{2mm}  

We choose $(E_1,\ldots,E_n)$ to be an orthonormal basis relative to $\stackrel{\circ}{\gamma}$, so $\stackrel{\circ}{\gamma}_{AB}=\delta_{AB}$. This is compatible with the previous condition $\omega(E_1,\ldots,E_n)=1$ for a suitable volume form $\omega$ on $\mathcal V$. 
This choice of $\omega$ corresponds to a choice of unit of mass so that the mass density associated to the configuration $\stackrel{\circ}{\gamma}$ is equal to $1$.
$\stackrel{\circ}{\gamma}$ defines a metric $\stackrel{\circ}{n}$ on $\mathcal N$ by
\begin{equation*}
\stackrel{\circ}{n}=\sum\limits_{A,B=1}^n\stackrel{\circ}{\gamma}_{AB}\omega^A \otimes \omega^B = \sum\limits_{A=1}^n \omega^A \otimes \omega^A \, ,
\end{equation*}
where $(\omega^1,\ldots,\omega^n)$ is the dual basis to $(E_1,\ldots,E_n)$. If $\mathcal V$ is a Lie algebra, so that $\mathcal N$ is a Lie group, then $\stackrel{\circ}{n}$ is a left-invariant metric, i.e.~invariant under the actions of $\mathcal N$ on itself by left multiplications by elements of $\mathcal N$.
Thus $(\mathcal N, \stackrel{\circ}{n})$ is a homogeneous Riemannian manifold (which in general is not isotropic).
The Riemannian manifold $(\mathcal N, \stackrel{\circ}{n})$ has curvature except when $\mathcal V$ is Abelian, so there are no dislocations. 

Note that $d\mu_{\stackrel{\circ}{n}}=\sqrt{\det \stackrel{\circ}{n}} \, d^n y = d\mu_{\omega}$, therefore we have for any domain $\Omega$ in $\mathcal N$
\begin{equation*}
M(\Omega)=\int\limits_{\Omega} d\mu_{\omega} = \int\limits_{\Omega} d\mu_{\stackrel{\circ}{n}} \, , 
\end{equation*}
that is, in our choice of units the mass of $\Omega$ is equal to its volume with respect to $\stackrel{\circ}{n}$.

\subsection{The Isotropic Case} \label{isotropiccase}
Let us define the orthogonal group corresponding to $\stackrel{\circ}{\gamma}$,
\begin{equation*}
O_{\stackrel{\circ}{\gamma}}=\left\{  O \in \mathcal L(\mathcal V, \mathcal V) : \, \stackrel{\circ}{\gamma}(OX,OY)= \, \stackrel{\circ}{\gamma}(X,Y) \, , \, \forall X,Y \in \mathcal V \right\} \, .
\end{equation*}
$O_{\stackrel{\circ}{\gamma}}$ acts on $S_2^+(\mathcal V)$ in the following way: $\gamma  \mapsto  O \gamma $, where
\begin{equation*}
\left( O\gamma\right)(X,Y) = \gamma(OX,OY) \, .
\end{equation*}
Since $\stackrel{\circ}{\gamma} \, \mapsto O \! \stackrel{\circ}{\gamma} \, =\stackrel{\circ}{\gamma}$, the symmetric bilinear form $\stackrel{\circ}{\gamma}$ is a fixed point of the action of $O_{\stackrel{\circ}{\gamma}}$ on $S_2^+(\mathcal V)$.

If the energy density $e(\gamma)$ is invariant under $O_{\stackrel{\circ}{\gamma}}$ we are in the case of {\bf isotropic elasticity}.
Then $e(\gamma)$ depends only on the eigenvalues $\lambda_1,\ldots ,\lambda_n$ of $\gamma$ relative to $\stackrel{\circ}{\gamma}$,
\begin{equation*}
e(\gamma)=e(\lambda_1, \ldots , \lambda_n) \, ,
\end{equation*}
where $e(\lambda_1, \ldots , \lambda_n)$ is totally symmetric in its arguments. This is the simplest case of an energy density. 

Consider 
\begin{equation*}
q=f^*\stackrel{\circ}{n} \, .
\end{equation*}
This is a $2$-covariant symmetric tensorfield on $\mathcal M$, i.e.~at each $x \in \mathcal M$ $q_x$ is a quadratic form in $T_x \mathcal M$. The vector $u_x$ belongs to the null space of $q_x$ and the restriction $\left. q_x \right|_{\Sigma_x}$ is positive-definite.

Let $\Lambda_1,\ldots,\Lambda_n$ be the eigenvalues of $\left.q_x\right|_{\Sigma_x}$ relative to $\left.g_x\right|_{\Sigma_x}$. We then have
\begin{proposition}
The eigenvalues $\Lambda_1,\ldots,\Lambda_n$ are the inverses of the eigenvalues $\lambda_1,\ldots,\lambda_n$ of $\gamma=j_{f,x}^* g_x$ relative to $\stackrel{\circ}{\gamma}$. In particular,
\begin{equation*}
\sqrt{\Lambda_1 \cdot \ldots \cdot \Lambda_n}=\frac 1{\sqrt{\lambda_1 \cdot \ldots \cdot \lambda_n}} = \frac 1 v = N \, .
\end{equation*}
Therefore, in the variational principle, the crystalline structure $\mathcal V$ on $\mathcal N$ is eliminated in favor of the Riemannian metric $\stackrel{\circ}{n}$.
\end{proposition}

\section{Equivalences}

\subsection{Equivalence of Crystalline Structures}
{\definition Two crystalline structures $\mathcal V$ and $\mathcal V'$ on $\mathcal N$ are said to be equivalent if there is a diffeomorphism $\psi$ of $\mathcal N$ onto itself such that $\psi_*$, the push-forward of $\psi$, induces an isomorphism of $\mathcal V$ onto $\mathcal V'$.}

Let $\nu$ be the canonical $1$-form (\ref{defnu}) associated to $\mathcal V$ and $\nu'$ the one associated to $\mathcal V'$. We have
\begin{equation} \label{equivcs}
\begin{array}{rcl}
\nu(Y_y) & = & Y \in \mathcal V \quad : \, Y_y \in T_y\mathcal N \, ,\\
\nu'(d\psi \cdot Y_y) & = & \psi_* Y \in \mathcal V' \quad : \, d\psi \cdot Y_y \in T_{\psi(y)}\mathcal N \, .
\end{array}
\end{equation}
We may take the above as the definition of pullback for $\mathcal V$-valued $1$-forms:
\begin{equation*}
\nu = \psi^* \nu' \, .
\end{equation*}
It then follows that:
\begin{equation} \label{equivlambda}
\lambda=\psi^* \lambda' \, ,
\end{equation} 
with the natural extension of the above definition of pullback to $\mathcal V$-valued $2$-forms, i.e.~
\begin{equation*}
\lambda'\left(d\psi \cdot X_y, d\psi \cdot Y_y\right) = \psi_*(\lambda(X_y,Y_y)) \quad : \, \forall X,Y \in \mathcal V, \forall y \in \mathcal N \, .
\end{equation*} 
In fact, we have for all $X,Y \in \mathcal V$ and $y \in \mathcal N$:
\begin{eqnarray*} 
\Lambda'(\psi_*X,\psi_*Y)\left(\psi(y)\right) & = & \nu'\left([\psi_* X,\psi_* Y](\psi(y)) \right) \\
& = & \nu'\left(\psi_* [X,Y](\psi(y))\right) \\
& = & \nu' \left( d\psi \cdot [X,Y](y)\right) \\
& = & \psi_*\left(\nu([X,Y](y))\right) \\
& = & \psi_*\left(\Lambda(X,Y)(y)\right) \, ,
\end{eqnarray*}
where we have made use of the fact that $[\psi_* X, \psi_* Y]=\psi_* [X,Y]$.

\subsection{Equivalence of Mechanical Properties}
We investigate the question of the equivalence of the mechanical properties of a solid. A certain solid phase of a certain substance is described by an equation of state.

Let $A:\mathcal V \to \mathcal V'$ be a linear isomorphism. The group of all such $A$ is homomorphic to $GL_n(\mathbb R)$. Then $A^*: S_2^+(\mathcal V') \to S_2^+(\mathcal V)$ is the induced isomorphism defined by $\gamma = A^* \gamma'$, where
\begin{equation*}
\gamma(X,Y)=\gamma'(A X, A Y) \quad , \, \forall X,Y \in \mathcal V  \, .
\end{equation*} 
The corresponding energy functions on $S_2^+(\mathcal V) $ and $S_2^+(\mathcal V')$ are denoted by $e$ and $e'$, respectively.
{\definition  \label{equivnrg} Two energy functions $e$ on $\mathcal V$ and $e'$ on $\mathcal V$ are equivalent if there exists a linear isomorphism $A: \mathcal V \to \mathcal V'$ such that $e'(\gamma')=e(\gamma)$, where $ \gamma = A^*\gamma'$, for all $\gamma' \in S_2^+(\mathcal V')$.}

We illustrate the above definition of mechanical equivalence by giving an example in the isotropic case. Let $\gamma \in S_2^+(\mathcal V)$ be given, and define $M\in \mathcal L(\mathcal V, \mathcal V)$ by
\begin{equation*}
\stackrel{\circ}{\gamma}(M X, Y) = \gamma(X,Y) \quad , \, \forall X,Y \in \mathcal V  \, .
\end{equation*} 
Similarly, given an isomorphism $A$ as above (i.e.~$\gamma=A^*\gamma'$), we define $M' \in \mathcal L(\mathcal V', \mathcal V')$ from the corresponding $\gamma' \in S_2^+(\mathcal V')$ as
\begin{equation*}
\stackrel{\circ}{\gamma'}(M' X', Y') = \gamma'(X',Y') \quad , \, \forall X',Y' \in \mathcal V' \, .
\end{equation*}
In the above, $\stackrel{\circ}{\gamma}=A^*\stackrel{\circ}{\gamma'}$.
Then $M'=AMA^{-1}$. 
For, given any $X', Y' \in \mathcal V'$ let $X, Y \in \mathcal V$ be $X=A^{-1}X'$, $Y=A^{-1}Y'$. Then
\begin{eqnarray*}
\stackrel{\circ}{\gamma'}(M' X', Y') & = & \gamma'(X',Y')=\gamma(X,Y)=\stackrel{\circ}{\gamma}(M X, Y) \\ 
& = & \stackrel{\circ}{\gamma'}(AM X, A Y)=\stackrel{\circ}{\gamma'}(AMA^{-1} X', Y') \, .
\end{eqnarray*}

Therefore, the eigenvalues $\lambda_1',\ldots,\lambda_n'$ of $M'$ coincide with the eigenvalues $\lambda_1,\ldots,\lambda_n$ of $M$. In the isotropic case, $e'(\gamma')=e(\gamma)$ means 
\begin{equation*} 
e'(\lambda_1',\ldots,\lambda_n')=e(\lambda_1,\ldots,\lambda_n) \, ,
\end{equation*}
where both sides of the above equation are symmetric functions, so the energy functions coincide, i.e.~$e'=e$.

The equivalence of the energy functions, however, does not fully capture the equivalence of having two solids of the same substance in the same phase. What is required in addition is to have the same equilibrium mass density of infinitesimal portions.
In fact, it is the triplet $(\mathcal V, \omega, e)$ which defines a solid with its mechanical properties. Two solids of the same substance in the same phase, but with possibly different dislocation structures, to the extent that they can be described by the same manifold $\mathcal N$ (which is true if they are diffeomorphic) are defined by the triplets $(\mathcal V, \omega, e)$ and  $(\mathcal V', \omega', e')$, respectively. Additionally, there is an isomorphism $A: \mathcal V \to \mathcal V'$, such that
\begin{equation*}
\omega = A^* \omega' \, ,
\end{equation*}
i.e.~$\omega(X_1, \ldots , X_n)=\omega'(AX_1, \ldots, AX_n) \, \, : \, \forall X_1, \ldots, X_n \in \mathcal V$ as well as 
\begin{equation}
e'(\gamma')=e(\gamma) \quad , \quad \gamma = A^* \gamma' \quad : \, \forall \gamma' \in S_2^+(\mathcal V') \, 
\end{equation}
in accordance with Definition \ref{equivnrg}. Thus if $(E_1, \ldots, E_n)$ is a positive basis for $\mathcal V$ which is orthonormal relative to $\stackrel{\circ}{\gamma}$, then
\begin{equation*}
\omega_{\stackrel{\circ}{\gamma}}(E_1,\ldots,E_n)=1 \, ,
\end{equation*}
while $\omega(E_1,\ldots,E_n)=\mu_0$. On the other hand, with $E_i'=A E_i \, : i=1, \ldots , n$, $(E_1', \ldots , E_n')$ is orthonormal relative to $\stackrel{\circ}{\gamma'}$, thus 
\begin{equation*}
\omega_{\stackrel{\circ}{\gamma'}}(E_1',\ldots,E_n')=1 \, ,
\end{equation*}
while $\omega'(E_1',\ldots,E_n')=\omega(E_1,\ldots,E_n)=\mu_0$. Consequently, the mass density corresponding to $\stackrel{\circ}{\gamma'}$ relative to the triplet $(\mathcal V', \omega', e')$ is the same as the mass density corresponding to $\stackrel{\circ}{\gamma}$ relative to the triplet $(\mathcal V, \omega, e)$.

\begin{remark}
If we have two triplets $(\mathcal V, \omega, e)$ and $(\mathcal V', \omega', e')$ which correspond to the same substance in the same phase and, additionally, the isomorphism $A: \mathcal V \to \mathcal V'$ is of the form $A=\psi_*$, where $\psi$ is a diffeomorphism of $\mathcal N$ onto itself, then the dislocation structures are also equivalent. In this case, if $f: \mathcal M \to \mathcal N$ is a dynamical solution of the problem corresponding to $(\mathcal V, \omega, e)$, then $f':=\psi \circ f : \mathcal M \to \mathcal N$ is a dynamical solution of the problem corresponding to $(\mathcal V', \omega', e')$.
\end{remark}

\subsection{Similarity} \label{similarity}
Let now $\mathcal V$ be fixed, so that $S_2^+(\mathcal V)$ is also fixed. Two different energy functions $e$ and $e'$ may be related as follows. There is a constant $a>0$ and an isomorphism $\mathcal V \to \mathcal V$, defined by
\begin{equation*}
X \mapsto a X \quad : \, \forall X \in \mathcal V \, .
\end{equation*}
This defines an isomorphism $S_2^+(\mathcal V) \to S_2^+(\mathcal V)$ by $\gamma \mapsto \gamma'$, where
\begin{equation*}
\gamma'(X,Y)=\gamma(aX,aY)=a^2 \gamma(X,Y) \quad : \, \forall X,Y \in \mathcal V \, ,
\end{equation*}
i.e.~$\gamma \mapsto \gamma' = a^2 \gamma$.

According to the above discussion of equivalence of mechanical properties, we must have $e(\gamma)=e'(\gamma')$, $\forall \gamma \in S_2^+(\mathcal V)$, \emph{and}, additionally $\omega'=a^n\omega$ for the same material in the same phase, for
\begin{equation*}
\omega'(X_1,\ldots,X_n)=\omega(aX_1,\ldots,X_n)=a^n\omega(X_1,\ldots,X_n) \, .
\end{equation*}
Then the two theories (primed and unprimed) represent the same material in the same phase.

Moreover, since $\mathcal V$ is identical in the two theories, the dislocation structures may seem identical. However, for a given domain $\Omega \subset \subset \mathcal N$, which has the same crystalline structure $\left.\mathcal V\right|_{\Omega}$ (the restriction to $\Omega$ of the vectorfields in $\mathcal V$), the primed theory actually assigns a \emph{physical dislocation density} $a^{-2}$ times that of the unprimed theory. This is due to the fact that $\omega'=a^n\omega$ and the observation that in any dimension, the elementary dislocations are co-dimension two objects. Thus their density refers to the two-dimensional measure of a cross-section.


\section{The Eulerian Picture}

As a prelude to formulating the Eulerian picture, we view the crystalline structure $\mathcal V$ as an abstract $n$-dimensional real vector space, as in Remark \ref{mfo}. We then view the canonical form $\nu$ as a $1$-form on $\mathcal N$ with values in $\mathcal V$ such that $\left. \nu \right|_{T_y\mathcal N}$ is an isomorphism from $T_y \mathcal N$ onto $\mathcal V$ for all $y \in \mathcal N$. Thus, given any $v \in \mathcal V$, we have a tangent vector
\begin{equation*}
\left(\left. \nu \right|_{T_y\mathcal N}\right)^{-1} \cdot v = Y_y(v) \in T_y\mathcal N \quad : \, \forall y \in \mathcal N \, ,
\end{equation*}
that is a vectorfield $Y(v)$ on $\mathcal N$. Then the crystalline structure in the original sense is the space
\begin{equation*}
\{ Y(v) : v \in \mathcal V \} 
\end{equation*}
of vectorfields on $\mathcal N$.

To obtain the Eulerian description we must eliminate the material manifold $\mathcal N$. A fundamental variable is the material velocity $u$, a future-directed timelike unit vectorfield on $\mathcal M$. This defines the distribution of local simultaneous spaces
\begin{equation} \label{lss}
\Sigma=\{\Sigma_x : x \in \mathcal M \} \, ,
\end{equation}
where $\Sigma_x$ is the orthogonal complement of $u_x$ in $T_x \mathcal M$.

We then need another entity defined on $\mathcal M$ to play the role of the canonical form $\nu$. Consider
\begin{equation*}
\xi = f^* \nu \, .
\end{equation*}
This is a $1$-form on $\mathcal M$ with values in $\mathcal V$, viewed as an abstract $n$-dimensional vector space.

The $1$-form $\xi$ has the following properties:
\begin{itemize}
\item[{\bf 1)}]
$\xi \cdot u = 0$,
\item[{\bf 2)}]
for each $x \in \mathcal M$, $\left. \xi \right|_{\Sigma_x}$ is an isomorphism from $\Sigma_x$ onto $\mathcal V$,  {\it and} 
\item[{\bf 3)}]
$\mathcal L_u \xi =0$.
\end{itemize}
We thus introduce ab initio a $1$-form $\xi$ on $\mathcal M$ with values in $\mathcal V$ possessing the above three properties, as another fundamental Eulerian variable besides $u$.

By {\bf 2)}, given any $v \in \mathcal V$, we have a tangent vector
\begin{equation*}
\left(\left.\xi\right|_{\Sigma_x}\right)^{-1} \cdot v = X_x(v) \in \Sigma_x 
\end{equation*}
at each $x \in \mathcal M$, that is we obtain a vectorfield $X(v)$ whose value at each point belongs to the distribution (\ref{lss}), i.e.~it is orthogonal to $u$.

A mapping $\mathcal M \to S_2^+(\mathcal V)$, $x \mapsto \gamma_x$, is then defined by:
\begin{equation} \label{defgammax}
\gamma_x(v_1,v_2)=g_x(X_x(v_1),X_x(v_2))=\left.g\right|_{\Sigma_x}(X_x(v_1),X_x(v_2)) \quad : \, \forall v_1, v_2 \in \mathcal V \, .
\end{equation}
The last fundamental variable is the entropy $s$, a positive function on $\mathcal M$.

The volume form $\omega$, the space of thermodynamic configurations $S_2^+(\mathcal V)$, and the volume per unit mass $V$ are defined as in Section \ref{tdss}. The thermodynamic state space is $S_2^+(\mathcal V)\times \mathbb R^+$ and the energy per unit mass $e$ is a function on this space, as before. The thermodynamic stress $\pi$ at each $(\gamma, s) \in S_2^+(\mathcal V)\times \mathbb R^+$, $\pi(\gamma, s) \in \left(S_2(\mathcal V)\right)^*$, is defined by
\begin{equation*}
\frac{\partial e(\gamma, s)}{\partial\gamma}=-\frac 1 2 V(\gamma) \pi(\gamma,s)\, ,
\end{equation*}
as in Section \ref{tdss}. Also, the temperature $\theta$ is defined, as before, by
\begin{equation*}
\theta(\gamma,s)=\frac{\partial e(\gamma, s)}{\partial s} \, .
\end{equation*}
Thus 
\begin{equation*}
de = -\frac 1 2 V \pi \cdot d\gamma + \theta ds
\end{equation*}
expresses the first law of thermodynamics in the present framework.

The stress $S$ is a $2$-contravariant symmetric tensorfield on $\mathcal M$, an assignment of an element $S_x \in \left(S_2(\Sigma_x)\right)^*$ at each $x \in \mathcal M$. This is defined as follows. Given any $\dot g_x \in S_2(T_x \mathcal M)$, we define $\dot \gamma_x \in S_2^+(\mathcal V)$ as in (\ref{defgammax}) by:
\begin{equation*}
\dot \gamma_x(v_1,v_2)=\dot g_x(X_x(v_1),X_x(v_2))=\left.\dot g\right|_{\Sigma_x}(X_x(v_1),X_x(v_2)) \quad : \, \forall v_1, v_2 \in \mathcal V \, .
\end{equation*}
Then
\begin{equation*}
S_x(\dot g_x)= \pi(\gamma_x,s(x))(\dot \gamma_x) \, .
\end{equation*}
The mass-energy density $\rho$ is then the positive function on $\mathcal M$ given by
\begin{equation*}
\rho(x)=\frac{e(\gamma_x, s(x))}{V(\gamma_x)} \, ,
\end{equation*}
and the energy-momentum-stress tensor is defined according to (\ref{enmomstrtens}).

The Eulerian equations of motion are a first order system of partial differential equations consisting of 
\begin{itemize}
\item[{\bf (a)}] $\mathcal L_u \xi =0$, i.e.~property {\bf 3)} of $\xi$, {\it and}
\item[{\bf (b)}] $\nabla \cdot T=0$. 
\end{itemize}

\begin{remark}
In $n$ space dimensions, $\mathrm{dim} \mathcal M=n+1$, there are $n^2+n+1$ dependent variables, the $n^2$ algebraically independent components of $\xi$, the $n$ algebraically independent components of $u$, and also $s$. The Eulerian equations are also $n^2+n+1$ in number, $n^2$ independent equations in {\bf (a)} and $n+1$ independent equations in {\bf (b)}.
\end{remark}

For solutions of the Eulerian equations such that $u$, $\xi$, and $s$ are continuous, {\bf (b)} implies the adiabatic condition on $s$:
\begin{equation*}
u(s)=0 \, .
\end{equation*}
For the proof of this fact see \cite{C2}.

\begin{remark}
The case of absence of dislocations is the case $d \xi =0$.
\end{remark}

Let $(E_A: A=1,\ldots,n)$ be a basis for $\mathcal V$. Given any $\gamma \in S_2^+(\mathcal V)$, we set
\begin{equation*}
\gamma_{AB}=\gamma(E_A,E_B) \, .
\end{equation*}
According to the above, at each $x \in \mathcal M$ there is a unique $X_A \in \Sigma_x$ such that 
\begin{equation*}
\xi \cdot X_A = E_A \quad : \, A=1, \ldots, n \, .
\end{equation*}
Thus $X_A$ is a vectorfield on $\mathcal M$ belonging to $\Sigma$.
Then
\begin{equation*}
S = \pi^{AB} X_A \otimes X_B \, , 
\end{equation*}
where $\pi^{AB} \in S_2^+(\mathcal V)^*$ is the thermodynamic stress defined above. 

Fix an element $v \in \mathcal V$ and consider the vectorfield $X\subset \Sigma$ such that
\begin{equation*}
\xi \cdot X = v \, .
\end{equation*}
Then
\begin{equation*}
0=u(\xi \cdot X)=\left(\mathcal L_u \xi \right) \cdot X + \xi \cdot \left[ u , X \right] \, ,
\end{equation*}
i.e.~(according to {\bf 3)}) $ \xi \cdot \left[ u , X \right]=0$. It follows that 
\begin{equation}\label{projpi}
\Pi\left[u,X\right] = 0 \, ,
\end{equation}
where $\Pi$ is the $g$-orthogonal projection to $\Sigma$.

\begin{proposition}
Define $\mu:=\frac 1 V$. Then the mass current $I=\mu u$ satisfies the equation of continuity:
\begin{equation} \label{continuity}
\nabla \cdot I = 0 \, .
\end{equation}
\end{proposition}

\begin{proof}
$\mathcal V$ is endowed with a volume form $\omega$ and $\omega_{\gamma}=V \omega$.
Choose a basis $(E_1,\ldots,E_n)$ for $\mathcal V$ such that $\omega(E_1,\ldots,E_n)=1$.
Then $V=\sqrt{\det \gamma}$, and, by (\ref{defgammax}),
\begin{equation}\label{gammag}
\gamma_{AB}=\gamma(E_A,E_B)=g(X_A,X_B) \, ,
\end{equation}
hence
\begin{equation} \label{uv}
u (V) = \frac 1 2 V \left(\gamma^{-1}\right)^{AB} u (\gamma_{AB}) \, ,
\end{equation}
and, using (\ref{gammag}),
\begin{eqnarray*}
u(\gamma_{AB}) & = &  u(g(X_A,X_B)) \\
 & = & g(\nabla_u X_A, X_B) + g(X_A, \nabla_u X_B) \\
 & = & g(\nabla_{X_A}u, X_B) + g(X_A, \nabla_{X_B}u) +g([u,X_A],X_B)+g(X_A, [u,X_B]) \, .
\end{eqnarray*}
By (\ref{projpi}), the last two terms in the above equation vanish.
Let us define $\kappa$ by 
\begin{eqnarray}\label{defkappa}
\nabla_{X_A}u=\kappa_A^B X_B \, .
\end{eqnarray}
[$(X_A : A=1, \ldots, n)$ is a basis for $\Sigma_x$ at each point.] 
Then 
\begin{equation*}
g(\nabla_{X_A}u,X_B)= \kappa_A^C g(E_C,E_B)=\gamma_{CB} \kappa_A^C \, .
\end{equation*}
Hence we obtain
\begin{equation*}
u(\gamma_{AB})=\gamma_{CB}\kappa_A^C+\gamma_{AC}\kappa_B^C \, .
\end{equation*}
Substituting in (\ref{uv}) yields
\begin{equation*}
u(V)=\frac 1 2 V \left(\gamma^{-1}\right)^{AB}u(\gamma_{AB})= V \textrm{tr} \, \kappa  \, .
\end{equation*}
Then $\mu=V^{-1}$ satisfies 
\begin{equation*}
u(\mu)+\mu\textrm{tr} \, \kappa = 0 \, ,
\end{equation*}
and since, from (\ref{defkappa}), $\textrm{tr} \,\kappa = \nabla \cdot u$, we obtain
\begin{equation}
\nabla \cdot I = \nabla (\mu u) = u(\mu) + \mu \nabla \cdot u= 0 \, ,
\end{equation}
which establishes (\ref{continuity}). 
\end{proof}

\subsection{The Non-relativistic Limit}
First of all, we restrict ourselves to the case where $(\mathcal M, g)$ is the Minkowski space-time.
Then we consider the non-relativistic limit, where Minkowski space-time is replaced by Galilean space-time. There, we have the hyperplanes $\Sigma_t$ of absolute simultaneity, which are isometric to $n$-dimensional Euclidean space.
Any family of parallel lines transversal to the $\Sigma_t$ represents a Galilean frame, that is, a family of observers in uniform motion and at rest relative to each other. Any such family of parallel lines defines an isometry of the $\Sigma_t$ onto each other.

In terms of a Galilean frame and a rectangular coordinate system $(x^1, \ldots, x^n)$ in Euclidean space, and with $x^0=ct$ ($c$: the speed of light in vacuum), the space-time velocity $u = u^{\mu}\frac{\partial}{\partial x^{\mu}}$ is represented in terms of the space velocity $v=v^i\frac{\partial}{\partial x^i}$ by
\begin{equation*}
u^0=\frac 1{\sqrt{1-|v|^2/c^2}} \quad , \quad u^i = \frac{\frac{v^i}{c}}{\sqrt{1-|v|^2/c^2}} \, .
\end{equation*}
Therefore, in the non-relativistic limit $c \to \infty$
\begin{equation*}
c u = \frac{\partial}{\partial t}+ v^i \frac{\partial}{\partial x^i} \, .
\end{equation*}
Also, the condition 
\begin{equation*}
\xi_{\mu} u^{\mu} = 0
\end{equation*}
is simply $\xi_0=0$ and thus $\xi$ becomes a $\mathcal V$-valued $1$-form on each $\Sigma_t$,
\begin{equation*}
\xi=\xi_i dx^i \, .
\end{equation*}
Equations {\bf (a)} become
\begin{equation}
\frac{\partial \xi}{\partial t}+ \mathcal L_v \xi =0 \quad \left( \textrm{or} \quad \frac{\partial \xi_i}{\partial t}+v^j \frac{\partial \xi_i}{\partial x^j}+ \xi_j \frac{\partial v^j}{\partial x^i}=0 \right) \, ,
\end{equation}
and
\begin{equation*}
S^{\mu \nu} u_{\nu} = 0 \quad , \quad u_{\nu}=g_{\mu\nu}u^{\mu} \, ,
\end{equation*}
reads (note that $u_{\nu}=g_{\kappa \nu}u^{\kappa}$, i.e.~$u_0=-u^0$, while $u_i=u^i$)
\begin{eqnarray} \label{sio}
S^{i0}-S^{ij} \frac{v^j}{c} & = & 0 \, , \\
S^{00}-S^{0i} \frac{v^i}{c} & = & 0  \, .
\end{eqnarray}
So, in the limit $c \to \infty$, we have $S^{i0}=S^{00}=0$ and, therefore,
\begin{equation*}
S=S^{ij} \frac{\partial}{\partial x^i}\otimes \frac{\partial}{\partial x^j} \, .
\end{equation*}
Let 
\begin{eqnarray*}
e & = & c^2 + e' \, , \\
\rho &  = & \mu e = \mu c^2 + \varepsilon \quad , \quad \varepsilon=\mu e' \, .
\end{eqnarray*}
We have for the mass current $I^{\nu}=\mu u^{\nu}$: 
\begin{eqnarray*}
I^0 & = & \frac{\mu}{\sqrt{1-|v|^2/c^2}} \, = \, \mu + \frac 1 2  \frac{\mu |v|^2}{c^2} + O(c^{-4}) \, , \\
I^i & = & \frac{\mu\frac{v^i}{c}}{\sqrt{1-|v|^2/c^2}} \, = \, \frac{\mu v^i}{c^2} + \frac 1 2  \frac{\mu|v|^2 v^i}{c^3} + O(c^{-5})  \, .
\end{eqnarray*}
In the non-relativistic limit the equation of continuity $\nabla_{\nu} I^{\nu}=0$ becomes the classical continuity equation:
\begin{equation}
\frac{\partial \mu}{\partial t}+ \frac{\partial \left( \mu v^i \right)}{\partial x^i} = 0 \, .
\end{equation}
Similarly, for the energy-momentum-stress tensor $T_{\kappa\lambda}$, we have
\begin{eqnarray*}
T^{00} & = & \rho (u^0)^2 + S^{00} \, = \, \mu c^2 + (\varepsilon + \mu v^2) + O(c^{-2}) \, , \\
T^{0i} & = & \rho u^0 u^i + S^{0i} \, = \, \mu v^i c + (\varepsilon + \mu v^2)\frac{v^i}{c^2} - S^{ij} \frac{v^j}{c} +  O(c^{-2}) \, ,\\
T^{ij} & = & \frac{\rho v^i v^j/c^2}{1-|v|^2/c^2}+ S^{ij} \, = \, \mu v^i v^j + S^{ij} +  O(c^{-2}) \, .
\end{eqnarray*}
The equations of motion $\nabla_{\lambda} T^{\kappa\lambda}=0$ reads
\begin{eqnarray*}
\kappa=0: \quad \frac{\partial T^{00}}{\partial x^0} + \frac{\partial T^{0i}}{\partial x^i}= 0 & \stackrel{c\to \infty}{\longrightarrow} & \frac{\partial \mu}{\partial t}+ \frac{\partial \left( \mu v^i \right)}{\partial x^i} = 0  \, , \\
\kappa=i: \quad \frac{\partial T^{i0}}{\partial x^0} + \frac{\partial T^{ij}}{\partial x^j}= 0 & \stackrel{c\to \infty}{\longrightarrow} & \frac{\partial \left(\mu v^i\right)}{\partial t}+ \frac{\partial \left( \mu v^i v^j + S^{ij} \right)}{\partial x^j} = 0 \, .
\end{eqnarray*}
In order to obtain the energy equation, let us consider
\begin{equation*}
F^{\nu}= T^{0\nu}-c^2 I^{\nu} \quad , \, \textrm{where} \quad \nabla_{\nu}F^{\nu}=0 \, .
\end{equation*}
For the components of $F^{\nu}$ we have, using (\ref{sio}), i.e.~$S^{i0}=S^{ij}v^j/c$ and $S^{00}=S^{ij}v^i v^j / c^2$,
\begin{eqnarray*}
F^0 & = & \frac{\mu |v|^2 + \varepsilon}{1-|v|^2/c^2}+ \frac{S^{ij}v^i v^j}{c^2} - \frac{c^2 \mu}{\sqrt{1-|v|^2/c^2}} \, = \, \frac 1 2 \mu |v|^2 + \varepsilon + O(c^{-2}) \\
F^i & = &  \frac{\left(\mu |v|^2 + \varepsilon\right)v^i/c}{1-|v|^2/c^2}+ S^{0i}  - \frac{c \mu v^i}{\sqrt{1-|v|^2/c^2}} \\
 & = & \left(\mu |v|^2 + \varepsilon \right) \frac{v^i}{c} + S^{ij} \frac{v^j}{c} - \frac 1 2 \mu |v|^2 \frac{v^i}{c} + O(c^{-2}) \, .
\end{eqnarray*}
Hence,
\begin{equation*}
c \nabla_{\nu} F^{\nu} = \frac{\partial F^0}{\partial t}+ c \frac{\partial F^i}{\partial x^i}=0
\end{equation*}
is the energy equation
\begin{equation}
 \frac{\partial}{\partial t}\left(\frac 1 2 \mu |v|^2 + \varepsilon \right) +  \frac{\partial}{\partial x^i}\left[ \left(\frac 1 2 \mu |v|^2 + \varepsilon \right) v^i + S^{ij} v^j \right] = 0 \, .
\end{equation}
The non-relativistic Eulerian equations are
\begin{eqnarray*}
{\bf (a)}  &  & \frac{\partial \xi}{\partial t}+ \mathcal L_v \xi = 0 \, , \\
{\bf (b0)} &  & \frac{\partial \mu}{\partial t}+\frac{\partial}{\partial x^i} \left(\mu v^i\right) = 0 \, , \\
{\bf (b1)} &  & \frac{\partial \left(\mu v^i\right)}{\partial t}+ \frac{\partial}{\partial x^j} \left( \mu v^i v^j + S^{ij} \right) = 0 \, , \\
{\bf (b2)} &  & \frac{\partial}{\partial t}\left(\frac 1 2 \mu |v|^2 + \varepsilon \right) +  \frac{\partial}{\partial x^i}\left[ \left(\frac 1 2 \mu |v|^2 + \varepsilon \right) v^i + S^{ij} v^j \right] = 0 \, .
\end{eqnarray*}
Notice:
\begin{itemize}
\item ${\bf (b0)}$ is the differential mass conservation law, a consequence of ${\bf (a)}$,
\item ${\bf (b1)}$ is the differential momentum conservation law, {\it and}
\item ${\bf (b2)}$ is the differential energy conservation law.
\end{itemize}

\begin{remark}
{\bf (a)} may be thought of as expressing a law of conservation of dislocations. In fact, let $C_0$ be a closed curve on $\Sigma_0$. Let $\phi_t$ be the flow in Galilean space-time generated by the vectorfield
\begin{equation*}
u= \frac{\partial}{\partial t}+ v^i \frac{\partial}{\partial x^i} \, .
\end{equation*}
Then $\left. \phi_t\right|_{\Sigma_0}$ is a diffeomorphism from $\Sigma_0$ onto $\Sigma_t$. Let $C_t=\phi_t(C_0)$. Then, according to {\bf (a)},
\begin{equation*}
\int\limits_{C_t} \xi = \int\limits_{C_0} \xi \, .
\end{equation*}
The left-hand side corresponds, in the continuum limit, to minus the sum of the Burger's vectors of all the dislocation lines enclosed by $C_t$.
\end{remark}

\begin{remark}
For continuous solutions, i.e.~($\xi$, $v$, $s$) continuous, {\bf (b2)} is equivalent to the adiabatic condition
\begin{itemize}
\item[{\bf (c)}] $\frac{\partial s}{\partial t}+v^i \frac{\partial s}{\partial x^i}=0$, 
\end{itemize}
modulo the other equations. When discontinuities such as shocks develop, this equivalence no longer holds. However, the Eulerian equations {\bf (a)}-{\bf (b)} still hold, but in a weak or integral sense.
\end{remark}

\subsection{From the Eulerian to the Lagrangian Picture}
In order to go from the Eulerian picture to the Lagrangian formulation we have to extract the canonical form $\nu$ from $\xi$.
First note that 
\begin{equation}\label{liuxi}
\mathcal L_u \xi =0 \quad \Leftrightarrow \quad \phi_t^* \xi = \xi \, ,
\end{equation}
where $\phi_t$ is the flow generated by $u$.


Let $\mathcal H$ be a Cauchy hypersurface in $\mathcal M$. We identify $\mathcal H$ with $\mathcal N$ and define $f: \mathcal M \to \mathcal N$ as follows:
\begin{equation*}
f(x)=y \in \mathcal H
\end{equation*}
is the point at which the integral curve of $u$ through intersects $\mathcal H$. 

For $X_x \in T_x\mathcal M$ we have from (\ref{liuxi})
\begin{equation*}
\left(\phi^*\xi\right) = \xi\left(d\phi_t \cdot X_x\right)=\xi(X_x) \, .
\end{equation*}
Define $\nu$ to be the $\mathcal V$-valued $1$-form induced by $\xi$ on $\mathcal H$:
\begin{equation*}
\nu(X_x)=\xi(X_x) \quad : \, \forall X_x \in T_x {\mathcal H} , x \in \mathcal H \, .
\end{equation*}
We must show
\begin{proposition}
\begin{equation*}
f^* \nu = \xi \, .
\end{equation*}
\end{proposition}
\begin{proof}
Let $V_p \in T_p \mathcal M$, $p\in \phi_t(\mathcal H)$ for some $t \in \mathbb R$. $V_p$ uniquely decomposes into
\begin{equation*}
V_p= \lambda u + Y_p \quad , \, \textrm{where} \quad Y_p \in T_p \phi_t(\mathcal H) \, .
\end{equation*}
Therefore, it suffices to show that $(f^*\nu)(u_p)=\xi(u_p)=0$, which is obvious, and
\begin{equation*}
\left(f^*\nu\right)(Y_p)=\xi(Y_p) \, .
\end{equation*}
Since $\left.\phi_t \right|_{\mathcal H}$ is a diffeomorphism from $\mathcal H$ onto $\phi_t(\mathcal H)$, $d\phi_t(q)$ is an isomorphism from $T_q \mathcal H$ onto $T_p\phi_t(\mathcal H)$, $\phi_t(q)=p$. Therefore, there exists a unique $X_q \in T_q \mathcal H$, such that
\begin{equation*}
d\phi_t \cdot X_q = Y_p \, .
\end{equation*}
Now,
\begin{equation*}
\xi(Y_p)=\xi(d\phi_t \cdot X_q)=\left(\phi_t^* \xi\right)(X_q)=\xi(X_q)=\nu (X_q) 
\end{equation*}
and
\begin{equation*}
\left(f^* \nu\right) = \left(f^*\nu\right)(d\phi_t \cdot X_q) = \left(\phi_t^* f^*\nu \right)(X_q) = \left(f^* \nu \right)(X_q) = \nu (df \cdot X_q) \, .
\end{equation*}
But $\left. f \right|_{\mathcal H}= id$, hence $df \cdot X_q = X_q$ because
\begin{equation*}
\left. df \right|_{T_q\mathcal H}= id \, .
\end{equation*}
\end{proof}

\begin{remark}
The Eulerian picture is the one more related to direct physical experience and also the one which may serve as a basis for extending the theory beyond the domain of elasticity theory, where the dislocations are no longer anchored in the solid. We shall see in the next part, however, that the Lagrangian picture provides the suitable framework for the study of static problems.
\end{remark}

\newpage


\part{The Static Case} \label{static}
\setcounter{section}{0}



\section{Formulation of the Static Problem} \label{staticsetting}

Let $\mathcal N=\Omega \subset \mathbb R^n$ be the material manifold and $\mathcal M=E^n$ Euclidean space ($n=2,3$ being the cases of interest). 
In the static case, we adopt the material picture, that is, we consider one to one mappings $\phi$ from the material manifold into space,
\begin{eqnarray}\label{defphi} 
\begin{array}{rcl}
\phi : \mathcal N & \to & \mathcal M \\
 y & \mapsto & \phi(y)=x  \, .
\end{array}
\end{eqnarray}
They correspond to mappings $ f : \mathcal M \times \mathbb R \to \mathcal N$ such that
\begin{equation*} 
 f(x,t)=\phi^{-1}(x) \quad : \, \forall t \, .
\end{equation*}
In the material picture, the interchange of the roles of the domain and target space transforms a free boundary problem into a fixed boundary problem.

We recall that the \emph{configuration} $\gamma$ is an element of the space of inner product on the crystalline structure $\mathcal V$, $\gamma \in S_2^+(\mathcal V)$. In terms of the mapping $\phi$ it is defined as the pullback of the metric $g$ on $\mathcal M$ by the isomorphism $i_{\phi , y}$
\begin{equation} \label{defthermconfig}
\gamma = i_{\phi, y}^* g \, ,
\end{equation}
where $i_{\phi, y}=d\phi(y)\circ \epsilon_y$ and $\epsilon_y$ is the evaluation map (\ref{evalmap}) from $\mathcal V$ to $T_y\mathcal N$. 
The energy per unit mass $e(\gamma)$ defines the {\it thermodynamic stress} $\pi$, an element in $\left(S_2(\mathcal V)\right)^*$,  by
\begin{equation} \label{defthermstress}
\frac{\partial e}{\partial \gamma}=-\frac 1 2 \pi V \, ,
\end{equation}
where $V\equiv V(\gamma)$ is the volume per unit mass, related to the mass density $\mu$ by
\begin{equation*}
V(\gamma)=\frac 1{\mu(x)} \quad , \quad x=\phi(y) \, .
\end{equation*}
In the following, we assume the entropy $\sigma$ to be constant, what is called isentropic.

Since we are in the relativistic framework, $e$ includes the contribution $c^2$ of the rest mass-energy, so
\begin{equation*}
e=c^2+e'
\end{equation*}
(in conventional units), where $e'$ is what we would call energy per unit mass in the non-relativistic framework. Note that the additive constant $c^2$ does \emph{not} affect the definition (\ref{defthermstress}) of the thermodynamic stress $\pi$, which may equally well be written in the form
\begin{equation*} 
\frac{\partial e'}{\partial \gamma}=-\frac 1 2 \pi V \, .
\end{equation*}
Also, in view of (\ref{energy}) below, we have
\begin{equation*}
E=Mc^2+E' \, ,
\end{equation*}
where 
\begin{equation*}
M= \int_{\Omega} d\mu_{\omega} \quad , \quad E'=\int_{\Omega} e'(\gamma) d\mu_{\omega} \, .
\end{equation*}
The additive constant $Mc^2$ does not affect variations of $E$, which are the same as those of $E'$. 
For this reason we shall not distinguish in the static theory $e'$, $E'$ from $e$, $E$, and we shall denote the former by the latter.

{\remark The static theory would be formally identical if formulated wholly within the non-relativistic framework. What distinguishes the two theories is that the non-relativistic theory is applicable only when the characteristic speeds $\eta_i$ (\ref{defetai}) are negligible in comparison to $c$, while the relativistic theory holds for any values of the $\eta_i$ less than $c$.}

\section{Boundary Value Problem and Legendre-Hadamard Condition} \label{ELeqs}
\subsection{Euler-Lagrange Equations}
We choose a basis $E_1,\ldots ,E_n$ of $\mathcal V$ such that $\omega (E_1,\ldots ,E_n)=1$, where $\omega$ is the volume form on $\mathcal V$. 
We denote by $\gamma_{AB}=\gamma(E_A, E_B)$ the matrix representing the configuration in this frame.
We then have 
\begin{equation*}
\omega_{\gamma} = V(\gamma)\omega = \sqrt{\det \gamma} \cdot \omega \, ,
\end{equation*}
and for the total energy of a domain $\Omega$ in the material manifold
\begin{equation}\label{energy}
E=\int\limits_{\Omega} e(\gamma) d\mu_{\omega} \, ,
\end{equation}
where $d \mu_{\omega}$ is the mass element on $\mathcal N$ induced by $\omega$:
\begin{equation} \label{detomega}
d \mu_{\omega}\left(\left.\frac{\partial}{\partial y_1}\right|_y, \ldots , \left.\frac{\partial}{\partial y_n}\right|_y\right) = \omega \left( \epsilon_y^{-1}\left(\left.\frac{\partial}{\partial y_1}\right|_y\right), \ldots , \epsilon_y^{-1}\left(\left.\frac{\partial}{\partial y_n}\right|_y\right)\right)\, .
\end{equation}
Let $(\omega^A : A=1,\ldots,n)$ be the basis for $\mathcal V^*$ which is dual to the basis $(E_A: A=1,\ldots,n)$ for $\mathcal V$. The $\omega^A$ are $1$-forms on $\mathcal N$ such that $\omega^A(E_B)=\delta^A_B$. 
Since $(E_A(y): A=1,\ldots,n)$ is a basis for $T_y\mathcal N$, given any $Y \in T_y \mathcal N$, there are coefficients $Y^A : A=1,\ldots,n$ such that $Y=Y^A E_A(y)$. Then:
\begin{equation*}
\omega^A(Y) = \omega^A \left(Y^B E_B(y)\right)=Y^B(\omega^A(E_B))(y)= Y^B \delta^A_B = Y^A \, .
\end{equation*}
Setting $Y=\left. \frac{\partial}{\partial y^a}\right|_y$ we obtain
\begin{equation*}
Y^A=\omega^A \left(\left. \frac{\partial}{\partial y^a}\right|_y\right) =\omega_a^A(y) \, ,
\end{equation*}
and we have $\epsilon_y^{-1}(Y)= Y^A E_A \in \mathcal V$.
Thus
\begin{equation}\label{omegacoeff}
\epsilon_y^{-1}\left(\left.\frac{\partial}{\partial y^a}\right|_y\right)=\omega_a^A(y)E_A \, .
\end{equation}
Let us denote by $\det \omega (y)$ the determinant of the $n$-dimensional matrix with entries $\omega^A_a(y)$.
We conclude from (\ref{detomega}), (\ref{omegacoeff}) that
\begin{equation} \label{dmuomega}
d \mu_{\omega} = \det \omega (y) d^ny \, .
\end{equation}
The first variation of the energy (\ref{energy}) is:
\begin{equation}\label{firstvarE}
\dot E = \left.\frac{\partial}{\partial \lambda}\right|_{\lambda =0} E(\gamma+\lambda \dot\gamma)=\left.\frac{\partial}{\partial \lambda}\right|_{\lambda =0}\int\limits_{\Omega} e(\gamma+\lambda \dot\gamma) d\mu_{\omega}=\int\limits_{\Omega}\frac{\partial e(\gamma)}{\partial \gamma} \cdot \dot \gamma \,  d\mu_{\omega} \, ,
\end{equation} 
where $\dot \gamma \in S_2^+(\mathcal V)$ is the variation of the configuration $\gamma$.
By definition of the thermodynamic stress (\ref{defthermstress}), we conclude, in view of (\ref{dmuomega}),
\begin{equation*}
\dot E = - \int\limits_{\Omega} \frac 1 2 \pi^{AB}\dot \gamma_{AB} \sqrt{\det \gamma}\det \omega (y) \ d^ny \, . 
\end{equation*}
Denoting by $m_{ab}$ the pullback by $\phi$ to $\mathcal N$ of the Euclidean metric $g$ on $\mathcal M$, we have
\begin{eqnarray}
m & = & \phi^*g  \label{mpbphi} \\
m_{ab} & = &\frac{\partial x^i}{\partial y^a}\frac{\partial x^j}{\partial y^b} g_{ij} \qquad [x^i=\phi^i(y)] \, , \label{mabpbm} \\
\gamma_{AB} & = & E_A^aE_B^b m_{ab} \, , \label{gammapbm} 
\end{eqnarray}
where the $n$-dimensional matrix with entries $E_A^a(y)$ is the reciprocal of the $n$-dimensional matrix with entries $\omega_a^A(y)$. Thus $\det E(y)=\left( \det \omega(y)\right)^{-1}$, and by (\ref{gammapbm}) we have
\begin{eqnarray*}
\det \gamma & = & \det m \left(\det E\right)^2 \, , \\
\sqrt{\det \gamma} \det \omega & = & \sqrt{\det m} \, .
\end{eqnarray*}
For the stresses $S$ on $\mathcal N$ and $T=\phi_* S$ on $\mathcal M$ we have:
\begin{eqnarray}
S^{ab} & = & \pi^{AB} E_A^aE_B^b \, , \label{defSpfpi} \\
T^{ij} & = & S^{ab}\frac{\partial x^i}{\partial y^a}\frac{\partial x^j}{\partial y^b} \, .
\end{eqnarray}
We will now formulate the Euler-Lagrange equations on both the material manifold and on space. 
We calculate for the variation of $\gamma_{AB}$ from (\ref{gammapbm}) using linear coordinates in $E^n$ (i.e.~$\partial_k g_{ij}=0$)
\begin{equation*}
\dot \gamma_{AB}=E_A^a E_B^b \dot m_{ab} \quad , \quad \dot m_{ab} = g_{ij}\left(\frac{\partial \dot x^i}{\partial y^a}\frac{\partial x^j}{\partial y^b}+\frac{\partial x^i}{\partial y^a}\frac{\partial \dot x^j}{\partial y^b}\right) \, . 
\end{equation*}
The first variation of the energy thus reads 
\begin{equation} \label{dotE} 
\dot E =  -\int\limits_{\Omega} \frac 1 2 S^{ab}\dot m_{ab} \sqrt{\det m} \, d^ny = -\int\limits_{\Omega}  S^{ab}g_{ij}\frac{\partial \dot x^i}{\partial y^a}\frac{\partial x^j}{\partial y^b}\sqrt{\det m} \, d^ny  \, ,
\end{equation}
where $\dot x^i$ is the variation of $x^i$. Hence (\ref{dotE}) is
\begin{equation} \label{starS}
-\int\limits_{\Omega} S^{ab}g_{ij}\frac{\partial \dot x^i}{\partial y^a}\frac{\partial x^j}{\partial y^b} d\mu_m \, ,
\end{equation}
where $d\mu_m$ is the volume form on $\Omega$ associated to the metric $m$. By the divergence theorem, (\ref{starS}) becomes
\begin{equation} \label{dstarS}
-\int\limits_{\partial \Omega} S^{ab} g_{ij}\frac{\partial x^j}{\partial y^b} M_a \dot x^i\, dA_m + \int\limits_{\Omega} \dot x^i \overset{m}{\nabla}_a Q_i^a d\mu_m \, ,
\end{equation}
where $Q_i$ are the vectorfields with components
\begin{equation} \label{Qia}
Q_i^a = S^{ab}g_{ij}\frac{\partial x^j}{\partial y^b}
\end{equation}
and $\overset{m}{\nabla}$ is the covariant derivative operator on $\mathcal N$ associated to $m$. Thus, in local coordinates on $\mathcal N$,
\begin{equation} \label{covderQ}
\overset{m}{\nabla}_a Q_i^a = \frac{1}{\sqrt{\det m}} \frac{\partial}{\partial y^a} \left(\sqrt{\det m}Q_i^a\right) \, .
\end{equation}
In (\ref{dstarS}), $M^a$ are the components of the outward unit normal to $\partial \Omega$ with respect to $m$, $M_a=m_{ab}M^b$, and $dA_m$ is the area element of $\partial\Omega$ associated to the metric induced by $m$. So $M_a$ are the components of the covectorfield along $\partial \Omega$ whose null space is the tangent plane to $\partial \Omega$ at each point. 

The Euler-Lagrange equations are obtained by requiring $\dot E=0$ for variations $\dot x^i$ which are supported in $\Omega$ and vanish on $\partial \Omega$. In view of (\ref{dstarS}), (\ref{Qia}), (\ref{covderQ}), the Euler-Lagrange equations are:
\begin{equation} \label{equivEL}
\frac{\partial}{\partial y^a}\left(Q_i^a\sqrt{\det m}\right)=0 \quad \textrm{on} \, \, \Omega \, .
\end{equation}
Note that with $Q_i^a$ defined in (\ref{Qia}) we have using (\ref{mabpbm})
\begin{equation} \label{QS}
Q_i^a \frac{\partial x^i}{\partial y^b}= g_{ij}S^{ac} \frac{\partial x^j}{\partial y^c} \frac{\partial x^i}{\partial y^b} = m_{bc} S^{ac} = S_b^a \, .
\end{equation}
Consider
\begin{equation} \label{covdermS}
\overset{m}{\nabla}_a S_b^a = \frac{1}{\sqrt{\det m}} \frac{\partial}{\partial y^a} \left(\sqrt{\det m}S_b^a\right)-\overset{m}{\Gamma_{ab}^c} S_c^a \, .
\end{equation}
The first term on the right-hand side of (\ref{covdermS}) is, by (\ref{equivEL}), (\ref{QS}),
\begin{equation*} 
\frac{1}{\sqrt{\det m}}\frac{\partial}{\partial y^a} \left(\sqrt{\det m}Q_i^a \frac{\partial x^i}{\partial y^b}\right) = Q_i^a \frac{\partial^2 x^i}{\partial y^a \partial y^b}=  S^a_c \frac{\partial y^c}{\partial x^i} \frac{\partial^2 x^i}{\partial y^a \partial y^b} \, .
\end{equation*} 
Thus (\ref{equivEL}) is equivalent to:
\begin{equation} \label{nablaS}
\overset{m}{\nabla}S_b^a = \left( \frac{\partial y^c}{\partial x^i} \frac{\partial^2 x^i}{\partial y^a \partial y^b}-\overset{m}{\Gamma_{ab}^c} \right) S_c^a \, .
\end{equation}
We shall show now that
\begin{equation}\label{gammam}
\overset{m}{\Gamma_{ab}^c}=\frac{\partial y^c}{\partial x^i} \frac{\partial^2 x^i}{\partial y^a \partial y^b} \, ,
\end{equation}
so that the factor in parenthesis in (\ref{nablaS}) vanishes. In fact, (\ref{gammam}) is equivalent to:
\begin{equation*}
\overset{m}{\Gamma}_{ab,c}=m_{cd} \frac{\partial y^d}{\partial x^i} \frac{\partial^2 x^i}{\partial y^a \partial y^b} \, ,
\end{equation*}
which, since $m_{cd} \frac{\partial y^d}{\partial x^i}= g_{ij}\frac{\partial x^j}{\partial y^c}$, is in turn equivalent to 
\begin{equation} \label{gammaml}
\overset{m}{\Gamma}_{ab,c}=  g_{ij}\frac{\partial x^j}{\partial y^c} \frac{\partial^2 x^i}{\partial y^a \partial y^b} \, .
\end{equation}
Now we have by the definition of the Christoffel symbols $\overset{m}{\Gamma}_{ab,c}$ and $m$ from (\ref{mabpbm}):
\begin{equation*}
\overset{m}{\Gamma}_{ab,c} =  \frac 1 2 \left(\frac{\partial m_{ac}}{\partial y^b}+ \frac{\partial m_{bc}}{\partial y^a} - \frac{\partial m_{ab}}{\partial y^c}\right) = g_{ij}  \frac{\partial x^j}{\partial y^c} \frac{\partial^2 x^i}{\partial y^a \partial y^b}  \, .
\end{equation*}
Thus (\ref{gammaml}) is indeed verified. We conclude that the Euler-Lagrange equations (\ref{equivEL}) are equivalent to:
\begin{equation} \label{divstresszeroN}
\overset{m}{\nabla}_aS_b^a=0 \quad , \, \textrm{or} \quad \overset{m}{\nabla}_a S^{ab} = 0 \quad \textrm{on} \, \, \Omega \, .
\end{equation}
{\remark $T^{ij}$ being the push-forward of $S^{ab}$ by $\phi$ and $m_{ab}$ being the pull-back of $g_{ij}$ by $\phi$, the last equations (\ref{divstresszeroN}) are in turn equivalent to
\begin{equation} \label{divstresszeroE}
\overset{g}{\nabla}_i T^{ij} =0 \quad , \, \textrm{or} \quad \partial_i T^{ij} = 0  \quad \textrm{on} \, \, \phi(\Omega) 
\end{equation}
in linear coordinates.}

\subsection{Boundary Conditions}
Let now equations (\ref{equivEL}) hold. Requiring again $\dot E=0$, this time for arbitrary variations $\dot x$, we obtain from the first term of (\ref{dstarS})
\begin{equation} \label{bdrycondS}
0 = \int\limits_{\partial\Omega} S^{ab}g_{ij}\frac{\partial x^j}{\partial y^b}M_a \dot x^i \, dA_m \, .
\end{equation}
Here the covectorfield along $\partial \Omega$ with components $M_a$ is defined by
\begin{equation*}
M_aY^a=0 \quad : \, \forall Y \in T_y \partial \Omega \, ,
\end{equation*}
together with the condition $M_aY^a>0$ for all $Y \in T_y \Omega$, $y \in \partial \Omega$, such that $Y$ points to the exterior of $\Omega$, and the normalization condition $\left(m^{-1}\right)^{ab} M_a M_b = 1$.
As (\ref{bdrycondS}) is to hold for arbitrary variations $\dot x^i$, we obtain the boundary conditions
\begin{equation} \label{bdryc}
S^{ab}g_{ij}\frac{\partial x^j}{\partial y^b} M_a = 0  \quad , \, \textrm{or} \quad S^{ab}M_b=0 \quad \textrm{on} \, \,  \partial \Omega \, .
\end{equation}
Let  $X^i=\frac{\partial x^i}{\partial y^a}Y^a$. Setting $M_a=\frac{\partial x^i}{\partial y^a}N_i$ we have:
\begin{equation} \label{defNi}
N_i X^i =0 \quad : \, \forall X \in T_x \phi(\partial \Omega) \, . 
\end{equation}
In fact, the covectorfield along $\phi(\partial \Omega)$ with components $N_i$ is defined by (\ref{defNi}), together with the condition that 
\begin{equation*}
N_i X^i > 0 \quad: \, \forall X \in T_x \phi(\Omega) \, ,
\end{equation*}
$x \in \phi(\partial \Omega)=\partial (\phi(\Omega))$, such that $X$ points to the exterior of $\phi(\Omega)$, and the normalization condition $\left(g^{-1}\right)^{ij} N_i N_j = 1$.
In other words, $N^i:=(g^{-1})^{ij} N_j$ are the components of the outward unit normal to $\partial(\phi(\Omega))$ in $E^n$. 
In view of the fact that $S^{ab}\frac{\partial x^j}{\partial y^b}=T^{ij} \frac{\partial y^a}{\partial x^i}$, the boundary conditions (\ref{bdryc}) are equivalent to
\begin{equation}\label{bdrycE}
T^{ij}N_i = 0 \quad \textrm{on} \, \,\phi(\partial \Omega)=\partial \phi(\Omega) \, ,
\end{equation}
which means that no forces are acting on the boundary $\partial \phi(\Omega)$.

\subsection{Legendre-Hadamard Condition} \label{SectLegHad}

The \emph{Legendre-Hadamard} condition in the static case is just the second part of the hyperbolicity condition stated in (\ref{hposdef}), see \cite{C1} for details. 
Here, it is formulated in terms of $\frac{\partial \phi^i}{\partial y^a}(y)=v^i_a$.
It requires for $\xi_a, \xi_b \in T^*_y\mathcal N, \eta^i, \eta^j \in T_{\phi(y)}\mathcal M$ that
\begin{equation} \label{L-H}
\frac 1 4 \frac{\partial ^2 e}{\partial v^i_a \partial v^j_b} \xi_a \xi_b \eta^i \eta^j \quad : \quad \textrm{positive for}\,\, \xi, \eta \neq 0 \, .
\end{equation}
We differentiate $e$ with respect to $v^j_b$:
\begin{equation*}
\frac{\partial e}{\partial v^j_b}=\frac{\partial e}{\partial \gamma_{AB}}\frac{\partial\gamma_{AB}}{\partial v^j_b} \, , 
\end{equation*}
where $\gamma_{AB}=E_A^a E_B^b g_{ij} v^i_a v^j_b$, hence
\begin{equation*}
\frac{\partial\gamma_{AB}}{\partial v^j_b}=2E_{(A}^a E_{B)}^b g_{ij} v^i_a \, ,
\end{equation*}
($E_{(A}^a E_{B)}^b$ denotes symmetrization in $A$ and $B$) and thus
\begin{equation} \label{firstdere}
\frac 1 2 \frac{\partial e}{\partial v^j_b}=\frac{\partial e}{\partial \gamma_{AB}}E_A^a E_B^b g_{ij} v^i_a \, .
\end{equation}
We differentiate (\ref{firstdere}) once more with respect to $v^i_a$ to get
\begin{equation*}
\frac 1 4 \frac{\partial ^2 e}{\partial v^i_a \partial v^j_b} = \frac{\partial ^2 e}{\partial\gamma_{AB}\partial\gamma_{CD}}g_{li}g_{kj} v^l_c v^k_d E_A^d E_B^b E_C^c E_D^a + \frac 1 2 \frac{\partial e}{\partial\gamma_{AB}}E_A^a E_B^b g_{ij}  \, .
\end{equation*}
Therefore,
\begin{equation}
\frac 1 4 \frac{\partial ^2 e}{\partial v^i_a \partial v^j_b} \xi_a \xi_b \eta^i \eta ^j = \frac 1 2 \frac{\partial e}{\partial\gamma_{AB}}|\eta|^2 \xi_A \xi_B + \frac{\partial ^2 e}{\partial\gamma_{AB}\partial\gamma_{CD}} \eta_C \eta_A \xi_B \xi_D \, , \label{LegHadgen}
\end{equation}
where $|\eta|^2=g_{ij}\eta^i\eta^j$, $\xi_A=E_A^a\xi_a$ and $\eta_C= E_C^c v^l_c g_{li} \eta^i$. 
Thus the Legendre-Hadamard condition reads:
\begin{equation*}
\frac{\partial^2 e}{\partial\gamma_{AB}\partial\gamma_{CD}} \eta_C \eta_A \xi_B \xi_D + \frac 1 2 \frac{\partial e}{\partial\gamma_{AB}}|\eta|^2 \xi_A \xi_B > 0 \quad (\eta, \xi \neq 0) \, .
\end{equation*}
At $\stackrel{\circ}{\gamma}_{AB}$ the Legendre-Hadamard condition reduces to
\begin{equation*}
 \frac{\partial ^2 e}{\partial \gamma_{AB} \partial \gamma_{CD}} \eta_C \eta_A \xi_B \xi_D > 0  \qquad (\eta ,\, \xi \neq 0) \, ,
\end{equation*}
because the first term on the right-hand side of (\ref{LegHadgen}) (containing $\frac{\partial e}{\partial \gamma_{AB}}$) vanishes.
Note that this condition is satisfied by virtue of the Postulate \ref{pmine} that $e$ has a strict minimum at $\stackrel{\circ}{\gamma}$, since this implies
\begin{equation*}
\frac{\partial ^2 e}{\partial \gamma_{AB} \partial \gamma_{CD}} (\stackrel{\circ}{\gamma}) \sigma_{AB} \sigma_{CD} > 0
\end{equation*}
for any non-zero symmetric $\sigma_{AB}$ and, in particular, for
\begin{equation*}
\sigma_{AB}=\xi_A\eta_B+\eta_A\xi_B \, .
\end{equation*} 
In the general case, we define $\eta^A$ by $\eta^A E_A^a v^i_a = \eta^i$ and we have
\begin{equation*}
|\eta|^2=g_{ij}\eta^i \eta^j=g_{ij} v^i_a v^j_b E_A^a E_B^b \eta^A \eta^B = \gamma_{AB}\eta^A \eta^B \, .
\end{equation*}
Additionally,
\begin{equation*}
\gamma_{AB} \eta^B= g_{ij} v^i_a v^j_b E_A^a E_B^b \eta^B= g_{ij} v^i_a E_A^a \eta^j = \eta_A \, ,
\end{equation*}
and hence $\eta^A=\left(\gamma^{-1}\right)^{AB}\eta_B$.
We finally have for the inverse $\left(\gamma^{-1}\right)^{AB}$ of $\gamma_{AB}$ with $|\eta|^2=\eta_A\eta^A= \eta_A\left(\gamma^{-1}\right)^{AB}\eta_B$ 
\begin{equation*}
|\eta|^2=\left(\gamma^{-1}\right)^{AB}\eta_A\eta_B \, .
\end{equation*}

\section{Uniform Distribution of Edge Dislocations in $2$d}

The case of a uniform distribution of 2-dimensional \emph{edge dislocations} is realized by the material manifold $\mathcal N$ being the affine group, see \ref{edag}.
The corresponding metric $\stackrel{\circ}{n}$ from (\ref{hypmetafgr}) is isometric to the metric of the hyperbolic plane.

In general, we consider the total energy $E$ (\ref{energy}) of a domain $\Omega$ in the material manifold $\mathcal N$.
In agreement with Postulate \ref{pmine}, we assume that the energy per unit mass $e$ has a strict minimum at $\stackrel{\circ}{\gamma}$.
The symmetry of the problem motivates the choice of an isotropic energy function $e$ that is a symmetric function of the eigenvalues of $m=\phi^* g$ relative to $\stackrel{\circ}{n}$.
Therefore, in this model case, the crystalline structure is eliminated in favor of the Riemannian manifold $(\mathcal N, \stackrel{\circ}{n})$. 

\subsection{Isotropic Energy Function}\label{ToyNrgS}
Concerning our toy energy function, we make the following basic choice:
\begin{equation} \label{toynrg}
e(\gamma)=e(\lambda_1, \ldots , \lambda_n)= \frac 1 2 \sum_{k=1}^n \left(\lambda_k -1\right)^2 \geq 0 
\end{equation}
that satisfies $e(\lambda_1=1, \ldots , \lambda_n=1)=0$. Thus we have a strict minimum of the energy density at $\gamma= \, \stackrel{\circ}{\gamma}$. Note that, in (\ref{toynrg}), we are subtracting the rest mass energy, something which does not affect the variation of $E$, see also the discussion at the end of Section \ref{staticsetting}.

Therefore,
\begin{equation} \label{ToyEnergy}
e(\gamma)=\frac 1 2 \sum_{k=1}^n \left(\lambda_k^2-2\lambda_k+1\right)=\frac 1 2 \mathrm{tr}\, \gamma^2 -\mathrm{tr}\, \gamma +\frac n 2 =\frac 1 2 \sum_{A,B=1}^n \gamma_{AB}^2- \sum_{A=1}^n \gamma_{AA} +\frac n 2 \, . 
\end{equation}
Here and in the following the basis $(E_A: A=1, \ldots, n)$ for $\mathcal V$ is chosen to be orthonormal relative to $\stackrel{\circ}{\gamma}$, hence $\stackrel{\circ}{\gamma}_{AB}=\delta_{AB}$. For such a choice to be compatible with the condition $\omega(E_1, \ldots, E_n)=1$, we must choose the physical unit of mass so that the mass density corresponding to $\stackrel{\circ}{\gamma}$ is equal to $1$.

In case that the metric of the target space is Euclidean, $g_{ij}=\delta_{ij}$, we find from (\ref{gammapbm})
\begin{equation}\label{gammacomp}
\gamma_{AB}= E_A^a E_B^b \frac{\partial x^i}{\partial y^a}\frac{\partial x^j}{\partial y^b}\delta_{ij}=E_A^a E_B^b \frac{\partial \phi}{\partial y^a}\cdot\frac{\partial \phi}{\partial y^b} \, ,
\end{equation}
where the dot denotes the Euclidean inner product.
For the total energy $E$ (\ref{energy}) of a domain $\Omega \subset \subset \mathcal N$, where $\omega$ coincides with the volume form on $\mathcal V$ induced by $\stackrel{\circ}{\gamma}$, we have $d\mu_{\omega}=d\mu_{\stackrel{\circ}{n}}$ (the volume form on $\mathcal N$ corresponding to the Riemannian metric $\stackrel{\circ}{n}$), and we obtain for the total energy $E$ expressed in terms of the components of $\gamma$,
\begin{equation} \label{nrgcomp}
E=\int\limits_{\Omega}\left( \frac 1 2 \sum_{A,B=1}^n \gamma_{AB}^2 -\sum_{A=1}^n \gamma_{AA} + \frac n 2 \right) d\mu_{\stackrel{\circ}{n}} \, .
\end{equation}

In the case $n=2$, the invariants of a linear mapping are, respectively, the trace and the determinant. Expressed in terms of the eigenvalues $\lambda_1, \lambda_2$, they read:
\begin{eqnarray*}
\tau & = & \lambda_1 + \lambda_2= \textrm{tr}\, m \, , \\
\delta & = & \lambda_1 \lambda_2 = \det m \, .
\end{eqnarray*}
We have:
\begin{equation*}
\tau^2=\lambda_1^2 + \lambda_2^2 + 2 \delta  \, .
\end{equation*}
The toy energy (\ref{toynrg}) thus reads
\begin{equation*}
e(\lambda_1, \lambda_2)=\frac 1 2 \left( (\lambda_1 -1)^2+ (\lambda_1 -1)^2\right) = \frac 1 2 \left( \tau^2  -2 \tau - 2 \delta + 2\right) \, .
\end{equation*} 

\subsection{Expansion of the Hyperbolic Metric} \label{HyperbolicExpansion}
We shall perform a dilation of the hyperbolic plane in order to be able to establish an existence result for a fixed domain, from which we can deduce an analogous result for energy minimizing mappings from a suitably small domain in the original hyperbolic plane to the Euclidean plane.

We start with the usual metric, of constant curvature $-1$ for the hyperbolic plane $H$, expressed in polar normal coordinates $(R, \theta)$:
\begin{equation*}
ds_{H}^2 = dR^2 + \sinh^2 R d\theta^2 \, .
\end{equation*}
Given a parameter $\varepsilon \in \mathbb R^+$, we dilate $H$ by a factor $\varepsilon^{-1}$ obtaining $H_{\varepsilon}$. The corresponding metric, of curvature $-\varepsilon^2$, is:
\begin{equation*}
ds_{H_{\varepsilon}}^2 = \varepsilon^{-2} \left( dR^2 + \sinh^2 R d\theta^2 \right) =dr^2 + F_{\varepsilon}^2(r) d\theta^2 \, ,
\end{equation*}
where $r=\varepsilon^{-1}R$ and 
\begin{equation*}
F_{\varepsilon}(r)= \varepsilon^{-1} \sinh \left(\varepsilon r \right)=\varepsilon^{-1}\left(\varepsilon r + \frac 1 {3!}\left(\varepsilon r\right)^3 +\frac 1 {5!}\left(\varepsilon r\right)^5+  \ldots \right) \, .
\end{equation*}
$(r, \theta)$ are polar normal coordinates in $H_{\varepsilon}$. The expression for $F_{\epsilon}^2(r)$ thus reads:
\begin{equation} \label{epsilonexp}
F_{\varepsilon}^2(r)=r^2\left(1+\frac 1 3 (\varepsilon r)^2 + \frac 2 {45} (\varepsilon r)^4 + \ldots \right) \, .
\end{equation}
The main point in the expansion (\ref{epsilonexp}) is that $F_{\varepsilon}^2(r)$ is a real analytic function converging to $r^2$ for $\varepsilon \to 0$, corresponding to the transition of the negatively curved hyperbolic plane to the flat Euclidean plane.

In order to simplify the calculations, we go back to rectangular coordinates.
We denote rectangular normal coordinates in $H_{\varepsilon}$ by $(y^i: i=1,2)$
\begin{eqnarray*}
y^1 & = & r \cos \theta \, , \\
y^2 & = & r \sin \theta \, .
\end{eqnarray*}
Then
\begin{equation*}
ds_{H_{\varepsilon}}^2 = \left(dy^1\right)^2 + \left(dy^2\right)^2 + \varepsilon^2\left(\frac 1 3 + \frac 2{45}(\varepsilon r)^2+ \ldots \right) r^4 d\theta^2 \, .
\end{equation*}
Since
\begin{equation*}
d\theta = \frac 1{r^2}\left(-y^2 dy^1+ y^1 dy^2 \right) \, ,
\end{equation*}
denoting by $ds_E^2$ the Euclidean metric
\begin{equation*}
ds_E^2= \left(dy^1\right)^2 + \left(dy^2\right)^2 \, ,
\end{equation*}
the metric of $H_{\varepsilon}$ takes the form:
\begin{equation} \label{hypmetrectcoords}
ds_{H_{\varepsilon}}^2 = ds_E^2 + \varepsilon^2 \left(\frac 1 3 + \frac 2 {45} (\varepsilon r)^2+ \ldots \right)\left(-y^2 dy^1 + y^1 dy^2 \right)^2 \, .
\end{equation}

The crucial point in the above expression for the metric $h$ of $H_{\varepsilon}$ 
is that it takes the form 
\begin{equation} \label{hypmetexpeps}
h_{ab}=\delta_{ab}+ \varepsilon^2 f_{ab} \, ,
\end{equation}
where $f_{ab}$ are analytic functions in $(y^1,y^2)$ and $\varepsilon^2$. In fact,
\begin{equation} \label{rectmeth}
\left(-y^2 dy^1+y^1 dy^2 \right)^2 = \left(r^2 \delta_{ab}-y^a y^b \right)dy^a dy^b \, ,
\end{equation}
hence, by (\ref{hypmetrectcoords}), (\ref{hypmetexpeps}) and (\ref{rectmeth}),
\begin{equation}
f_{ab}=\left(\frac 1 3 + \frac 2 {45} (\varepsilon r)^2+ \ldots \right)\left(r^2 \delta_{ab} - y^a y^b \right) \, .
\end{equation}
That is,
\begin{equation*}
f_{ab}= f(\varepsilon^2 r^2)\left(r^2 \delta_{ab} - y^a y^b \right) \quad , \, \textrm{where} \quad f(x)= \frac 1 3 + \frac 2 {45} x + \ldots  \, .
\end{equation*}

\section{Uniform Distribution of Screw Dislocations in $3$d} \label{udsd}
The case of a uniform distribution of (3-dimensional) \emph{screw dislocations} is realized when the material manifold $\mathcal N$ is the Heisenberg group, see \ref{sdhg}.
The corresponding metric $\stackrel{\circ}{n}$ from (\ref{homspacemet}) is isometric to the metric of a homogeneous but anisotropic space.

\subsection{Anisotropic Energy Function}
The symmetry of the problem in this case motivates the choice of an anisotropic energy function $e$, since the dislocation lines have a preferred direction.
Therefore, in this model case, the crystalline structure cannot be eliminated in favor of the Riemannian manifold $(\mathcal N, \stackrel{\circ}{n})$, in contrast to the $2$-dimensional case previously discussed.  

The question of the right choice of energy per unit mass for the Heisenberg group is investigated in the following.
Consider the orthogonal transformation $O$ of $\mathcal V$,
given by $(X',Y',Z')=O(X,Y,Z)$ with
\begin{eqnarray}
X' & = & \cos \theta \cdot X + \sin \theta \cdot Y \, , \nonumber \\
Y' & = & -\sin \theta \cdot X + \cos \theta \cdot Y \, , \label{orthtrsf} \\
Z' & = & Z \, . \nonumber
\end{eqnarray}
The commutation relations $[X,Y]=Z$, $[X,Z]=[Y,Z]=0$ are preserved by this transformation since
\begin{eqnarray*}
\left[X',Y'\right] & = & \cos^2 \theta [X,Y] - \sin^2 \theta [Y,X]=\left(\cos^2 \theta +\sin^2 \theta\right) [X,Y]=Z=Z' \, ,  \\ 
\left[X',Z'\right] & = & [Y',Z'] = 0 \, .
\end{eqnarray*}
Consider the metric $m=\phi^* g$. Since $m$ is symmetric, we have generally six independent components in three dimensions. This also holds for $\gamma$, the corresponding inner product on the crystalline structure:
\begin{equation*}
\gamma = \left( \begin{array}{cc} \bar{\gamma}_{AB} & \theta_A \\ \theta_A^T & \rho \end{array} \right) \, ,
\end{equation*}
where the six components are 
\begin{eqnarray*}
\bar{\gamma}_{AB} & = & \gamma_{AB} \qquad A,B=1,2  \, ,\\
\theta_A & = & \gamma_{A3} \qquad A=1,2 \, , \\
\rho & = & \gamma_{33} \, .
\end{eqnarray*}
The energy density $e$ depends on the variables $(\bar{\gamma}, \theta, \rho)$. $e$ must be invariant under orthogonal transformations (\ref{orthtrsf}), that is $e(\bar{\gamma}, \theta, \rho)= e (O \bar{\gamma} \widetilde O, O \theta, \rho)$, for any $O$ given by (\ref{orthtrsf}). 
$|\theta|^2$ is such an invariant; together with the invariants $\textrm{tr}\,\bar{\gamma}$ and $\textrm{tr}\,\bar{\gamma}^2$, which are the same as in the two-dimensional case, we have the following four invariants:
\begin{equation} \label{invariants3d}
\textrm{tr}\,\bar{\gamma} \, , \, \textrm{tr}\,\bar{\gamma}^2 \, ,\, \abs{\theta}^2 \, , \,  \rho \, .
\end{equation}
Thus, starting with six variables, we have eliminated two and are left with an energy per unit mass of the form
\begin{equation*}
e=e(\mu_1, \mu_2, \abs{\theta}^2, \rho) \, ,
\end{equation*}
where $\mu_{1,2}$ are the eigenvalues of $\bar{\gamma}_{AB}$ with respect to $\overset{\circ}{\bar{\gamma}}_{AB}$ in the $XY$-plane ($A,B=1,2$).

We may choose in particular: 
\begin{equation}\label{AnisotropicEnergy}
e=\frac 1 2 \left( (\mu_1 -1)^2 + (\mu_2 -1)^2\right) +\frac{\alpha}{2} \abs{\theta}^2+ \frac{\beta}{2} (\rho-1)^2 \, ,
\end{equation}
which has a minimum at the identity $\gamma_{AB}=\stackrel{\circ}{\gamma}_{AB}=\delta_{AB}$.
It is not surprising that the screw dislocations, which have a distinguished direction, break the isotropy and we are lead to an anisotropic energy of the form (\ref{AnisotropicEnergy}), which is invariant under the transformation (\ref{orthtrsf}).

\subsection{The Heisenberg Group as a Homogeneous Space} \label{HGahs}
Let 
\begin{equation} \label{vectfHeis}
X=\frac{\partial}{\partial x} \quad , \quad Y=\frac{\partial}{\partial y}+x \frac{\partial}{\partial z} \quad , \quad  Z=\frac{\partial}{\partial z}  
\end{equation}
be, as in Section \ref{udsd}, the basis for the Lie algebra $\mathcal V$ of the Heisenberg group. We have the commutation relations
\begin{equation} \label{commutationHeisenberg}
[X, Y]= Z \quad , \quad  [X, Z]=[Y, Z]=0 \, .
\end{equation}
The basis of $1$-forms $(\xi, \zeta, \eta)$ dual to $X, Y, Z$ is
\begin{equation*}
\xi=dx \quad , \quad \zeta = dy \quad , \quad \eta= dz - x dy \, .
\end{equation*}
A left-invariant metric on the Heisenberg group is, up to isometries and an overall scale factor, given by:
\begin{equation} \label{HGmetrit}
\stackrel{\circ}{n} = \xi \otimes \xi + \zeta \otimes \zeta + e^{2\beta} \eta \otimes \eta \, ,
\end{equation}
that is:
\begin{equation} \label{Heisenbergmetric}
ds^2= dx^2 + dy^2 + e^{2\beta}(dz- x dy)^2  \, ,
\end{equation}
where $\beta$ is a real constant.
The metric (\ref{Heisenbergmetric}) is called a Bianchi type VII metric (for the classification into Bianchi types, see \cite{LL2}).
It represents a homogeneous space which is not isotropic.

The inner product $\stackrel{\circ}{\gamma}$ on $\mathcal V$ giving rise to the metric $\stackrel{\circ}{n}$ is given in the basis $(X,Y,Z)$ by:
\begin{eqnarray*}
\stackrel{\circ}{\gamma}(X,X) & = & \stackrel{\circ}{\gamma}(Y,Y) \quad , \quad \stackrel{\circ}{\gamma}(Z,Z)=e^{2\beta} \, ,  \\
\stackrel{\circ}{\gamma}(X,Y) & = & \stackrel{\circ}{\gamma}(X,Z) = \stackrel{\circ}{\gamma} (Y,Z) = 0 \, . 
\end{eqnarray*}
Thus defining 
\begin{equation} \label{onbHG}
E_1=X \quad , \quad E_2 = Y \quad , \quad E_3=e^{-\beta} Z \, ,
\end{equation}
$(E_A: A=1,2,3)$ is an orthonormal basis for $\stackrel{\circ}{\gamma}$:
\begin{equation} \label{minmetHG}
\stackrel{\circ}{\gamma}_{AB} := \stackrel{\circ}{\gamma}(E_A,E_B) = \delta_{AB} \, .
\end{equation}
The dual basis is:
\begin{equation}\label{dualbHG}
\omega^1 = \xi \quad , \quad \omega^2 = \zeta \quad , \quad \omega^3 = e^{\beta} \eta \, .
\end{equation}
In the following we shall denote the metric $\stackrel{\circ}{n}$ by $h$. Thus, in terms of the basis $(\omega^A: A=1,2,3)$ we have:
\begin{equation*}
h = \sum_A \omega^A \otimes \omega^A \, .
\end{equation*}

Now $(\mathcal N, h)$ being a homogeneous space, any point in $\mathcal N$ may be taken as the origin. Consider a local coordinate system $(y^a: a=1,2,3)$ on $\mathcal N$ with origin a given point. Since in the Euclidean space $E^3$ we may set up a rectangular coordinate system $(x^i: i=1,2,3)$, a mapping $id$ is then defined in the domain of the coordinate system on $\mathcal N$ by:
\begin{equation*}
id(y^1,y^2,y^3)=(y^1,y^2,y^3) \in E^3 \, ,
\end{equation*}
that is, $id$ is expressed in the respective coordinate systems as the identity mapping. Thus $m$, the pullback to $\mathcal N$ by $id$ of the Euclidean metric $\delta_{ij} dx^i \otimes dx^j$ is 
\begin{equation*}
m = \delta_{ab} dy^a \otimes dy^b 
\end{equation*} 
in the local coordinates $(y^a: a=1,2,3)$ on $\mathcal N$. The following problem then arises in connection with the analytic method to be presented below: given any point in $\mathcal N$, find a local coordinate system $(y^a: a=1,2,3)$ on $\mathcal N$ with origin the given point such that $\gamma_{AB}$, the components at $y$ of the corresponding inner product on $\mathcal V$ in the basis $(E_A: A=1,2,3)$, that is:
\begin{equation*}
\gamma_{AB}(y) = m_{ab}(y) E_A^a(y) E_B^b(y) = \delta_{ab} E_A^a(y) E_B^b(y) 
\end{equation*}
satisfy:
\begin{equation*}
\gamma_{AB}(y)-\stackrel{\circ}{\gamma}_{AB} = O(|y|^2) \, .
\end{equation*}
(Recall that $\stackrel{\circ}{\gamma}_{AB}=\delta_{AB}$.) The solution is simply to set up Riemannian normal coordinates on $(\mathcal N,h)$ with origin the given point. For, the components of $h$ in such coordinates satisfy:
\begin{equation*}
h_{ab}(y)=\delta_{ab}+ O(|y|^2) \, .
\end{equation*}
Since
\begin{equation*}
h_{ab}(y) E_A^a(y) E_B^b(y) = \delta_{AB}
\end{equation*}
identically, then
\begin{equation*}
\gamma_{AB}(y)-\delta_{AB}= \left(\delta_{ab}-h_{ab}(y)\right) E_A^a(y) E_B^b(y) = O(|y|^2) \, ,
\end{equation*}
as required.

\section{Scaling Properties} \label{scaling}

\subsection{The General Isotropic Case}
Let $\mathcal N=\Omega \subset \mathbb R^n$ and $\phi(\Omega) \subset E^n$ be the domain and the target space, respectively, of the mapping $\phi$ with the same properties as in (\ref{defphi}). We fix coordinates in $\mathcal N$ and also fix an origin in $E^n$.
We may assume that the coordinate origin in $\mathcal N$ is mapped by $\phi$ to the origin in $E^n$. Consider
\begin{equation*}
\tilde \phi(y)=l \phi \left(\frac y l \right) \, ,
\end{equation*}
where $l>0$ is a given positive constant.
If $U$ is a domain in $E^n$, we denote by $lU$ the domain $\{ lx : x \in U\}$. Similarly for a domain in $\mathbb R^n$. The domain of $\tilde \phi$ is $\tilde \Omega= l \Omega$, and
\begin{equation*}
\tilde \phi : \tilde \Omega = l \Omega \to \tilde \phi (\tilde \Omega) =l \phi(\Omega) \, .
\end{equation*}
The change from $\phi$ to $\tilde \phi$ induces a change from $m$ to $\tilde m$ (and similarly for the inverses) as follows
\begin{equation}\label{scalemetric}
\tilde m_{ab} (y) = \sum_{i=1}^n \frac{\partial \tilde \phi^i(y)}{\partial y^a}\frac{\partial \tilde \phi^i(y)}{\partial y^b}= m_{ab}\left(\frac y l\right) \quad , \quad \left(\tilde m^{-1}\right)^{ab}(y)=\left(m^{-1}\right)^{ab}\left(\frac y l \right) \, .
\end{equation}
For, since $\tilde \phi^i(y)=l \phi^i \left( \frac y l \right)$, we have
\begin{equation} \label{scalephi}
\frac{\partial \tilde \phi^i(y)}{\partial y^a}=l\frac{\partial \phi^i(y/l)}{\partial y^a}\frac 1 l = \frac{\partial \phi^i(y/l)}{\partial y^a} \, .
\end{equation}
We consider here the case of an isotropic energy function. We have seen that in this case the crystalline structure $\mathcal V$ can be eliminated in favor of a Riemannian metric $h$.

Let us define a new metric $\tilde h$ on $\mathcal N$ by $\tilde h_{ab}(y)=h_{ab}\left(\frac y l \right)$.
Note that $\tilde h$ is \emph{not} isometric to $h$. In the case of the toy energy function (\ref{ToyEnergy}), we have (see (\ref{stressonN}) below),
\begin{equation*}
\sqrt{\frac{\det \tilde m}{\det \tilde h}}\tilde S^{ab}= 2 \left(\tilde h^{-1}\right)^{ac} \left(\tilde h^{-1}\right)^{bd}\left(\tilde h_{cd}-\tilde m_{cd}\right) \, ,
\end{equation*}
and similarly with $\tilde h$, $\tilde m$, $\tilde S$ replaced by $h$, $m$, $S$.
Using (\ref{scalemetric}) we then obtain:
\begin{equation} \label{scalestress}
\tilde S^{ab}(y) =S^{ab}\left(\frac y l \right) \, .
\end{equation}
Now the equations (\ref{divstresszeroN}) transform as follows. From the definition of the covariant derivative, we have
\begin{equation} \label{defcovderS}
\overset{m}{\nabla}_b S^{ab}=\frac{\partial S^{ab}}{\partial y^b} + \overset{m}{\Gamma_{bc}^a} S^{bc} + \overset{m}{\Gamma_{ac}^b} S^{ac} \, ,
\end{equation}
where $\overset{m}{\Gamma_{bc}^a}$ are the Christoffel symbols with respect to $m$ defined by
\begin{equation} \label{defChristoffel}
\overset{m}{\Gamma_{bc}^a}=\frac 1 2 (m^{-1})^{ad} \left(\frac{\partial m_{bd}}{\partial y^c} + \frac{\partial m_{cd}}{\partial y^b}-\frac{\partial m_{bc}}{\partial y^d}  \right) \, ,
\end{equation}
as well as analogous expressions for $\overset{\tilde m}{\nabla}_b \tilde S^{ab}$ and $\overset{\tilde m}{\Gamma_{bc}^a}$.
From (\ref{scalemetric}) and (\ref{scalestress}) we then obtain
\begin{equation} \label{derscalemandS}
\frac{\partial \tilde m_{bc}}{\partial y^d}(y)=\frac 1 l \frac{\partial m_{bc}}{\partial y^d} \left(\frac y l \right)  \qquad  \textrm{and} \qquad \frac{\partial \tilde S^{ab}}{\partial y^b}(y)=\frac 1 l \frac{\partial S^{ab}}{\partial y^b}\left(\frac y l \right) \, ,
\end{equation}
and, consequently, from (\ref{defChristoffel}),
\begin{equation} \label{scaleChristoffel}
\overset{\tilde m}{\Gamma_{bc}^a} (y)= \frac 1 l \overset{m}{\Gamma_{bc}^a} \left(\frac y l \right) \, .
\end{equation}
We conclude, using (\ref{defcovderS}) with (\ref{derscalemandS}) and (\ref{scaleChristoffel}),
\begin{equation*}
\overset{\tilde m}{\nabla}_b \tilde S^{ab}(y)=\frac 1 l \left(\overset{m}{\nabla}_b S^{ab} \right) \left(\frac y l \right) \, .
\end{equation*}
So, if $\phi$ is a solution relative to $h$ and $\Omega$, $\tilde \phi$ is a solution relative to $\tilde h$ and $\tilde \Omega$, the tangent plane to $\tilde \Omega = l\Omega$ at the point $l y \in \partial \tilde \Omega$ being parallel to the tangent plane to $\partial \Omega$ at $y$ relative to the linear structure of $\mathbb R^n \supset \Omega$.

If $(E_A: A=1,\ldots,n)$ is an orthonormal frame field for $h$ then $(\tilde E_A: A=1,\ldots,n)$, with $\tilde E_A$ defined by
\begin{equation*}
\tilde E_A^a(y)=E_A^a\left(\frac y l\right) \, ,
\end{equation*}
is an orthonormal frame field for $\tilde h$. Also, using a formula similar to (\ref{scaleChristoffel}) with $h$ in the role of $m$, we deduce
\begin{equation*}
\tilde R^a_{bcd}(y)= l^{-2} R^a_{bcd}\left(\frac y l \right) \, .
\end{equation*}
It follows that, if $K_{\Pi}$ is the sectional curvature of $h$ corresponding to the plane $\Pi$ at $y$ and $\tilde K_{\tilde \Pi}$ is the sectional curvature of $\tilde h$ corresponding to the plane $\tilde \Pi$ at $\tilde y=S_l y=ly$, where $\tilde \Pi=dS_l(\Pi)$ ($S_l$ the scaling map), then
\begin{equation*}
\tilde K_{\tilde \Pi} = l^{-2} K_{\Pi} \, .
\end{equation*} 

\subsection{The $2$d Case}
Consider the two-dimensional case. The metric of the hyperbolic plane of curvature $-\varepsilon^2$ is given in Riemannian normal coordinates by: 
\begin{eqnarray*}
h_{ab} & = & \delta_{ab}+\varepsilon^2 f_{ab}(y) \, , \\
f_{ab}(y) & = & f(\varepsilon^2 r^2) l_{ab}(y) \, ,\\
l_{ab} (y)  & = & \delta_{ab} r^2 -y^a y^b \, ,
\end{eqnarray*}
and $f(z)$ is an entire function with $f(0)=\frac 1 3$.

We now consider the metric $\tilde h$, the components of which in the above coordinates are: $\tilde h_{ab}(y)=h_{ab}(y/l)$. We have:
\begin{eqnarray*}
\tilde h_{ab}(y) & = & \delta_{ab}+\varepsilon^2 f_{ab}\left(\frac y l \right) \, ,\\
f_{ab} \left(\frac y l \right) & = & f\left(\varepsilon^2 \frac{r^2}{l^2}\right) \frac{1}{l^2} l_{ab}(y)  \, ,
\end{eqnarray*}
therefore,
\begin{equation*}
\tilde h_{ab}(y) = \delta_{ab}+ \tilde \varepsilon^2 f \left( \tilde \varepsilon^2 r^2\right) l_{ab}(y)=\delta_{ab}+ \tilde \varepsilon^2 \tilde f_{ab}(y) \, ,
\end{equation*}
where $\tilde \varepsilon = \varepsilon / l$ and $\tilde f_{ab}(y)$ is $f_{ab}(y)$ with $\varepsilon$ replaced by $\tilde \varepsilon$.
While the curvature $K$ of $h$ is $-\varepsilon^2$, the curvature $\tilde K$ of $\tilde h$ is
\begin{equation*}
-\tilde \varepsilon^2 = -\frac{\varepsilon^2}{l^2} \, . 
\end{equation*}

Given a domain $\Omega_1$, we shall show in Part III below that there is a $\varepsilon_1>0$ such that we can solve equations (\ref{divstresszeroN}) for all $0< \varepsilon < \varepsilon_1$.
Let $\phi_{1,\varepsilon}$ be the solution corresponding to $\Omega_1$ and $\varepsilon \in (0,\varepsilon_1)$. $\Omega_1$ is a domain in the hyperbolic plane of curvature $-\varepsilon^2$. We define $\phi_{l,\varepsilon/l}$ by
\begin{equation} \label{defphiscale}
\phi_{l, \varepsilon/l}(y)= l \phi_{1,\varepsilon}\left(\frac y l \right) \, .
\end{equation}
This is the solution corresponding to the domain $\Omega_l=l\Omega_1$ in the hyperbolic plane of curvature $-\varepsilon^2/l^2$. 

Choosing then $l=\varepsilon$ we have from (\ref{defphiscale})
\begin{equation*}
\phi_{l,1}(y)= l \phi_{1,l} \left(\frac y l \right) = l \phi_{1,\varepsilon}\left(\frac y l \right) 
\end{equation*} 
a solution of the problem for the domain $\Omega_l:=l\Omega_1=\varepsilon \Omega_1$ in the hyperbolic plane of curvature $-\varepsilon^2/l^2=-1$, that is, the standard hyperbolic plane. In conclusion, once we have a solution for the domain $\Omega_1$ and curvature $-\varepsilon^2$, we automatically have a solution for the smaller (rescaled) domain $\Omega_l=\varepsilon \Omega_1$ and curvature $-1$.

Consider the stress $T^{ij}$ and its rescaled version $\tilde T^{ij}$ at the respective points in Euclidean space. From (\ref{scalestress}) and (\ref{scalephi}) we conclude:
\begin{eqnarray*}
T^{ij}\left(\phi(y)\right) & = & S^{ab}(y) \frac{\partial \phi^i(y)}{\partial y^a}\frac{\partial \phi^j(y)}{\partial y^b}  \, ,\\
\tilde T^{ij}\left(\tilde \phi(y)\right) & = & \tilde S^{ab}(y) \frac{\partial \tilde \phi^i(y)}{\partial y^a}\frac{\partial \tilde \phi^j(y)}{\partial y^b} = S^{ab}\left( \frac y l \right)  \frac{\partial \phi^i(y/l)}{\partial y^a}\frac{\partial \phi^j(y/l)}{\partial y^b} \\
& = & T^{ij}\left(\phi\left(\frac y l \right) \right) \, .
\end{eqnarray*}
The rescaled stress $\tilde T$ at the rescaled point $\tilde \phi(y)$ is thus the same as the original stress $T$ at the point $\phi(y/l)$. 

\subsection{The $3$d Case}
Let $(E_A: A=1, \ldots , 3)$ be the vectorfields (\ref{onbHG}) that satisfy the commutation relations $[E_1, E_2]=e^{\beta}E_3$, $[E_1, E_3]=[E_2,E_3]=0$ and $(\omega^A :A=1, \ldots , 3)$ the dual $1$-forms (\ref{dualbHG}), so that
\begin{equation}
h=\sum_{A=1}^3 \omega^A \otimes \omega^A \quad , \quad h_{AB}=h(E_A,E_B)=\delta_{AB} \, .
\end{equation}
We introduce Riemannian normal coordinates $(y^a : a=1,2,3)$ for $h$ at a given point in $\mathcal N$, which we take as the origin. The components of the metric $h$ in these coordinates are of the form:
\begin{equation} \label{RNKh}
h_{ab}(y)=\delta_{ab}+\epsilon_{ab}(y) \quad , \quad \epsilon_{ab}(y) = O(\abs{y}^2) \, .
\end{equation}
Let $id$ be the identity map defined in Section \ref{HGahs} and $m$ the pullback by $id$ of the Euclidean metric $\delta_{ij} dx^i \otimes dx^j$. Then the components $m_{ab}$ of $m$ in the coordinates $(y^a: a=1,2,3)$ are simply $m_{ab}=\delta_{ab}$ and the corresponding inner product on $\mathcal V$, which depends on $y$, is given by:
\begin{eqnarray*}
\gamma_{AB}(y) & := & \gamma(y)(E_A,E_B) \\
 & = & m_{ab}(y)E_A^a(y) E_B^b(y) \\
 & = & \delta_{ab} E_A^a(y) E_B^b(y) \\
 & = & \left(h_{ab}(y)-\epsilon_{ab}(y)\right)E_A^a(y) E_B^b(y) \\
 & = & \delta_{AB}-\epsilon_{AB}(y) \, ,
\end{eqnarray*}
where $\epsilon_{AB}(y)=\epsilon_{ab}(y) E_A^a(y) E_B^b(y)= O(\abs{y}^2)$. 

We now dilate the Heisenberg group metric $h$ by the factor $l>1$, i.e.~we set:
\begin{equation*} 
\tilde h = l^2 h \quad , \quad \tilde h(E_A, E_B)= l^2 h(E_A, E_B)=l^2 \delta_{AB} \, .
\end{equation*}
Define now $\tilde E_A=l^{-1} E_A$, so
\begin{equation*}
\tilde h(\tilde E_A, \tilde E_B)= \delta_{AB} \, ,
\end{equation*}
that is $(\tilde E_A: A=1,2,3)$ is an orthonormal frame field relative to $\tilde h$. We denote by $ \overset{\tilde\circ}{\gamma}$ the corresponding inner product on $\mathcal V$, i.e.~$  \overset{\tilde\circ}{\gamma}(\tilde E_A, \tilde E_B)=\delta_{AB}$. Remark that the commutation relations of the frame field $(\tilde E_A : A=1,2,3)$ read
\begin{equation*}
[\tilde E_1, \tilde E_2]=l^{-1}\tilde E_3 \quad , \quad [\tilde E_1, \tilde E_3]=[\tilde E_2, \tilde E_3]=0 \, .
\end{equation*}
Now:
\begin{equation*}
\tilde h = l^2 h = l^2 h_{ab}(y) dy^a \otimes dy^b \, .
\end{equation*}
Set $\tilde y^a= l y^a$, then $\tilde h = \tilde h_{ab}(\tilde y) d\tilde y^a \otimes \tilde y^b$ and hence
\begin{equation} \label{scaleHGmet}
\tilde h_{ab}(\tilde y) = h_{ab}\left(\frac{\tilde y}{l}\right) \, .
\end{equation}
Consider the geodesic ray through the origin of the coordinate system $y^a$:
\begin{equation*}
y^a = \lambda^a t \quad , \quad \tilde y^a = \tilde \lambda^a t \quad \textrm{where} \quad \tilde \lambda^a= l \lambda^a  \, . 
\end{equation*}
For the Christoffel symbols of the metric $h$ with respect to the coordinates $y^a$ we have:
\begin{equation*}
\Gamma_{bc}^a(\lambda t)\lambda^b \lambda^c = 0 \, .
\end{equation*}
But from (\ref{scaleChristoffel}),
\begin{equation*}
\tilde \Gamma_{bc}^a(\tilde y)= \frac 1 l \Gamma_{bc}^a\left(\frac{\tilde y}{l}\right) \, , 
\end{equation*}
hence:
\begin{equation}
\tilde \Gamma_{bc}^a(\tilde \lambda t) \tilde \lambda^b \tilde \lambda^c= \frac 1 l \Gamma_{bc}^a(\lambda t) l^2 \lambda^b \lambda^c=0 \, .
\end{equation}
Thus $(\tilde y^a: a=1,2,3)$ are Riemannian normal coordinates for $\tilde h$.

We remark that the mapping $id$ defined in Section \ref{HGahs} depends on the choice of local coordinates in $\mathcal N$. Let us denote by $\tilde{id}$ the mapping associated to the coordinates $(\tilde y^a : a=1,2,3)$, reserving the notation $id$ for the mapping associated to the original coordinates $(y^a : a=1,2,3)$. Let $\tilde m$ be the pullback of the Euclidean metric $\delta_{ij} dx^i \otimes dx^j$ by $\tilde{id}$. Then the components $\tilde m_{ab}$ of $\tilde m$ in the coordinates $(\tilde y^a : a=1,2,3)$ are again simply $\tilde m_{ab}=\delta_{ab}$.

Now, by (\ref{RNKh}) and (\ref{scaleHGmet}),
\begin{equation*}
\tilde h_{ab}(\tilde y) = \delta_{ab} + \tilde \epsilon_{ab} (\tilde y) \, ,
\end{equation*}
where 
\begin{equation*}
\tilde \epsilon_{ab}(\tilde y) = \epsilon_{ab}\left(\frac{\tilde y}{l}\right) \, .
\end{equation*}
Therefore:
\begin{equation*}
\tilde \epsilon_{ab}(\tilde y) = O\left(l^{-2} |\tilde y|^2\right) \, ,
\end{equation*}
and we obtain, in analogy with the above:
\begin{eqnarray*}
\tilde \gamma_{AB}(\tilde y) & := & \tilde \gamma(\tilde y)(\tilde E_A,\tilde E_B) \\
 & = & \tilde m_{ab}(\tilde y)\tilde E_A^a(\tilde y) \tilde E_B^b(\tilde y) \\
 & = & \delta_{ab} \tilde E_A^a(\tilde y) \tilde E_B^b(\tilde y) \\
 & = & \left(\tilde h_{ab}(\tilde y)-\tilde \epsilon_{ab}(\tilde y)\right)\tilde E_A^a(\tilde y) \tilde E_B^b(\tilde y) \\
 & = & \delta_{AB}-\tilde \epsilon_{AB}(\tilde y) \, ,
\end{eqnarray*}
where $\tilde \epsilon_{AB}(\tilde y)=\tilde \epsilon_{ab}(\tilde y) \tilde E_A^a(\tilde y) \tilde E_B^b(\tilde y)= O(l^{-2}\abs{\tilde y}^2)$. Here $\tilde E_A^a$ are the components of the vectorfield $\tilde E_A$ in the coordinate system $(\tilde y^a: a=1,2,3)$.

Let now $\tilde \Omega$ be a fixed domain in the $\tilde y$ coordinates containing the origin. Let $\tilde \phi$ be a solution of the boundary value problem
\begin{equation} \label{tildebvp}
\left\{ \begin{array}{rll} \overset{\tilde m}{\nabla}_b \tilde S^{ab} & = 0 & :\quad \textrm{in} \quad \tilde \Omega \, , \\  \tilde S^{ab} \tilde M_b & = 0 & : \quad \textrm{on} \quad \partial \tilde \Omega \, , \end{array} \right.
\end{equation}
such that $\tilde \phi$ takes the $\tilde y$ coordinate origin in $\mathcal N$ to the $x$ coordinate origin in $E^3$.
In (\ref{tildebvp}) $\tilde S^{ab}= \tilde \pi^{AB} \tilde E_A^a \tilde E_B^b$ and, by definition, 
\begin{equation*}
\sqrt{\frac{\det \tilde \gamma}{\det \overset{\tilde \circ}{\gamma}}}\tilde \pi^{AB}= -2\frac{\partial \tilde e}{\partial \tilde \gamma_{AB}} \, .
\end{equation*}
Since $\tilde h=\tilde h_{ab}(\tilde y) d\tilde y^a \otimes d\tilde y^b$ and $y^a=l^{-1} \tilde y^a$ we have
\begin{equation*}
h=l^{-2} \tilde h = l^{-2}\tilde h_{ab}(\tilde y) d\tilde y^a \otimes d\tilde y^b=\tilde h_{ab}(l y) dy^a \otimes d y^b \, ,
\end{equation*}
hence $h_{ab}(y)=\tilde h_{ab}(ly)$.
We now consider the rescaled (smaller) domain $\Omega = \{ y= l^{-1}\tilde y , \tilde y \in \tilde \Omega \}$ and define the mapping 
\begin{equation*}
\phi(y)=l^{-1} \tilde \phi(ly) \qquad : \forall y \in \Omega \, .
\end{equation*}
The pullback by $\phi$ of the Euclidean metric is $m$, the components of which in the $y^a$ coordinates are
\begin{equation*}
m_{ab}(y)=\delta_{ij} \frac{\partial \phi^i}{\partial y^a}(y) \frac{\partial \phi^j}{\partial y^b}(y) \, .
\end{equation*}
Since
\begin{equation*}
\frac{\partial \phi^i}{\partial y^a}(y)= \frac{\partial \tilde \phi^i}{\partial \tilde y^a}(\tilde y) \, ,
\end{equation*}
we have $m_{ab}(y)=\tilde m_{ab}(\tilde y)$.

Hence, the corresponding inner product on $\mathcal V$, $\gamma_{AB}(y)=\gamma(y)(E_A,E_B)$ in terms of the original basis $(E_A: A=1,2,3)$, is:
\begin{equation*}
\gamma_{AB}(y)=m_{ab}(y) E_A^a(y) E_B^b(y)= \tilde m_{ab}(\tilde y) \tilde E_A^a(\tilde y) \tilde E_B^b(\tilde y)=\tilde \gamma_{AB}(\tilde y) \, .
\end{equation*}
For, 
\begin{equation*}
\tilde E_A^a(\tilde y) \frac{\partial}{\partial \tilde y^a}=l^{-1}E_A^a(y)\frac{\partial}{\partial y^a}  \, ,
\end{equation*}
hence $\tilde E_A^a(\tilde y)=E_A^a(y)$. 
For the stress tensors in both coordinate systems we have $\pi^{AB}(\gamma(y))=\tilde \pi^{AB}(\tilde \gamma(\tilde y))$ because intrinsically $\tilde e(\tilde \gamma)=e(\gamma)$. Therefore:
\begin{equation*}
\tilde S^{ab}(\tilde y)=\tilde \pi^{AB}(\tilde \gamma(\tilde y)) \tilde E_A^a(\tilde y) \tilde E_B^b(\tilde y)= \pi^{AB}(\gamma(y)) E_A^a(y) E_B^b(y)=S^{ab}(y) \, ,
\end{equation*}
which shows that the stress is scaling invariant, as expected. 

\begin{remark}
The energy per unit mass $e$ is a function of the configuration $\gamma$, an inner product on $\mathcal V$. However, when representing $e$ as a function of the components $\gamma_{AB}$ of $\gamma$ in a basis $(E_A: A=1,\ldots,n)$, which is orthonormal relative to $\stackrel{\circ}{\gamma}$, $e$ in this representation depends indirectly on $\stackrel{\circ}{\gamma}$.
\end{remark}

Consider the equations in (\ref{tildebvp}). Starting from $\tilde m_{ab}(\tilde y)=m_{ab}(y)$,
\begin{equation*}
\frac{\partial \tilde m_{ab}}{\partial \tilde y^c}(\tilde y)=l^{-1}\frac{\partial \tilde m_{ab}}{\partial y^c}(\tilde y) = l^{-1} \frac{\partial m_{ab}}{\partial y^c}(y) \, ,
\end{equation*}
the Christoffel symbols $\Gamma_{ab}^c$ transform accordingly
\begin{equation} \label{scalechristoffelH}
\overset{\tilde m}{\Gamma _{ab}^c} (\tilde y)= l^{-1} \overset{m}{\Gamma_{ab}^c}(y) \, .
\end{equation}
Since we have 
\begin{equation} \label{scalederSH}
\frac{\partial \tilde S^{ab}}{\partial \tilde y^c}(\tilde y)= l^{-1} \frac{\partial S^{ab}}{\partial y^c}(y) \, ,
\end{equation}
we finally obtain, using (\ref{scalechristoffelH}), (\ref{scalederSH}),
\begin{equation*}
\left(\overset{\tilde m}{\nabla}_b \tilde S^{ab}\right)(\tilde y)= \frac{\partial \tilde S^{ab}}{\partial \tilde y^b}(\tilde y)+ \overset{\tilde m}{\Gamma_{bc}^a}(\tilde y)\tilde S^{bc}(\tilde y) + \overset{\tilde m}{\Gamma_{ac}^b}(\tilde y) \tilde S^{ac}(\tilde y)=l^{-1} \left(\overset{m}{\nabla}_b  S^{ab} \right)(y) \, .
\end{equation*}
Therefore, once we have a solution for the equation in the $\tilde y$ coordinates, we immediately obtain the solution for the original equation and the boundary condition is also satisfied since $\tilde M_a(\tilde y)=M_a(y)$.

\newpage


\part{The Analysis of Equilibrium Configurations} \label{analysises}
\setcounter{section}{0}


\section{Stress Tensor}
 We will now show that the stress tensor $S$ on the material manifold $\mathcal N$ is given by
\begin{equation} \label{stressonN}
\sqrt{\frac{\det m}{\det \stackrel{\circ}{n}}} S^{ab}=-2 \frac{\partial e}{\partial m_{ab}} \, .
\end{equation}
Recall first the definition (\ref{defthermstress}) of the thermodynamic stress $\pi$ on the crystalline structure $\mathcal V$ 
\begin{equation} \label{defstresstilde}
\sqrt{\frac{\det \gamma}{\det \stackrel{\circ}{\gamma}}} \pi^{AB} = -2 \frac{\partial e}{\partial \gamma _{AB}} \, ,
\end{equation}
where we have made use of the definition of the volume $V(\gamma)=\sqrt{\frac{\det \gamma}{\det \stackrel{\circ}{\gamma}}}$ in the frame $(E_1,\ldots,E_n)$ satisfying $\omega(E_1,\ldots,E_n)=1$, $\stackrel{\circ}{\gamma}_{AB}=\delta_{AB}$ (so in fact $\det \stackrel{\circ}{\gamma}=1$).

Now, from (\ref{gammapbm}) we have
\begin{equation*}
\frac{\partial e}{\partial m_{ab}}= \frac{\partial e}{\partial \gamma_{AB}}\frac{\partial \gamma_{AB}}{\partial m_{ab}} = \frac{\partial e}{\partial \gamma_{AB}} E_A^a E_B^b \, ,
\end{equation*}
hence, with $E(y)=\omega(y)^{-1}$,
\begin{equation} \label{deregammaderem}
\frac{\partial e}{\partial \gamma_{AB}}=\frac{\partial e}{\partial m_{ab}} \omega_a^A \omega_b^B \, ,
\end{equation}
and finally, using (\ref{defSpfpi}), (\ref{defstresstilde}), (\ref{deregammaderem}) and the equality
\begin{equation} \label{detratio}
\frac{\det m}{\det \stackrel{\circ}{n}}=\frac{\det \gamma}{\det \stackrel{\circ}{\gamma}} \, ,
\end{equation}
we obtain:
\begin{equation*}
\sqrt{\frac{\det m}{\det \stackrel{\circ}{n}}} S^{ab} \omega_a^A \omega_b^B = \sqrt{\frac{\det \gamma}{\det \stackrel{\circ}{\gamma}}} \pi^{AB} = -2 \frac{\partial e}{\partial \gamma _{AB}}= -2 \frac{\partial e}{\partial m_{ab}} \omega_a^A \omega_b^B \, ,
\end{equation*}
which directly implies (\ref{stressonN}). The equality (\ref{detratio}) is a special case of the following proposition which applies to the isotropic case.
{\proposition \label{eigenvalues} The eigenvalues $\lambda_1, \ldots, \lambda_n$ of $\gamma$ with respect to $\stackrel{\circ}{\gamma}$ coincide with the eigenvalues of $m$ with respect to $\stackrel{\circ}{n}$.}

\begin{proof}
We define $A(y) \in \mathcal L(\mathcal V, \mathcal V)$ by
\begin{equation*}
\stackrel{\circ}{\gamma}\left(A(y)Y_1,Y_2\right) = \gamma(y) (Y_1,Y_2) \quad : \, \forall Y_1, Y_2 \in \mathcal V \, .
\end{equation*}
Hence $A(y)E_B=A_B^A(y)E_A$, where 
\begin{equation*}
A_B^A(y) = (\stackrel{\circ}{\gamma}^{-1})^{AC}\gamma_{BC}(y) \, .
\end{equation*}
Thus $\lambda_1, \ldots, \lambda_n$ are the eigenvalues of $A$.

Similarly, we define $B(y) \in \mathcal L(T_y \mathcal N, T_y \mathcal N)$ by
\begin{equation*}
\stackrel{\circ}{n}\left(B(y)Y_{1,y},Y_{2,y}\right) = m(y) (Y_{1,y},Y_{2,y}) \quad : \, \forall Y_{1,y}, Y_{2,y} \in T_y\mathcal N \, .
\end{equation*}
and using $B(y)\left.\frac{\partial}{\partial y^a}\right|_y=B_a^b(y)\left.\frac{\partial}{\partial y^b}\right|_y$, we want to express $A(y)$ in terms of $B(y)$.

Since the evaluation map $\epsilon_y$ is an isomorphism from $\mathcal V$ to $T_y \mathcal N$ for each $y \in \mathcal N$, we have
\begin{equation*}
A(y)= \epsilon_y^{-1} \circ B(y) \circ \epsilon_y \, .
\end{equation*}
Hence, using $E_B(y)=E_B^b(y)\left.\frac{\partial}{\partial y^b}\right|_y$,
\begin{eqnarray*}
A_B^A(y) E_A & = &  A(y)E_B = \epsilon_y^{-1} \left( B(y) E_B(y) \right) \\
 & = & \epsilon_y^{-1} \left(E_B^b(y) B_b^a(y) \left.\frac{\partial}{\partial y^a}\right|_y \right) \\
 & = &  E_B^b(y) B_b^a(y) \omega_a^A(y) E_A \, ,
\end{eqnarray*}
we conclude
\begin{equation*}
A(y)=E(y)B(y)\omega(y)=\omega(y)^{-1}B(y)\omega(y) \quad , \quad E(y)=\omega(y)^{-1} \, ,
\end{equation*}
i.e.~the linear mappings $A(y)$ and $B(y)$ are conjugate and therefore $\lambda_1, \ldots, \lambda_n$ are also the eigenvalues of $B(y)$. \end{proof} 

Let us now calculate the stress tensor $S^{ab}$ corresponding to the toy energy (\ref{toynrg}). Here, we denote $\stackrel{\circ}{n}$ by $h$, the metric of the hyperbolic plane that is the metric of the material manifold for a uniform distribution of elementary edge dislocations. Recall that 
\begin{eqnarray*}
e = \frac 1 2 \left( (\lambda_1-1)^2+(\lambda_2-1)^2 \right) & = & \frac 1 2 \left( \lambda_1^2 + \lambda_2^2 -2 (\lambda_1+\lambda_2)+2\right) \\
& = & \frac 1 2 \textrm{tr}_h(m^2)-\textrm{tr}_h m + 1 \\
& = & \frac 1 2 \left(h^{-1}\right)^{ac}\left(h^{-1}\right)^{bd}m_{ab}m_{cd}-\left(h^{-1}\right)^{ab}m_{ab}+1 \, .
\end{eqnarray*} 
Then:
\begin{equation}\label{stressedgeup}
\sqrt{\frac{\det m}{\det h}} S^{ab}=-2 \frac{\partial e}{\partial m_{ab}}=-2\left(\left(h^{-1}\right)^{ac}\left(h^{-1}\right)^{bd}m_{cd}-\left(h^{-1}\right)^{ab}\right) \, ,
\end{equation}
or, 
\begin{equation}\label{stressedge}
 S^{ab}=2 \sqrt{\frac{\det h}{\det m}}\left(h^{-1}\right)^{ac}\left(h^{-1}\right)^{bd}\left(h_{cd}-m_{cd}\right) \, .
\end{equation}
The physical interpretation of (\ref{stressedge}) is the following:
\begin{itemize}
\item[{\bf i)}] If $m$ is smaller than $h$ then the stress is positive,
\item[{\bf ii)}] if $m$ is larger than $h$ then the stress is negative.
\end{itemize}
Recall that a quadratic form $q=h-m$ on $T_y \mathcal N$ is said to be positive (negative) if, for all $v \in T_y\mathcal N$,
\begin{equation*}
q(v,v) > 0 \, (<0) \quad : \, \forall v \neq 0 \, .
\end{equation*}

\section{Setup and Method in $2$d}\label{setupPDE2d}
As was shown in Part I, Section \ref{edag}, the material manifold $\mathcal N$ for the case of a uniform distribution of edge dislocations in two dimensions is given by the affine group, and a left-invariant metric on $\mathcal N$ gives $\mathcal N$ the structure of the hyperbolic plane $H_{\varepsilon}$ of curvature $-\varepsilon^2$.

To solve the problem in this case, we fix an origin in $H_{\varepsilon}$ and set up Riemannian normal coordinates $(y^a: a=1,2)$ as in Section \ref{HGahs} of Part II. Let $\Omega$ be any smooth bounded domain in these coordinates, containing the origin. Note that as $\varepsilon \to 0$, $\left.H_{\varepsilon}\right|_{\Omega}$ tends to $\left.H_0\right|_{\Omega}$, where $H_0=E$ is the Euclidean plane. We also choose an origin and set up rectangular coordinates $(x^i:i=1,2)$ in $E$. An identity mapping 
\begin{equation} \label{idmap}
\begin{array}{rcl}
id: H_{\varepsilon} & \to & E \\
(y^1,y^2) & \mapsto & (x^1,x^2)=(y^1,y^2)
\end{array}
\end{equation}
is then defined as in Section \ref{HGahs} of Part II.

We now restrict the allowed mappings $\phi: \Omega \subset H_{\varepsilon} \to E$ by the following two requirements. First, $\phi$ should map the origin in $H_{\varepsilon}$ into the origin in $E$. Second, $d\phi(0)$ should map the vector $\left.\frac{\partial}{\partial y^1}\right|_0$ into a vector of the form $\lambda \left.\frac{\partial}{\partial x^1}\right|_0$ for some $\lambda > 0$. By virtue of this restriction, the identity mapping (\ref{idmap}) (restricted to $\Omega$) is for $\varepsilon = 0$ the unique minimizer of our toy energy $e$ from (\ref{toynrg}).

{\remark The restriction is needed to ensure uniqueness. Otherwise, composition on the left with a rigid motion of $E$ gives another minimizer. The appropriate restriction in the three dimensional case will be stated below in Section \ref{uniqueness}. Analogously, it can be formulated in any number of space dimensions. Thus uniqueness is ensured in general. The argument of Part I, Section \ref{HyperbolicExpansion} applies with the hyperbolic plane $H_{\varepsilon}$ and the Euclidean plane replaced by the $n$-dimensional hyperbolic space $H_{\varepsilon}^n$ and the $n$-dimensional Euclidean space. This is as long as the toy energy (\ref{toynrg}) is considered.}

\vspace{2mm}

From (\ref{divstresszeroN}) and (\ref{bdryc}) the system of partial differential equations and the corresponding boundary conditions for the static problem of a uniform distribution of edge dislocations in $n=2$ dimensions is of the form 
\begin{equation*}
F_{\varepsilon}\left[ \phi \right]=0 \, ,
\end{equation*}
where
\begin{equation} \label{Fepspb}
F_{\varepsilon}\left[ \phi \right]=\left( \begin{array}{c} \overset{m}{\nabla}_b S^{ab} \\ S^{ab}M_b \end{array} \right) \, .
\end{equation}
We linearize the equations at the identity mapping, which is a solution for $\varepsilon=0$. First, we solve the linearized problem using the theorem of Lax-Milgram (as in \cite{ABS}). An iteration will then show that there exists a solution to the nonlinear problem for sufficiently small $\varepsilon$, and, therefore, by the scaling argument of Section \ref{scaling} of Part II, that there is a mapping from a rescaled domain $\tilde \Omega$ in the standard hyperbolic plane $H_1=H$ of curvature $-1$ to the Euclidean plane, $\phi: \tilde \Omega \subset H \to E$, satisfying the conditions of the problem.

We set $\phi = id + \psi$, where $\psi$ is a small deviation from the identity mapping. We have:
\begin{equation*}
F_0\left[ id \right]=0 \, .
\end{equation*}
In a neighborhood of the identity $F_{\varepsilon}$ is of the form
\begin{equation*}
F_{\varepsilon}\left[ \phi \right]=F_{\varepsilon}\left[id\right] + D_{id} F_{\varepsilon} \cdot \psi + N_{\varepsilon}[\psi] \, ,
\end{equation*}
where $N_{\varepsilon}[\psi]$ is to leading order quadratic in $\psi$. We denote by $L_{\varepsilon}$ the linearized operator $D_{id}F_{\varepsilon}$. Then $F_{\varepsilon}\left[ \phi \right]=0$ reads
\begin{equation} \label{linearizationF}
L_{\varepsilon} \cdot \psi = -F_{\varepsilon}\left[id\right] -N_{\varepsilon}[\psi] \, .
\end{equation}
We will solve this by an iteration starting at $\psi_0=0$. In the first step of the iteration we have the linearized equations
\begin{equation} \label{stepone}
L_{\varepsilon} \cdot \psi_1 = -F_{\varepsilon}\left[ id \right] \, .
\end{equation}
The term $-F_{\varepsilon}\left[id\right]$ in (\ref{linearizationF}), (\ref{stepone}) can be interpreted as a source term.

A possible approach to the problem is to study the iteration
\begin{equation} \label{firstit}
L_{\varepsilon} \cdot \psi_{n+1} = -F_{\varepsilon}\left[ id \right]-N_{\varepsilon}[\psi_n]
\end{equation}
and show that $\psi_n$ converges to a solution $\psi=\phi-id$ of (\ref{linearizationF}), provided that we choose $\varepsilon$ appropriately small. However, we follow a different approach.

Adding $L_0 \cdot \psi$ on both sides of (\ref{linearizationF}) yields
\begin{equation*}
L_0 \cdot \psi = -\left(L_{\varepsilon}-L_0\right)\cdot \psi -F_{\varepsilon}[id]-N_{\varepsilon}[\psi] \, .
\end{equation*} 
What we actually do is to consider the iteration
\begin{equation}\label{iteration}
L_0 \cdot \psi_{n+1} = - \left(L_{\varepsilon} - L_0\right)\cdot \psi_n - F_{\varepsilon}[id] - N_{\varepsilon}[\psi_n] \, .
\end{equation}
Note that while both (\ref{firstit}), (\ref{iteration}) are linear in the next iterate $\psi_{n+1}$, in (\ref{firstit}) the operator $L_{\varepsilon}$ which refers to $H_{\varepsilon}$ acts on $\psi_{n+1}$ whereas in (\ref{iteration}) the operator which refers to $H_0=E$ acts on $\psi_{n+1}$.

In (\ref{iteration}), $L_{\varepsilon}-L_0$ is a pair of linear operators, a second order operator in $\Omega$ and a first order operator on $\partial \Omega$. The coefficients of these operators are of order $\varepsilon^2$, which we may write symbolically in the form:
\begin{equation}\label{linit}
L_{\varepsilon}-L_0 \sim \partial \left(h - \delta\right) = \varepsilon^2 \partial f \, ,
\end{equation}
since, from (\ref{hypmetexpeps}), $h-\delta = \varepsilon^2 f$.

The iteration (\ref{iteration}) starts also with $\psi_0=0$. Then, setting $n=0$ in (\ref{iteration}), we have: 
\begin{equation} \label{step0it}
L_0 \cdot \psi_1 = -F_{\varepsilon}[id] \, ,
\end{equation}
thus, taking into account the fact that $\det h=1+O(\varepsilon^2)$, (\ref{stressedge}) implies $S^{ab}(id)=-2\varepsilon^2 \left.f_{ab}\right|_{\varepsilon=0}+O(\varepsilon^4)$. It follows that $\psi_1$ is of order $\varepsilon^2$. This is step one of the iteration, the linear level. 

For $n \geq 1$, taking the difference between (\ref{iteration}) and the same with $n$ replaced by $n-1$, we obtain
\begin{equation*}
L_0\cdot\left(\psi_{n+1}-\psi_n\right)= -\left(L_{\varepsilon}-L_0\right)\cdot \left(\psi_n-\psi_{n-1}\right)-\left( N_{\varepsilon}[\psi_n]-N_{\varepsilon}[\psi_{n-1}]\right) \, .
\end{equation*}
In view of (\ref{linit}) and the fact that $N_{\varepsilon}[\psi]$ is to leading order quadratic in $\psi$, for sufficiently small $\varepsilon$, contraction will hold and $\psi_n$ will converge to a solution $\psi$ of (\ref{linearizationF}). 
Thus $\phi=id+\psi$ solves the problem (\ref{Fepspb}).

Consider now the equations in (\ref{Fepspb}).
To analyze the problem at the linearized level, we consider the variation $\dot m_{ab}$ of the metric $m_{ab}$ at the identity mapping $id$. Setting
\begin{equation}\label{variationid}
\phi^i= y^i+s \psi^i \, ,
\end{equation} 
and recalling that 
\begin{equation*}
m_{ab}(y)= \frac{\partial \phi^i(y)}{\partial y^a}\frac{\partial \phi^i(y)}{\partial y^b} \, ,
\end{equation*}
we obtain:
\begin{equation} \label{Deltamab}
\dot m_{ab}=\left. \frac{d}{ds}\right|_{s=0} \left[\left(\delta_a^i+s\frac{\partial \psi^i}{\partial y^a}\right)\cdot \left(\delta_b^i+s\frac{\partial \psi^i}{\partial y^b}\right)\right]=\frac{\partial \psi^i}{\partial y^a}\delta_b^i+\delta_a^i\frac{\partial\psi^i}{\partial y^b}=\frac{\partial \psi^b}{\partial y^a}+\frac{\partial\psi^a}{\partial y^b} \, .
\end{equation}
Hence, $\dot m_{ab}$ is the Lie derivative of the metric $\delta_{ab}$ on $H_{\varepsilon}$ (the pullback by $id$ of the metric $\delta_{ij}$ on $E$) with respect to the vectorfield $\psi^a \frac{\partial}{\partial y^a}$ on $H_{\varepsilon}$ (which the push-forward by $id$ takes to the vectorfield $\psi^i \frac{\partial}{\partial x^i}$ on $E$).
 
\section{Analogous Geometric Linear Problem} \label{AnalogousP}

In the work on the stability of the Minkowski space-time \cite{CK}, a linear geometric problem was studied analogous to the linear problems in (\ref{firstit}) and (\ref{iteration}).

Let $(M,g)$ be a compact Riemannian manifold with boundary $\partial M$ and $X$ a vectorfield on $M$. We set
\begin{equation} \label{stresspi}
\pi= \mathcal L_X g \quad , \quad \pi_{ij}=\nabla_i X_j + \nabla_j X_i  \quad , \quad X_i=g_{ij}X^j \, ,
\end{equation} 
the Lie derivative of the metric $g$ along $X$, a symmetric $2$-covariant tensorfield on $M$. The variation of $\pi$ with respect to $X$ reads:
\begin{equation} \label{varofpi}
\dot \pi_{ij}=\nabla_i \dot X_j +\nabla_j\dot X_i \, .
\end{equation}
We consider the problem of free minimization of the action integral
\begin{equation} \label{actionM}
A=\int_{M}\left(\frac 1 4 |\pi|_g^2 +\rho^i X_i \right)d\mu_g-\int_{\partial M} \tau^i X_i \left.d\mu_g\right|_{\partial M} \,  .
\end{equation}
Here $\rho$ is a given vectorfield on $M$ and $\tau$ is a given vectorfield along $\partial M$. In mechanical terms $\rho$ is the body force and $\tau$ the boundary force.

The first variation of (\ref{actionM}), using (\ref{varofpi}), is
\begin{equation} \label{firstvarAinh}
\dot A  = \int_{M}\left(\pi^{ij}\nabla_j\dot X_i +\rho^i \dot X_i\right) d\mu_g - \int_{\partial M} \tau^i \dot X_i \left.d\mu_g\right|_{\partial M}  \, .
\end{equation}
For variations of $X_i$ which vanish near the boundary we have
\begin{equation*} 
\dot A = -\int_M \left( \nabla_j \pi^{ij} - \rho^i \right) \dot X_i d\mu_g \, ,
\end{equation*}
and requiring $\dot A=0$ for such variations yields the Euler-Lagrange equations:
\begin{equation} \label{ELinhom}
\nabla_j \pi^{ij} = \rho^i  \quad :\textrm{in} \,\, M \, .
\end{equation}
Requiring then $\dot A=0$ for arbitrary variations:
\begin{equation*}
- \int_{\partial M} \tau^i \dot X_i \left.d\mu_g\right|_{\partial M} = \int_{\partial M}\left(\pi^{ij} N_j - \tau^i\right)\dot X_i  \left.d\mu_g\right|_{\partial M}=0 \, ,
\end{equation*}
yields the boundary conditions:
\begin{equation} \label{bdryinhom}
\pi^{ij}N_j = \tau^i  \quad :\textrm{on} \,\, \partial M \, .
\end{equation} 
The equations (\ref{ELinhom}), together with the boundary conditions (\ref{bdryinhom}), correspond to the linearized boundary value problem (\ref{linit}) corresponding to a crystalline solid with a uniform distribution of dislocations in equilibrium if we identify $(\Omega, \dot m, id^*\delta)$ with $(M,\pi,g)$, $\delta$ being the Euclidean metric.

We proceed in showing self adjointness of the above operators in the following sense. Let $Y$ be a vectorfield on $M$, $\sigma=\mathcal L_Y g$, i.e.~ 
\begin{equation} \label{defsigmacomp}
\sigma_{ij}=\nabla_i Y_j+\nabla_j Y_i=\sigma_{ji} \quad , \quad Y_i=g_{ij}Y^{j} \, .
\end{equation}
Then we have by repeated partial integration using (\ref{stresspi}),
\begin{equation*}
\left<Y, \nabla \cdot \pi \right>_{L^2(M)}-\left<Y, \pi \cdot N\right>_{L^2(\partial M)}=
\left<X, \nabla \cdot \sigma \right>_{L^2(M)}-\left<X, \sigma \cdot N\right>_{L^2(\partial M)} \, .
\end{equation*}
Suppose now that $Y$ is a Killing field, i.e.~$\sigma=\mathcal L_Y g=0$.
Then we have by (\ref{ELinhom}), (\ref{bdryinhom}) and (\ref{defsigmacomp})
\begin{equation*}
\int_M Y_i \rho^i  =  \int_M Y_i \nabla_j \pi^{ij} = -\frac 1 2 \int_M \sigma_{ij} \pi^{ij} + \int_{\partial M} Y_i \pi^{ij} N_j =  \int_{\partial M} Y_i \tau^i \, ,
\end{equation*}
since $\sigma_{ij}=0$ on $M$. Thus the integrability condition reads:
\begin{equation} \label{intcond}
\int_M Y_i \rho^i = \int_{\partial M} Y_i \tau^i \, .
\end{equation}
This condition guarantees the existence of a solution for the boundary value problem (\ref{ELinhom}), (\ref{bdryinhom}) by the theorem of Lax-Milgram.

\subsection{Uniqueness of the Solution} \label{uniqueness}
In fact, the solution is unique up to an additive Killing field. For, if we take two solutions $X_1$ and $X_2$ of (\ref{ELinhom}) with (\ref{bdryinhom}), their difference $X=X_1-X_2$ satisfies the homogeneous equation of (\ref{ELinhom}), i.e.~$\rho = 0$ with zero boundary conditions, (\ref{bdryinhom}) for $\tau=0$.
Therefore, we have, setting $\pi=\mathcal L_X g$, $\nabla_j\left(\pi^{ij}X_i\right)=\pi^{ij} \nabla_j X_i$,
\begin{equation*}
A= \frac 1 4 \int_M |\pi|^2 =\frac 1 2 \int_M \pi^{ij}\nabla_j X_i=\frac 1 2 \int_M \nabla_j \left(\pi^{ij}X_i \right)  \, .
\end{equation*}
Using Gauss's theorem we obtain:
\begin{equation*}
A=\frac 1 2 \int_M \nabla_j \left(\pi^{ij}X_i \right) =\frac 1 2 \int_{\partial M} \underbrace{\left(\pi^{ij}N_j \right)}_{=0} X_i =0 \, .
\end{equation*}
It follows that $\pi=0$, and thus $X=X_1-X_2$ is a Killing field, i.e.~the solutions $X_1, X_2$ only differ by a Killing field.

For our problem, where $(M,g)$ is isometric to a domain in the Euclidean plane, recalling the restriction
\begin{eqnarray*}
\phi(0) & = & 0 \, , \\
d\phi(0)\cdot \left. \frac{\partial}{\partial y^1} \right|_0 & = & \lambda  \left. \frac{\partial}{\partial x^1} \right|_0 \, ,
\end{eqnarray*}
we obtain, setting
\begin{equation*}
\phi^i(x)=y^i+s X^i(y) \, ,
\end{equation*}
the conditions
\begin{eqnarray*}
s X^i(0) & = & 0 \, , \\
\delta_1^i + s \frac{\partial X^i}{\partial y^1}(0) & = & \lambda(s) \delta_1^i  \, .
\end{eqnarray*}
Taking then the derivative with respect to $s$ at $s=0$ yields the linearized conditions:
\begin{eqnarray}
X^i(0) & = & 0 \, ,  \label{firstcond} \\
\frac{\partial X^i}{\partial y^1}(0) & = & \mu \delta_1^i \, , \label{secondcond}
\end{eqnarray}
where $\mu=\dot \lambda(0)$. The second of the above conditions is
\begin{equation*}
\frac{\partial X^2}{\partial y^1}(0)=0 \, .
\end{equation*} 
Substituting the general form of a Killing field,
\begin{equation*}
X^i=\alpha^i_j y^j + \beta^i \quad , \quad \alpha^i_j=-\alpha^j_i \quad , 
\end{equation*}
it follows from the first condition (\ref{firstcond}) that $\beta^i=0$, and from the second condition (\ref{secondcond}) that $\alpha_1^2=0$, hence $\alpha^i_j=0$, i.e.~$X=0$. So the solution of the linearized problem is in fact unique.

We remark that in three dimensions, one must add, for uniqueness, the condition that 
\begin{equation*}
d\phi(0)\cdot \left. \frac{\partial}{\partial y^2} \right|_0 \, 
\end{equation*}
is a vector at the origin contained in the plane spanned by $\left. \frac{\partial}{\partial x^1} \right|_0 $ and $\left. \frac{\partial}{\partial x^2} \right|_0$. This can always be arranged by a suitable rotation in three-dimensional Euclidean space. Thus
\begin{equation*}
d\phi(0)\cdot \left. \frac{\partial}{\partial y^2} \right|_0 = \lambda_1  \left. \frac{\partial}{\partial x^1} \right|_0 + \lambda_2  \left. \frac{\partial}{\partial x^2} \right|_0 \, .
\end{equation*}
At the linearized level, this additional condition gives:
\begin{equation} \label{addcond}
\frac{\partial X^i}{\partial y^2}(0) = \mu_1 \delta_1^i + \mu_2 \delta_2^i  \, ,
\end{equation}
where $\mu_1=\dot \lambda_1(0)$, $\mu_2=\dot \lambda_2(0)$. Setting $i=3$ in (\ref{addcond}), we have:
\begin{equation*}
\frac{\partial X^3}{\partial y^2}(0)=0 \, ,
\end{equation*}
while the conditions from (\ref{secondcond}) read
\begin{equation*}
\frac{\partial X^2}{\partial y^1}(0)=\frac{\partial X^3}{\partial y^1}(0)=0 \, .
\end{equation*}
Hence, if $X$ is a Killing field,
\begin{equation*}
\alpha_1^2=\alpha_1^3=0 \quad \textrm{and} \quad \alpha_2^3=0 \, ,
\end{equation*} 
or, taking into account that also $\beta^i=0$ ($i=1,\ldots,3$), we conclude that $X=0$.

\section{The Linear Problem}\label{linearcase}
The unknown of the problem being the mapping $\phi: \Omega \to E$, the pullback metric $m=\phi^* \delta$ depends on $\phi$ according to:
\begin{equation*}
m_{ab}[\phi](y)=\sum_i \frac{\partial \phi^i(y)}{\partial y^a}\frac{\partial \phi^i(y)}{\partial y^b} \, .
\end{equation*}
Since $m_{ab}(id)=\delta_{ab}$, and $\phi= id + \psi$, i.e.~$\phi^i=y^i+\psi^i(y)$, we have:
\begin{equation*}
m_{ab}=\delta_{ab}+\dot m_{ab}+\mu_{ab}[\psi] \quad , \, \textrm{where} \quad \dot m_{ab}=\frac{\partial \psi^b}{\partial y^a}+\frac{\partial \psi^a}{\partial y^b} 
\end{equation*}
and $\mu_{ab}[\psi]$ is quadratic in $\psi$.
It follows from (\ref{iteration}), together with (\ref{linit}), that $\psi=O(\varepsilon^2)$ at the linear level, whence $\psi^2=O(\varepsilon^4)$. Furthermore, from the expansion of the hyperbolic metric in rectangular coordinates we have 
\begin{equation*}
h_{ab}=\delta_{ab}+\varepsilon^2 l_{ab}+O(\varepsilon^4) \, ,
\end{equation*}
where $l_{ab}= \left. f_{ab} \right|_{\varepsilon = 0}$.
 
Linearizing the operator $F_{\varepsilon}$ from (\ref{Fepspb}) at the identity mapping, $L_0 \psi =D_{id}F_0\cdot \psi$, we find:
\begin{equation*}
L_0 \psi = \left\{ \begin{array}{rl} -2 \partial_b \dot m_{ab} & : \, \textrm{in} \,\, \Omega \, , \\ -2\dot m_{ab} M_b & :  \, \textrm{on} \,\, \partial \Omega \, , \end{array} \right.
\end{equation*}
with $\dot m_{ab}$ as above. Moreover, for $\phi=id$ we have, $m_{ab}=\delta_{ab}$, $\overset{m}\nabla_b =\frac{\partial}{\partial y^b}$. Hence:
\begin{equation} \label{fracstress}
\frac 1 2 S^{ab}(id)=\sqrt{\det h}\left(h^{-1}\right)^{ac}\left(h^{-1}\right)^{bd}\left(h_{cd}-m_{cd}\right)= \varepsilon^2 l_{ab} + O(\varepsilon^4) \, ,
\end{equation}
where we have made use of 
\begin{equation*}
\sqrt{\det h}=\sqrt{\det(\delta+\varepsilon^2 l)}=1+O(\varepsilon^2) \, .
\end{equation*} 
Therefore,  
\begin{equation*}
F_{\varepsilon}(id) = \left\{ \begin{array}{rll} \partial_b S^{ab}(id) \, = & 2 \varepsilon^2 \partial_b l^{ab} + O(\varepsilon^4) & : \, \textrm{in} \,\, \Omega \, , \\  S^{ab}(id)M_b \, = & 2\varepsilon^2 l^{ab}M_b + O(\varepsilon^4) & :  \, \textrm{on} \,\, \partial \Omega \, . \end{array} \right.
\end{equation*} 
Thus, dropping terms of $O(\varepsilon^4)$ in $F_{\varepsilon}(id)$ which come from the $O(\varepsilon^4)$ terms in (\ref{fracstress}), the linearized problem (\ref{step0it}) reduces to the boundary value problem
\begin{equation}\label{LinearPb}
\left\{ \begin{array}{rll} \frac{\partial}{\partial y^b} \left(\dot m_{ab} - \varepsilon^2 l_{ab} \right) & = 0 & :\quad \textrm{in} \quad \Omega \, , \\  \left(\dot m_{ab} - \varepsilon^2 l_{ab} \right) M_b & = 0 & : \quad \textrm{on} \quad \partial \Omega \, . \end{array} \right.
\end{equation}
We consider the following problem analogous to (\ref{LinearPb}):
\begin{equation} \label{anabvp}
\left\{ \begin{array}{rll} \nabla_j\left(\pi^{ij}-\sigma^{ij} \right) & = 0 & :\quad \textrm{in} \quad M \, , \\ \left(\pi^{ij}-\sigma^{ij}\right)N_j & = 0 & : \quad \textrm{on} \quad \partial M \, , \end{array} \right.
\end{equation}
where $\pi_{ij}=\left(\mathcal L_X g\right)_{ij}$, $\sigma_{ij}$ is a given symmetric $2$-covariant tensorfield on $M$, $N_j$ is a covector whose null space is the tangent plane $T_x M$ at $x\in\partial M$, and $N^i=(g^{-1})^{ij}N_j$ is the corresponding outer unit normal vector to $\partial M$.

To obtain the solution of (\ref{anabvp}), we find a vectorfield $\tilde X$ such that $\tilde \pi=\mathcal L_{\tilde X}g$ satisfies
\begin{equation} \label{bdryvalue}
\left(\tilde \pi^{ij}-\sigma^{ij}\right)N_j=0  \qquad : \quad \textrm{on} \quad \partial M \, .
\end{equation}
We define $X'=X-\tilde X$ and $\pi'=\mathcal L_{X'} g=\mathcal L_X g-\mathcal L_{\tilde X}g=\pi-\tilde \pi$, to obtain from (\ref{bdryvalue}) and the boundary conditions of (\ref{anabvp}):
\begin{equation*}
(\pi')^{ij}N_j=\left(\pi^{ij}-\tilde \pi^{ij} \right)N_j=\left(\pi^{ij}-\sigma^{ij} \right)N_j=0 \qquad : \quad \textrm{on} \quad \partial M \, .
\end{equation*}
Defining the vectorfield $\rho$ in $M$ by $\rho^i=\nabla_j \left(\sigma^{ij}-\tilde \pi^{ij}\right)$, the problem then reduces to:
\begin{equation*} 
\left\{ \begin{array}{rll} \nabla_j (\pi')^{ij} & = \rho^i & :\quad \textrm{in} \quad M \, , \\ (\pi')^{ij}N_j & = 0 & : \quad \textrm{on} \quad \partial M \, . \end{array} \right. 
\end{equation*}
This is of the same form as (\ref{anabvp}), but with zero boundary conditions. To see whether this problem has a solution, we need to check the integrability condition (\ref{intcond}) (i.e.~orthogonality to the Killing fields in the $L^2$ sense). Here, the boundary terms vanish since we have zero boundary conditions and it remains for us to show that
\begin{equation*}
\int_M \xi_i \rho^i d\mu_g =0 \, ,
\end{equation*}
for all Killing fields $\xi$. We have:
\begin{eqnarray*}
\int_M\xi_i \rho^i d\mu_g & = & \int_M \xi_i\nabla_j\left(\sigma^{ij}-\tilde \pi^{ij} \right) d\mu_g \\
 & = &  \int_{\partial M} \xi_i\left(\sigma^{ij}-\tilde \pi^{ij} \right)N_j \left. d\mu_g\right|_{\partial M}- \int_M \nabla_j\xi_i \left(\sigma^{ij}-\tilde \pi^{ij} \right) d\mu_g=0  \, ,
\end{eqnarray*}
where the first term is zero due to (\ref{bdryvalue}) (that is, the choice of $\tilde X$), and the second one vanishes by virtue of the fact that $\xi$ is a Killing field,
\begin{equation*}
\nabla_j \xi_i \left(\sigma^{ij}-\tilde \pi^{ij}\right)= \frac 1 2 \left( \nabla_i \xi_j + \nabla_j \xi_i \right) \left(\sigma^{ij}-\tilde \pi^{ij}\right)=0 \, .
\end{equation*}
By applying the Lax-Milgram theorem as in \cite{ABS}, we conclude that there is a solution $X$ to the generalized linear problem which is unique up to an additive Killing field. Thus, also the linear case of the original problem (\ref{LinearPb}), viewed as a special case of the above, has a unique solution up to an additive Euclidean Killing field.

\section{The Nonlinear Case}
We will show that the nonlinear system (\ref{Fepspb}) is solvable in $\Omega$ under a certain smallness assumption on the parameter $\varepsilon$. Let us first state the nonlinear case of the original problem $(P)$, again.
\begin{equation} \label{nonlinpb}
(P) \left\{ \begin{array}{rll} \overset{m}{\nabla}_b S^{ab} & = 0 & :\quad \textrm{in} \quad \Omega \, , \\ S^{ab}M_b & = 0 & : \quad \textrm{on} \quad \partial \Omega \, , \end{array} \right. 
\end{equation}
where $S^{ab}$ is given by (\ref{stressedge}) and $M_b$ is a covector whose null space is the tangent plane $T_y \Omega$ at $y \in \partial \Omega$. The orientation may be defined by $M_bX^a>0$, whenever $X^b$ is a vector pointing to the exterior of $\Omega$. $M^a=(m^{-1})^{ab}M_b$ is the outer normal to $\partial \Omega$. 
The strategy for solving the problem $(P)$ is the following. We set up an iteration scheme, where the first step is the linearized problem. By the result for problem (\ref{anabvp}), the linearized problem is solvable because the integrability condition is automatically satisfied. The integrability condition of the iteration can then be satisfied by applying a doping technique similar to the one in \cite{Ka}.

\subsection{Iteration} \label{iterationmethod}
We first study the iteration scheme.
For the analogous (generalized) problem $(AP)$ we have from (\ref{iteration}), (\ref{nonlinpb}) (now written in terms of the coordinates $y^a$ on $\Omega$, and $\pi$ denoting the linearized metric $\dot m$):
\begin{equation} \label{AP}
(AP) \left\{ \begin{array}{rll} \frac{\partial \pi_{n+1}^{ab}}{\partial y^b} & = \rho_n^a & :\quad \textrm{in} \quad \Omega \, , \\ \left(\pi_{n+1}^{ab}-\sigma_n^{ab}\right)M_b & = 0 & : \quad \textrm{on} \quad \partial \Omega \, , \end{array} \right.
\end{equation}
where 
\begin{equation*}
\pi^{ab}_{n+1}=\frac{\partial \psi^b_{n+1}}{\partial y^a}+\frac{\partial \psi^a_{n+1}}{\partial y^b} \, .
\end{equation*}
Now we set $\psi_{n+1}'=\psi_{n+1}-\tilde \psi_{n+1}$, where the $\tilde \psi$-part satisfies the boundary conditions, i.e.~
\begin{equation} \label{piprimesigmabdry}
\left(\tilde \pi^{ab}_{n+1}-\sigma^{ab}_n\right)M_b=0 \, ,
\end{equation}
and we have a modified problem $(AP')$ with zero boundary conditions
\begin{equation*}
(AP') \left\{ \begin{array}{rll} \frac{\partial \pi_{n+1}^{' \, ab}}{\partial y^b} & = \rho_n^a-\frac{\partial \tilde \pi_{n+1}^{ab}}{\partial y^b} & :\quad \textrm{in} \quad \Omega \, , \\ \pi_{n+1}^{' \, ab}M_b & = 0 & : \quad \textrm{on} \quad \partial \Omega  \, . \end{array} \right.
\end{equation*}
The integrability condition, which yields existence of solutions of the problem by the Lax-Milgram theorem \cite{Ev}, now reads:
\begin{equation*}
0=\int_{\Omega}\xi^a \left(\rho_n^a-\frac{\partial \tilde\pi_{n+1}^{ab}}{\partial y^b}\right)  =  \int_{\Omega}\xi^a \rho_n^a -\int_{\partial \Omega} \xi^a  \tilde \pi_{n+1}^{ab} M_b + \int_{\Omega}  \tilde\pi_{n+1}^{ab}\frac{\partial \xi^a}{\partial y^b}  \, .
\end{equation*} 
That is:
\begin{equation} \label{intcondnl} 
 \int_{\Omega}\xi^a \rho_n^a - \int_{\partial \Omega} \xi^a  \sigma_n^{ab} M_b = 0 \, ,
\end{equation}
for every Killing field $\xi$ of a background Euclidean metric on $\Omega$. This metric is $id^*\delta$, where $\delta$ is the Euclidean metric of $E$. The $y^a$ are rectangular coordinates on $\Omega$ relative to this metric. In particular, $M_b$ is a unit covector relative to this Euclidean metric and the integral $\int_{\partial \Omega}$ is taken with respect to the measure on $\partial \Omega$ corresponding to the metric (arc length if $\dim \Omega=2$) induced on $\partial \Omega$ by this Euclidean metric on $\Omega$.
In the above, we have made use of (\ref{piprimesigmabdry}) and the properties of $\xi$ as a Killing field. In particular,
\begin{equation*}
\xi^a=\alpha_b^a y^b + \beta^a \quad , \quad \alpha_b^a = -\alpha_a^b \quad  \Rightarrow \quad \frac{\partial \xi^a}{\partial y^b}=\alpha_b^a=\frac 1 2 \left(\alpha_a^b-\alpha_b^a \right) \, ,
\end{equation*} 
and thus the contraction of the symmetric tensors $\tilde \pi_{n+1}^{ab}$, respectively $\sigma_n^{ab}$, with $\frac{\partial \xi^a}{\partial y^b}$ vanishes.
In conclusion, the integrability condition at the $n+1$ step of the iteration is (\ref{intcondnl}).

\subsection{Killing Fields and Doping Technique}
In $n$ dimensional Euclidean space we have the following linearly independent Killing fields:
\begin{itemize}
\item[{\bf i)}] $n$ translations,
\item[{\bf ii)}] $\frac{n(n-1)}{2}$ rotations.
\end{itemize}
So the space of Killing fields on $E^n$ is $N$ dimensional, where
\begin{equation*}
N=n + \frac{n(n-1)}{2}=\frac{n(n+1)}{2} \, .
\end{equation*}
Let $\left(\xi_A: A=1,\ldots ,N\right)$ be a basis for the space of Killing fields. In the spirit of \cite{Ka}, we have to modify $\rho$, that is the inhomogeneity on the right-hand side of (\ref{AP}). This technique is called \emph{doping}. We replace $\rho$ by 
\begin{equation*}
  \rho' =\rho + \sum_A c_A \xi_A \, ,
\end{equation*}
and require in accordance with (\ref{intcondnl}) that
\begin{equation} \label{dopcond}
\int_{\Omega} \xi_A \cdot \rho' = \int_{\partial \Omega} \xi_A \cdot \sigma_M \quad : \quad \forall A=1, \ldots , N \, ,
\end{equation}
where $\sigma_M^a=\sigma^{ab}M_b$. Since 
\begin{equation*}
\int_{\Omega} \xi_A \cdot \xi_B = M_{AB}
\end{equation*}
is positive definite, we obtain a linear system of equations for the coefficients $c_A \, : \, A=1, \ldots , N$ as follows:
\begin{equation*}
\int_{\Omega} \xi_A \cdot \rho' =\int_{\Omega} \xi_A \cdot \rho+  \int_{\Omega} \xi_A \left(\sum_B c_B \xi_B\right) =  \int_{\partial \Omega} \xi_A \cdot \sigma_M \quad : \quad \forall A=1, \ldots , N \, ,
\end{equation*}
and thus
\begin{equation*}
\sum_B  M_{AB} c_B=  \int_{\partial \Omega} \xi_A \cdot \sigma_M - \int_{\Omega} \xi_A \cdot \rho =: \sigma_A  \quad : \quad \forall A=1, \ldots , N \, .
\end{equation*}
The system $\sum_B M_{AB} c_B= \sigma_A$ can be solved, and there is a solution
\begin{equation} \label{coeffsol}
c_A= \sum_B\left( M^{-1}\right)_{AB}\sigma_B \, ,
\end{equation}
$M_{AB}$ being positive definite, hence non-singular.

We reformulate the problem $(AP)$ as follows:
\begin{equation} \label{APtilde}
(\widetilde{AP}) \left\{ \begin{array}{rll} \frac{\partial \pi_{n+1}^{ab}}{\partial y^b} & =  \rho_n^{' \, a} & :\quad \textrm{in} \quad \Omega \, , \\ \left(\pi_{n+1}^{ab}-\sigma_n^{ab}\right)M_b & = 0 & : \quad \textrm{on} \quad \partial \Omega \, . \end{array} \right. 
\end{equation}
Writing $\psi_{n+1}^a=X^a$, $\rho_{n}^{' \, a}=\rho^{' \, a}$, $\sigma_n^{ab} M_b = \tau^a$, (\ref{APtilde}) is of the form
\begin{equation} \label{limitAPtilde}
\left\{ \begin{array}{rll} \frac{\partial}{\partial y^b}\left( \frac{\partial X^b}{\partial y^a}+\frac{\partial X^a}{\partial y^b} \right) & =  \rho^{' \, a} & :\quad \textrm{in} \quad \Omega \, , \\ \left( \frac{\partial X^b}{\partial y^a}+\frac{\partial X^a}{\partial y^b} \right)M_b & = \tau^a & : \quad \textrm{on} \quad \partial \Omega \, . \end{array} \right. 
\end{equation}
Under the restriction discussed above, which if $X$ is a Euclidean Killing field forces $X$ to vanish identically, the linear system (\ref{limitAPtilde}) has no kernel. The following estimate then holds (see \cite{ADN})
\begin{equation} \label{estimate}
|| X ||_{H_{s+2}(\Omega)} \leq C \left\{ || \rho' ||_{H_s(\Omega)}+ || \tau ||_{H_{s+1/2}(\partial \Omega)}\right\} \quad , \quad C=C(\Omega) \, .
\end{equation}
Applying this estimate to (\ref{APtilde}) and taking $\varepsilon$ suitably small we can prove contraction of the sequence $(\psi_n)$ in $H_{s+2}(\Omega)$ for $s> \frac n 2$, $H_s(\Omega)$ and $H_{s-1/2}(\partial \Omega)$ being under this condition Hilbert algebras.
\begin{equation*}
\left. \begin{array}{lrcl}  & \sigma_{n,A}  & \to & \sigma_A \\  & C_{n,A} & \to & C_A \\ \textrm{in} \, \, H_s(\Omega) & \rho_n^a & \to & \rho^a \\ \textrm{in} \, \, H_{s+1/2}(\partial \Omega) & \sigma_n^{ab} & \to & \sigma^{ab} \\ \textrm{in} \, \, H_{s+1}(\Omega) & \pi_{n+1}^{ab} & \to & \pi^{ab} \\  \textrm{in} \, \, H_s(\Omega) & \rho_n^{' \, a} & \to &   \rho^{' \, a} \end{array} \right\} \quad \textrm{for} \quad n \to \infty   \, ,
\end{equation*}
and, in the limit $n \to \infty$ we deduce in terms of the modified inhomogeneity $\rho'$
\begin{equation}\label{dopedpb}
\left\{ \begin{array}{rll} \frac{\partial \pi^{ab}}{\partial y^b} & = \rho^{' \, a} & :\quad \textrm{in} \quad \Omega \, , \\ \left(\pi^{ab}-\sigma^{ab}\right)M_b & = 0 & : \quad \textrm{on} \quad \partial \Omega \, . \end{array} \right. 
\end{equation}
We have:
\begin{equation} \label{rhoprime}
 \rho^{'\,a}=\rho^a + \sum_A c_A \xi_A^a \, ,
\end{equation}
and $c_A$ is given in terms of $\sigma_A$ by (\ref{coeffsol}).
From (\ref{nonlinpb}), (\ref{dopedpb}) and (\ref{rhoprime}), we find
\begin{equation} \label{dagger}
\left\{ \begin{array}{rll} 
\overset{m}{\nabla}_b S^{ab}=\frac{\partial \pi^{ab}}{\partial y^b}-\rho^a=\rho'^a-\rho^a & =:  X^a  & \quad : \, \textrm{in} \quad \Omega \, , \\
S^{ab} M_b & = 0  & \quad : \, \textrm{on} \quad \partial\Omega \, .
\end{array} \right.
\end{equation}
We have to show that
\begin{proposition} 
\begin{equation} \label{defX}
X^a = \sum_A c_A \xi_A^a = 0 \quad : \, a=1, \ldots , n \, ,
\end{equation}
if $\varepsilon$ is suitably small.
\end{proposition}

\begin{proof}
Let $\xi^a$ be a Killing field of $\delta_{ab}$. We have:
\begin{equation}\label{differenceKF}
\int_{\Omega} \delta_{ab} \xi^a X^b d^2y = \int_{\Omega} m_{ab} \xi^a X^b d^2\mu_m +O(\varepsilon^2 \abs{X}_{\infty}) \, ,
\end{equation}
where we say that a real-valued function $Q(\varepsilon,X)$ is $O(\varepsilon^2\left|X\right|_{\infty})$, if there is a constant $C$ such that
\begin{equation*}
\left| Q (\varepsilon,X)\right| \leq C \varepsilon^2 \left|X\right|_{\infty}  \, .
\end{equation*}
(\ref{differenceKF}) holds because $m_{ab}=\delta_{ab}+O(\varepsilon^2)$. Consider the vectorfield $\zeta$ on $\Omega$, the push-forward of which by $\phi$ to the Euclidean plane $E$ coincides with $\xi$
\begin{equation*}
\zeta^a= \xi^i \frac{\partial y^a}{\partial x^i} \quad  , \quad \left[ \, \xi^i = \frac{\partial \phi^i}{\partial y^a} \zeta^a \, \right] \, .
\end{equation*}
Then $\zeta$ is a Killing field of the metric $m=\phi^*\delta$ on $\Omega$. Hence:
\begin{equation} \label{KFzeta}
\int_{\Omega} m_{ab} \zeta^a X^b d\mu_m = \int_{\Omega} \zeta_a \overset{m}{\nabla}_b S^{ab} d\mu_m = - \int_{\Omega} \overset{m}{\nabla}_b \zeta_a S^{ab} d\mu_m + \int_{\partial \Omega} \zeta_a S^{ab} M_b \left. d\mu_m\right|_{\partial \Omega} \, , 
\end{equation}
where $\zeta_a=m_{ab}\zeta^b$. The integral on $\Omega$ in (\ref{KFzeta}) vanishes because
\begin{equation*}
 \overset{m}{\nabla}_b \zeta_a S^{ab}=\frac 1 2 \left( \overset{m}{\nabla}_b \zeta_a +  \overset{m}{\nabla}_a \zeta_b\right) S^{ab} = 0 \, .
\end{equation*}
The integral on $\partial \Omega$ vanishes by virtue of the boundary condition in (\ref{dagger}). Then, since 
\begin{equation*}
\xi^a = \zeta^a + O(\varepsilon^2) \, ,
\end{equation*}
we deduce:
\begin{equation} \label{orderKF}
\int_{\Omega} \delta_{ab} \xi^a X^b d^2y = \int_{\Omega} m_{ab} \zeta^a X^b d\mu_m + O(\varepsilon^2 \abs{X}_{\infty})=O(\varepsilon^2 \abs{X}_{\infty}) \, .
\end{equation}
However,
\begin{equation*}
X^a=\sum_A c_A \xi_A^a \quad : \quad a=1,\ldots,n \, ,
\end{equation*}
and since Killing fields are analytic functions, we have on a bounded domain $\Omega$:
\begin{equation} \label{KFcoeff}
\abs{X}_{\infty}=C \max_A \abs{c_A} \, .
\end{equation}
From (\ref{orderKF}),
\begin{equation}\label{estX}
\left|\int_{\Omega} \xi_A \cdot X \, d^ny \right| \leq C \varepsilon^2 \abs{X}_{\infty} \, .
\end{equation}
But from the definition of X in (\ref{defX}),
\begin{equation*}
\int_{\Omega}\xi_A \cdot X \, d^ny = \sum_{B} \int_{\Omega} \xi_A \xi_B c_B \, d^2y =\sum_{B}M_{AB} c_B = \sigma_A \, ,
\end{equation*}
and hence, from (\ref{estX}),
\begin{equation}\label{estsigmaX}
\abs{\sigma_A} \leq C \varepsilon^2 \abs{X}_{\infty} \, .
\end{equation}
On the other hand, we see from (\ref{coeffsol}) that
\begin{equation} \label{estcsigma}
\max_A \abs{c_A} \leq \tilde C \max_A \abs{\sigma_A} \, ,
\end{equation}
where $\tilde C=||M^{-1}||$. Finally, using (\ref{KFcoeff}), (\ref{estcsigma}) and (\ref{estsigmaX}), we obtain
\begin{equation*}
\abs{X}_{\infty}=C_1 \max_A \abs{c_A} \leq C_2 \max_A \abs{\sigma_A} \leq C_3 \varepsilon^2 \abs{X}_{\infty} \, ,
\end{equation*}
which implies $X \equiv 0$ for $\varepsilon$ sufficiently small ($C_3 \varepsilon^2 < 1 $). This finishes the proof.
\end{proof}

{\remark Concerning the regularity of the solution for $C^{\infty}$ domain $\Omega$ (smooth boundary) the solution is also $C^{\infty}$. For, $\psi \in H_{s+2}$, $s > n/2$ implies $\psi \in H_{s+3}(\Omega)$. Therefore, by induction, $\psi \in H_k(\Omega)$ for every $k$.}

\newpage

\bibliographystyle{my-h-elsevier}

\end{document}